\documentclass{myifcolog}

\usepackage{graphicx} 
\usepackage{tikz}
\usetikzlibrary{arrows.meta}
\usetikzlibrary{calc}
\usepackage{tikz-cd}
\usepackage{xfrac}
\usepackage{graphicx}
\usepackage{amsmath}
\usepackage{amssymb}
\usepackage{relsize}
\usepackage{mathtools}
\usepackage{bussproofs}
\usepackage{proof}
\usepackage{enumitem}
\usepackage{comment}
\usepackage{float}
\usepackage{amsthm}
\usepackage{mathtools}
\usepackage{mathrsfs} 
\usepackage{esvect,pgf, tikz, color}
\usetikzlibrary{arrows, automata, positioning}

{\Cdef\scalefactor{#1}\begin{center}\proofSkipAmount \leavevmode}%
	{\scalebox{\scalefactor}{\DisplayProof}\proofSkipAmount \end{center} }

\newcommand\C {C}
\newcommand{\K}{K}
\newcommand\IM{\mathsf{IM}}
\newcommand\Cl {\mathsf{Cl}}
\newcommand\Prop {\mathsf{Prop}}

\newcommand\X {\mathbin{\mathsf{X}}}
\newcommand\G {\mathbin{\mathsf{G}}}

\newcommand\Ck{\mathsf{C_0}}
\newcommand\Crt{\mathsf{C_1}}
\newcommand\R{\mathsf{R}}

\newcommand\A{\mathsf{A}}
\newcommand{\agi}{a}
\newcommand{\agii}{b}
\newcommand{\agiii}{c}
\newcommand{\E}{E}
\newcommand\CICK{\mathsf{cICK}}
\newcommand{\ICK}{\mathsf{ICK}}
\newcommand{\LICK}{\mathcal{L}_{\ICK}}
\newcommand{\fw}{\mathcal{M}}
\newcommand{\fv}{\mathcal{N}}
\newcommand{\Rel}{\mathrel{R}}
\newcommand{\f}{\mathsf{f}}
\renewcommand{\u}{\mathsf{u}}
\newcommand{\W}{\mathcal{W}}
\newcommand{\rel}{\mathcal{R}}
\newcommand{\Fr}{\mathcal{F}}
\newcommand{\Gr}{\mathcal{G}}
\newcommand{\Frc}{\mathscr{F}}
\newcommand{\Grc}{\mathscr{G}}
\newcommand{\V}{\mathcal{V}}

\newcommand{\fclass}{\mathscr{F}}
\newcommand{\hick}{\mathrm{ICK}}
\newcommand{\hickt}{\mathrm{ICKT}}
\newcommand{\hicks}{\mathrm{ICKS4}}
\newcommand{\hickss}{\mathrm{ICKS5}}
\newcommand{\hickp}{\mathrm{P}}
\newcommand{\cick}{\mathrm{cICK}}
\newcommand{\cickt}{\mathrm{cICKT}}
\newcommand{\cicks}{\mathrm{cICKS4}}
\newcommand{\cickss}{\mathrm{cICKS5}}
\newcommand{\nick}{\mathrm{nICK}}
\newcommand{\nickt}{\mathrm{nICKT}}
\newcommand{\nicks}{\mathrm{nICKS4}}

\newcommand{\an}{\mathsf{a}}

\newcommand{\dimp}{\mathbin{\tikz[baseline=-.5ex] \draw[-to reversed] (0,0) -- (1em,0);}}

\newtheorem{theorem}{Theorem}
\newtheorem{definition}[theorem]{Definition}
\newtheorem{lemma}[theorem]{Lemma}
\newtheorem{proposition}[theorem]{Proposition}
\newtheorem{corollary}[theorem]{Corollary}
\newtheorem{example}[theorem]{Example}

\numberwithin{theorem}{section}

\title{Intuitionistic Common Knowledge}
\addauthor[lukas.zenger@pku.edu.cn]{Lukas Zenger\thanks{This work has been support by the Swiss National Science Foundation, Grant P500-2\_235387}}{Department of Philosophy, Peking University}
\titlerunning{Intuitionistic Common Knowledge}
\authorrunning{Zenger}
\titlethanks{}

\date{}

\renewcommand{\firstfoottext}{}

\begin{document}

\maketitle

\begin{abstract}
\noindent We study an intuitionistic version of common knowledge logic ($\mathsf{CK}$), called $\ICK$, which was introduced by Jäger and Marti \cite{jager-intuitionistic_2016}. $\ICK$ extends intuitionistic propositional logic ($\mathsf{IPL}$) by multiple box modalities interpreted as knowledge operators for various agents and a common knowledge operator. Formulae are interpreted over birelational Kripke models satisfying a simple interaction principle between the intuitionistic order and the modal accessibility relations. Furthermore, we consider the restriction to reflexive, S4 and S5 models. We present axiomatizations as well as analytic cyclic sequent calculi for all considered logics and prove them to be sound and complete. Furthermore, we establish the finite model property and decidability, show that proof-search in the cyclic calculi can be automated, provide a translation of $\mathsf{CK}$ over S5 into $\ICK$ over S5 and establish that the proof-search and validity problems of all considered logics can be solved in exponential time.
\end{abstract}

\section{Introduction}\label{s: introduction}

Common knowledge is a central notion of group knowledge in epistemic logic, which captures not only that all agents know a fact, but also that every agent knows that every agent knows it, and so on ad infinitum. Epistemic logics in general and logics of common knowledge ($\mathsf{CK}$) in particular are typically devised as classical multi-modal systems in which each agent is equipped with a knowledge modality and common knowledge is expressed via a greatest fixed point construction. Since the seminal work of Halpern and Moses~\cite{Halpern_1990}, logics of common knowledge $(\mathsf{CK})$ have played an important role in areas such as distributed systems, multi-agent reasoning, and game theory. Accordingly, the mathematical theory of $\mathsf{CK}$ has been thoroughly investigated, for example its proof theory~\cite{alberucci_2005, studer_2009, Marti_2018, rooduijn_analytic_2022}. For a general introduction to epistemic logic and common knowledge, see~\cite{ditmarsch_2017}. 

In recent years, there has been growing interest in developing intuitionistic variants of epistemic logic~\cite{Williamson_1992, Artemov_2016, jager-intuitionistic_2016, proietti_2012, Hirai_2010}. Such systems aim to combine the constructive semantics of intuitionistic logic with modal operators for knowledge. Broadly speaking, two lines of research have emerged. On the one hand, intuitionistic epistemic logics, such as $\mathsf{IEL}$ introduced by Artemov and Protopopescu~\cite{Artemov_2016}, which aim to provide an intuitionistic reading of knowledge. To that end knowledge is given a Brouwer–Heyting–Kolmogorov (BHK) interpretation in terms of verification. On the other hand, more `constructive' versions of intuitionistic epistemic logics have been proposed, with potential applications to computer science~\cite{Hirai_2010}. Such logics generally treat knowledge in a way similar to classical epistemic logic and instead focus on their mathematical and computational properties, as well as extensions with group knowledge concepts such as common knowledge. As such these kind of logics belong to the family of intuitionistic modal fixed point logics, which has been extensively studied in recent years (see e.g.~\cite{fernandez-duque_2018, Boudou_2021, afshari_ill-founded_2023, afshari_intuitionistic_2024, afshari2025intuitionistic, Pacheco_2024, fernandez-duque_sound_2024, Fernandez-Duque_complexity_2025, Boudou_2017}).

A prominent example of the latter approach is intuitionistic common knowledge logic ($\ICK$), introduced by Jäger and Marti~\cite{jager-intuitionistic_2016}. The logic $\ICK$ extends intuitionistic propositional logic ($\mathsf{IPL}$) by epistemic modalities for a finite set of agents together with a common knowledge operator defined as a greatest fixed point. Semantically, formulae are interpreted over birelational Kripke models equipped with an intuitionistic order and modal accessibility relations. To ensure persistence of truth (i.e. monotonicity along the intuitionistic order) in our semantics, a compatibility condition is placed on the modal relations. In contrast to verification-based systems, knowledge in $\mathsf{ICK}$ follows the standard axioms of epistemic logic, and common knowledge is treated analogously to the classical case. Intuitively, $\mathsf{ICK}$ can be understood as a logic of knowledge over evolving information states: worlds represent stages of information growth, and knowledge captures what agents can justify given their current information. The common knowledge operator then expresses facts that remain stable across all epistemically possible developments of the system.

While the basic properties of $\mathsf{ICK}$ have been established — including sound and complete axiomatizations of $\ICK$ over the classes of all frames and of reflexive frames~\cite{marti_2017}, sequent calculi~\cite{jager-intuitionistic_2016} and variations accounting for extensions with distributed knowledge or public announcements~\cite{jager-distributed_2016, murai_2022, murai_2024} — its proof theory and computational aspects are not yet fully understood. In particular, existing systems rely on proof calculi with cut or induction principles, and do not provide fully analytic proof systems or direct methods for automated reasoning. Moreover, extensions of $\mathsf{ICK}$ to richer classes of epistemic frames, such as S4 and S5, have not been systematically developed, and the complexity of checking validity in such systems remains largely unexplored (with the notable exception of $\ICK$ over the class of all frames, which has recently been shown to be \textsc{ExpTime}-complete~\cite{Fernandez-Duque_complexity_2025}).

This paper addresses these gaps by providing four main contributions. First, we extend $\ICK$ to additional frame classes which satisfy the standard epistemic frame conditions S4 and S5. We analyze the properties of the resulting logics and provide sound and complete Hilbert-style axiomatizations for each case. We also present our interpretation of $\ICK$ as a logic for reasoning about knowledge over evolving information states, and by doing so we connect our work to the original interpretation of Kripke semantics for $\mathsf{IPL}$ by Kripke~\cite{Kripke_1965}. Second, we introduce Gentzen-style sequent calculi for all considered logics. In difference to~\cite{jager-intuitionistic_2016}, which develops (non-analytic) sequent calculi with induction rules, we work with cyclic and non-wellfounded calculi. These calculi eliminate the need for explicit induction rules and avoid unrestricted use of cut, thus remaining analytic in the sense that all formulae occurring in proofs are drawn from a suitable closure of the root formula. We then show soundness and analytic completeness for the non-wellfounded calculi by constructing a uniform approach for performing proof-search. Completeness of the cyclic calculi then follows by showing how non-wellfounded proofs can be pruned into cyclic proofs. This part builds on previous work from a joint publication with Afshari, Grotenhuis and Leigh~\cite{afshari_intuitionistic_2024}.  Third, we show that proof-search in our cyclic calculi can be reduced to solving parity games. This yields effective procedures for constructing proofs and establishes that proof-search can be automated. Finally, fourth, we establish exponential upper bounds for the complexity of proof-search and checking validity, which we obtain as a consequence of our proof-theoretic analysis. In the S5 case, we further obtain \textsc{ExpTime}-completeness by providing a translation of classical common knowledge logic over S5 into the intuitionistic setting, which serves as a polynomial-time reduction and simultaneously clarifies the relationship between classical and intuitionistic approaches to common knowledge. \smallskip

\noindent \textbf{Layout.} The remainder of the paper is structured as follows. Section~\ref{s: intuitionistic common knowledge logic} introduces the syntax and semantics of $\mathsf{ICK}$ and its variants and discusses our interpretation of $\ICK$. Section~\ref{s: axiomatization} presents Hilbert-style axiomatizations and proves soundness and completeness. Section~\ref{s: translation} establishes a translation from classical common knowledge logic into the intuitionistic setting. Section~\ref{s: cyclic proof systems} develops cyclic sequent calculi and proves their soundness and completeness. Section~\ref{s: automated proof-search} studies automated proof-search via parity games and derives complexity bounds. Section~\ref{s: conclusion} concludes with directions for future work.

\section{Intuitionistic Common Knowledge Logic}\label{s: intuitionistic common knowledge logic}

This section provides an overview of $\ICK$. We first introduce the syntax and the semantics in terms of birelational Kripke models. We then establish several basic properties of our semantics, outline a possible interpretation of $\ICK$ as a logic to reason about knowledge over evolving information states and remark upon the expressiveness of the language.

\subsection{Syntax}\label{ss: syntax}
 The \emph{language} $\LICK$ of $\ICK$ consists of a finite and non-empty set $\A$ of \emph{agents names}, the constant $\bot$ (\emph{absurdity}), a denumerable set $\Prop$ of \emph{atomic propositions}, the connectives $\wedge$ (\emph{conjunction}), $\vee$ (\emph{disjunction}), $\rightarrow$ (\emph{implication}), the \emph{knowledge modalities} $\K_\agi$ for $\agi \in \A$ and the \emph{common knowledge operator} $\C$, as well as brackets $($ and $)$. \emph{Formulae} of $\mathcal{L}_\mathrm{ICK}$ are given by the following grammar in Backus-Naur form:
\begin{equation*}
    \varphi ::= \bot \, \lvert \, p \, \lvert \, (\varphi \wedge \varphi) \, \lvert \, (\varphi \vee \varphi) \, \lvert \, (\varphi \rightarrow \varphi) \, \lvert \, (\K_\agi \varphi) \, \lvert \, (\C \varphi)
\end{equation*}
where $p \in \Prop$ and $\agi \in \A$. We use Latin letters $p,q,r$ and so on to denote atoms and Greek letters $\varphi, \psi, \gamma$ and so on to denote formulae (both atoms and formulae may be annotated). Agents are denoted by $\agi, \agii, \agiii$ and so on. Outermost brackets are typically omitted, i.e. we write for example $\varphi \rightarrow \psi$ instead of $(\varphi \rightarrow \psi)$, as long as the is no danger of confusion. We write $\varphi \in \LICK$ to denote that $\varphi$ is an $\LICK$-formula. The constant $\top$ (\emph{truth}) is defined by $\top \coloneqq \bot \rightarrow \bot$, and the connectives $\neg$ (\emph{negation}) by $\neg \varphi \coloneqq \varphi \rightarrow \bot,$ and $\leftrightarrow$ (\emph{if and only if}) by $\varphi \leftrightarrow \psi \coloneqq (\varphi \rightarrow \psi) \wedge (\psi \rightarrow \varphi)$. Furthermore, the group knowledge operator $\E$ (\emph{everybody knows}) is defined by
\begin{equation*}
    \E \varphi  \coloneqq \bigwedge_{\agi \in \A} \K_\agi \varphi.
\end{equation*}
The knowledge modalities are modal boxes. Note that the language $\LICK$ does not feature the dual diamond modalities. Importantly, box and diamond are not interdefinable in the intuitionistic setting. For readers familiar with modal fixed point logics, we remark that $\C \varphi$ can be defined as in the classical case as the greatest fixed point of the propositional function $\mathsf{x} \mapsto \varphi \wedge \E \mathsf{x}$, see~\cite{marti_2017} for a proof and~\cite{Zenger_2025} for a brief discussion of fixed point operators in intuitionistic modal fixed point logics.

\begin{definition}\label{d: closure}
         A set of formulae $\Gamma$ is called \emph{closed} if $\Gamma$ satisfies conditions 1 to 3 below. Moreover, $\Gamma$ is called \emph{negation closed} if $\Gamma$ is closed and additionally satisfies condition 4.
         \begin{enumerate}
        \item If $\psi_1 \ast \psi_2 \in \Gamma$, then $\psi_i \in \Gamma$ where $i \in \{0,1\}$ and $\ast \in \{\wedge, \vee, \rightarrow\}$.
        \item If $K_\agi \psi \in \Gamma$, then $\psi \in \Gamma$ for $\agi \in \A$.
        \item If $\C \psi \in \Gamma$, then $\psi \in \Gamma$ and $\K_\agi \C \psi \in \Gamma$ for all $\agi \in \A$.
        \item If $K_\agi \psi \in \Gamma$, then $\neg K_\agi \psi \in \Gamma$.
    \end{enumerate}
     The \emph{closure} of a formula $\varphi$ is the smallest closed set $\Cl(\varphi)$ of formulae which contains $\varphi$.  The \emph{negation closure} of a formula $\varphi$ is the smallest negation closed set $\Cl^\neg(\varphi)$ of formulae which contains $\varphi$. Finally, we define $\Cl(\Gamma):= \bigcup_{\varphi\in\Gamma}\Cl(\varphi)$ and $\Cl^\neg(\Gamma) \coloneqq \bigcup_{\varphi \in \Gamma} \Cl^\neg (\varphi)$.
     \end{definition}



 The following lemma is easily verified by an induction on the structure of formulae.

\begin{lemma}\label{l: negation closure is finite}
    For any formula $\varphi$, $\Cl (\varphi)$ and $\Cl^\neg(\varphi)$ are finite. Moreover, if $\Gamma$ is a finite set of formulae, then $\Cl(\Gamma)$ and $\Cl^\neg(\Gamma)$ are finite.
\end{lemma}

The \emph{complexity} of a formula is defined as follows.

\begin{definition}\label{d: complexity}
    The \emph{complexity} $c(\varphi)$ of a formula $\varphi$ is defined inductively as follows.
    \begin{itemize}
        \item $c(\bot) = c(p) = 0$
        \item $c(\varphi \ast \psi) = c(\varphi) + c(\psi) +1$ for $\ast \in \{\wedge, \vee, \rightarrow\}$
        \item $c(\K_\agi \varphi) = c(\C \varphi) = c(\varphi) + 1$
    \end{itemize}
    Given a finite set of formulae $\Gamma$, the \emph{complexity} of $\Gamma$ is defined as $c(\Gamma) = \sum_{\varphi \in \Gamma}c(\varphi)$
\end{definition}

\subsection{Semantics}\label{ss: semantics}

Formulae are evaluated over Kripke models equipped with an intuitionistic partial order capturing information growth, together with an epistemic accessibility relation for each agent. Such structures are called \emph{birelational models}. Given a set $X$, we denote by $\mathcal{P}(X)$ the \emph{power set} of $X$.

\begin{definition}
    A \emph{birelational frame} is a tuple $\Fr= (\W, \leq, \rel)$ such that the following hold.
    \begin{enumerate}
        \item $\W$ is a non-empty set.
        \item $(\W, \leq)$ is a poset.\footnote{Note that usually such semantics are defined where $\leq$ is a preorder and not a partial order. While this difference matters for the constructive modal logic $\mathsf{CK}$, it does not matter here due to the lack of a diamond operator in our language. We chose to use a partial order since this aligns better with our intended motivation of modeling knowledge under information growth.}
        \item $\rel = \{ R_\agi \mid \agi \in \A\}$ such that for each $\agi \in \A$, $R_\agi \subseteq W \times W$.
    \end{enumerate}
    A \emph{birelational model} is a tuple $\fw=(\Fr, \V)$ where $\Fr =(\W, \leq \rel)$ is a birelational frame and $\V:\W \longrightarrow \mathcal{P}(\Prop)$ is a function which is monotone in $\leq$: for all $w,v \in \W$ if $w \leq v$, then $\V(w) \subseteq \V(v)$. The model $\fw =(\Fr, \V)$ is said to be \emph{based} on $\Fr$.
\end{definition}

Birelational frames are denoted by $\Fr$ and $\Gr$ and we denote classes of birelational frames by $\Frc$ and $\Grc$. Birelational models are denoted by $\fw$ and $\fv$. Elements of $\W$ are called \emph{worlds} or \emph{states} and are denoted by $w,v,u$ and so on. The relation $\leq$ is called the \emph{intuitionistic order} and $v$ is called an \emph{intuitionistic successor} of $w$ if $w \leq v$. For each agent $\agi$ the relation $R_\agi$ is called the \emph{modal accessibility relation} of $\agi$ and $v$ is called a \emph{modal successor} of $w$ if $w \Rel_\agi v$. The function $\V$ is called a \emph{valuation} and assigns to each world the set of propositions `true' at that world. For our purposes we are interested in a specific class of birelational frames and models, whose elements satisfy a simple interaction principle between $\leq$ and each $R_\agi$. We follow~\cite{jager-intuitionistic_2016} and call such frames and models \emph{epistemic}.

\begin{definition}\label{d: model}
    An \emph{epistemic frame} is a birelational frame $\Fr= (\W, \leq, \rel)$ such that for each $\agi \in \A$ and for all $w,v,u \in W$:
        \begin{equation}\label{e: triangle confluence}
            \text{if } w \leq v \text{ and } v \Rel_\agi u, \text{ then } w \Rel_\agi u.
        \end{equation}
    An \emph{epistemic model} is a birelational model based on an epistemic frame.
\end{definition}

We will usually call epistemic frames and epistemic models simply frames and models. The property of $R_\agi$ expressed in (\ref{e: triangle confluence}) is called \emph{triangle confluence}, see Figure \ref{f: triangle confluence} for a depiction. There are two main reasons to consider such a property. On the one hand when evaluating modalities via the `classical' truth conditions (e.g. $w \models K_\agi \varphi$ if $v \models \varphi$ for all $w \Rel_\agi v$), triangle confluence ensures that knowledge is preserved under information growth and so allows us to keep the \emph{Monotonicity Property} of intuitionistic logic (c.f. Lemma~\ref{l: monotonicity}). Without triangle confluence, more complicated `intuitionistic' truth conditions would be required to retain monotonicity. On the other hand, triangle confluence is a natural condition to consider in our proposed interpretation of $\ICK$. Both reasons are discussed in more detail below in Subsection~\ref{ss: expressivity} and Subsection~\ref{ss: interpretation}, respectively. Given an epistemic frame $\Fr=(\W, \leq, \rel)$, define the relation
\begin{equation*}
    R := \bigcup_{R_\agi \in \rel} R_\agi
\end{equation*}
and let $R^*$ be the reflexive and transitive closure of $R$. Given $w \in \W$ define $R_\agi(w) \coloneqq \{v \in \W \mid w \Rel_\agi v\}$ for each $\agi \in \A$ and $R^*(w) \coloneqq {\{v \in \W \mid w \Rel^* v\}}$. Furthermore let $w^\uparrow \coloneqq \{ v \in \W \mid w \leq v\}$.

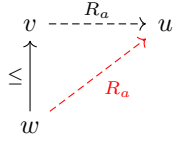
\begin{figure}[t]
    \centering
    \begin{tikzcd}[ampersand replacement=\&,row sep=large,column sep=large]
        v \ar[r, dashed, "R_\agi"] \& u\\
        w \ar[u,"\le"]\ar[ur,dashed,"R_\agi"', color=red] 
    \end{tikzcd}
    \caption{The diagram above depicts the triangle confluence condition. The partial order $\leq$ is depicted by solid arrows, the relation $R_\agi$ by dashed arrows and the part stipulated by the confluence condition is marked in red.}
    \label{f: triangle confluence}
\end{figure}

\begin{definition}
    The \emph{truth relation} $\models$ between worlds $w$ of an epistemic model $\fw$ and formulae $\varphi, \psi$ is defined inductively as follows.
\begin{center}
    \begin{tabular}{l l l}
    $\fw, w \not \models \bot$ & & \\
    $\fw,w \models p$ & iff & $p \in \V(w)$ \\
    $\fw,w \models \varphi \wedge \psi$ & iff & $\fw,w \models \varphi$ and $\fw,w \models \psi$ \\
    $\fw,w \models \varphi \vee \psi$ & iff & $\fw,w \models \varphi$ or $\fw,w \models \psi$ \\
    $\fw,w \models \varphi \rightarrow \psi$ & iff & for all $v \in w^\uparrow$ if $\fw,v \models \varphi$, then $\fw,v \models \psi$ \\
    $\fw,w \models \K_\agi \varphi$ & iff & for all $v \in R_\agi(w)$, $\fw,v \models \varphi$ \\
    $\fw, w \models \C \varphi$ & iff & for all $v \in R^*(w)$, $\fw,v \models \varphi$.
    \end{tabular}
\end{center}
A formula $\varphi$ is called \emph{true at world $w$} if $\fw, w \models \varphi$ . Furthermore, $\varphi$ is called \emph{true in $\fw$}, written $\fw \models \varphi$, if $\fw, w \models \varphi$ for all $w \in \W$. Let $\fclass$ be a class of epistemic frames. Then $\varphi$ is called \emph{satisfiable} over $\fclass$ if there exists an epistemic model $\fw$ based on a frame in $\fclass$ and a world $w$ such that $\fw, w \models \varphi$; otherwise $\varphi$ is called \emph{unsatisfiable} over $\fclass$. A formula $\varphi$ is called \emph{valid} over $\fclass$ if $\fw \models \varphi$ for all epistemic models $\fw$ based on a $\fclass$-frame; otherwise $\varphi$ is called \emph{invalid} over $\fclass$.
\end{definition}

A crucial feature of our semantics is the \emph{Monotonicity Property} which states that truth is monotone in $\leq$. This property is well-established for Kripke semantics of intuitionistic logic. That the extended language with knowledge and common knowledge operators retains this property is due to triangle confluence, as illustrated in the proof of the following lemma.

\begin{lemma}[Monotonicity]\label{l: monotonicity}
    Let $\fw=(\W, \leq, \rel, \V)$ be an epistemic model, $w, v \in \W$ and $\varphi \in \LICK$. If $\fw, w \models \varphi$ and $w \leq v$, then $\fw, v \models \varphi$.
\end{lemma}
\begin{proof}
    By induction on the structure of $\varphi$. The case for $\varphi= \bot$ trivially holds. For $\varphi = p \in \Prop$ if $\fw, w \models p$, then $p \in \V(w)$. Since $w \leq v$, the monotonicity of $\V$ implies that $p \in \V(v)$. Hence $\fw, v \models p$. The cases for $\varphi = \psi \wedge \chi$ and $\varphi = \psi \vee \chi$ follow directly from the truth conditions for $\wedge$ and $\vee$ and from the induction hypothesis. \smallskip
    
   \noindent Suppose $\varphi = \psi \rightarrow \chi$ and $\fw,w \models \psi \rightarrow \chi$. By definition, for all $u \in w^\uparrow$ if $\fw,u \models \psi$, then $\fw,u \models \chi$. If $u \in v^\uparrow$, then the transitivity of $\leq$ implies that $u \in w^\uparrow$ and so if $\fw, u \models \psi$, then $\fw, u \models \chi$. Hence $\fw, v \models \psi \rightarrow \chi$. \smallskip

   \noindent Let $\agi \in \A$ and suppose $\varphi = \K_\agi \psi$ and $\fw, w \models \K_\agi \psi$. By definition, for all $u \in R_\agi(w)$ holds $\fw, u \models \psi$. If $u \in R_\agi(v)$, then triangle confluence implies $u \in R_\agi(w)$ and therefore $\fw, u \models \psi$. Hence $\fw, v \models \K_\agi \psi$. \smallskip

   \noindent Suppose $\varphi = \C \psi$ and $\fw, w \models \C \psi$. By definition, for all $u \in R^*(w)$ holds $\fw, u \models \psi$. In particular, $\fw, w \models \psi$ since $w \Rel^* w$. Suppose $v \Rel^* u$. Then there exists $n < \omega$ and worlds $u_0, \ldots u_n$ with $u_0 = v$, $u_n = u$ and for all $k \in [n]$\footnote{We write $[n]$ for the set $\{0, 1, \ldots, n-1\}$ and $k \in [n]$ for $0 \leq k \leq n-1$.} there exists $\agi_k \in \A$ such that $u_k \Rel_{\agi_k} u_{k+1}$. We proceed by induction on $n$. If $n = 0$, then $v = u$. Since $\fw, w \models \psi$ and $w \leq v$, the induction hypothesis yields $\fw, u \models \psi$. If $n > 0$, then triangle confluence implies that $w \Rel_{\agi_0} u_1$ and so that $w \Rel^* u$. Hence $\fw, u \models \psi$ and therefore $\fw, v \models \C \psi$.
\end{proof}

We call an epistemic frame $\Fr=(\W, \leq, \rel)$ \emph{reflexive} if for all $\agi \in \A$ the modal accessibility relation $R_\agi \in \rel$ is reflexive. Similarly, $\Fr$ is called \emph{transitive} and \emph{symmetric} if the modal accessibility relations for each agent are transitive and symmetric, respectively. A frame $\Fr$ is an \emph{S4 frame} if each modal accessibility relation is reflexive and transitive. Moreover an S4 frame is an \emph{S5 frame} if, additionally, each modal accessibility relation is symmetric (i.e. the modal relations of S5 frames are equivalence relations). Furthermore, we call an epistemic model a \emph{reflexive/S4/S5 model} if it based on a reflexive/S4/S5 frame, respectively. A frame $\Fr=(\W, \leq, \rel)$ is \emph{finite} if $\W$ is finite and a model is finite if it is based on a finite frame. For the purposes of this paper a \emph{logic} is a set of formulae valid over a given class of frames.

\begin{definition}
    Define the following logics.
    \begin{itemize}
        \item  $\mathbf{ICK}$ is the set of valid $\LICK$-formulae over the class of all (epistemic) frames.
        \item $\mathbf{ICKT}$ is the set of valid $\LICK$-formulae over the class of reflexive frames.
        \item $\mathbf{ICKS4}$ is the set of valid $\LICK$-formulae over the class of S4 frames.
        \item $\mathbf{ICKS5}$ is the set of valid $\LICK$-formulae over the class of S5 frames.
    \end{itemize}
\end{definition}

Note that for reflexive frames triangle confluence implies that every intuitionistic successor of $w$ is also a modal successor. For S4 frames, due to transitivity, every node in the intuitionistic tree rooted at $w$ is a modal successor. The case for S5 is discussed below. These properties fit naturally into our proposed interpretation for $\ICK$. Before discussing the interpretation, we briefly establish some basic properties of our logics. First of all, we show that knowledge in $\ICK$ validates the standard axioms from classical epistemic logic.

\begin{lemma}\label{l: valid formulae}
    Let $\varphi, \psi$ be $\LICK$-formulae and $\agi \in \A$. Then
    \begin{enumerate}
        \item $\K_\agi (\varphi \rightarrow \psi) \rightarrow (\K_\agi \varphi \rightarrow \K_\agi \psi) \in \mathbf{ICK}$;
        \item $\C (\varphi \rightarrow \psi) \rightarrow (\C \varphi \rightarrow \C \psi) \in \mathbf{ICK}$;
        \item $\C \varphi \leftrightarrow \varphi \wedge \E \C \varphi \in \mathbf{ICK}$;
        \item $\K_\agi \varphi \rightarrow \varphi \in \mathbf{ICKT}$;
        \item $\K_\agi \varphi \rightarrow \K_\agi \K_\agi \varphi \in \mathbf{ICKS4}$.
    \end{enumerate}
\end{lemma}
\begin{proof}
    The proofs are standard albeit a bit tedious. For example for 1. let $\Fr=(\W, \leq, \rel)$ be an epistemic frame, $\fw =(\Fr, \V)$ an epistemic model based on $\Fr$ and $w \in \W$ a world. Let $v \in w^\uparrow$ and suppose $\fw, v \models \K_\agi(\varphi \rightarrow \psi)$. Then $\fw, u \models \varphi \rightarrow \psi$ for all $u \in R_\agi(v)$. Let $v' \in v^\uparrow$ and suppose that $\fw, v' \models \K_\agi \varphi$. Then for all $u' \in R_\agi(v')$ holds $\fw, u' \models \varphi$. By triangle confluence $u' \in R_\agi(v)$ (since $v \leq v'$ and $v' \Rel_\agi u'$) and therefore $\fw,u' \models \varphi \rightarrow \psi$. Hence $\fw, u' \models \psi$. Since $u' \in R_\agi(v')$ was arbitrary we conclude $\fw, u' \models \K_\agi \psi$ and so, since $v' \in v^\uparrow$ was arbitrary, we have $\fw, v \models \K_\agi \varphi \rightarrow \K_\agi \psi$. Since $v \in w^\uparrow$ was arbitrary as well, $\fw, w \models \K_\agi (\varphi \rightarrow \psi) \rightarrow (\K_\agi \varphi \rightarrow \K_\agi \psi)$. Hence, $\K_\agi (\varphi \rightarrow \psi) \rightarrow (\K_\agi \varphi \rightarrow \K_\agi \psi)$ is valid over the class of all epistemic frames. The case for 2. is similar. \smallskip
    
    \noindent For 3. let $\fw=(\W, \leq, \rel, \V)$ be an arbitrary epistemic model, $w \in \W$ an arbitrary world and $\varphi \in \LICK$ an arbitrary formula. For the direction from left-to-right let $v \in w^\uparrow$ and suppose that $\fw, v \models \C \varphi$. Then for any $u \in R^*(v)$ holds $\fw, u \models \varphi$. In particular, $v \in R^*(v)$ and so $\fw, v \models \varphi$. Furthermore, let $u \in W$ and suppose that there exists $\agi \in \A$ with $v \Rel_\agi u$. Let $u' \in W$ and suppose $u \Rel^* u'$. Clearly, this implies that $v \Rel^* u'$ and so that $\fw, u' \models \varphi$. Hence $\fw, u \models \C \varphi$. Since $u$ was arbitrary, we have $\fw, v \models \K_\agi \C \varphi$. Since $\agi \in \A$ was arbitrary, $\fw, v \models \bigwedge_{\agi \in \A} \K_\agi \C \varphi$, i.e. $\fw, v \models \E \C \varphi$. Hence $\fw, w \models \C \varphi \rightarrow \varphi \wedge \E \C \varphi$. For the direction from right-to-left let $v \in w^\uparrow$ and suppose that $\fw, v \models \varphi \wedge \E \C \varphi$. Moreover let $u \in R^*(v)$. Then there exists $n < \omega$ and worlds $u_0, \ldots, u_n \in W$ with $u_0 = v$, $u_n = u$ and for all $k \in [n]$ there exists $\agi_k \in \A$ with $u_k \Rel_{\agi_k} u_{k+1}$. We show $\fw, u \models \varphi$ by induction on $n$. For $n=0$ we have $v= u$ and $\fw, u \models \varphi$ by assumption. For $n>0$, since $\fw, v \models \E \C \varphi$ we have $\fw, u_1 \models \C \varphi$ and since $u_1 \Rel^* u$ it holds that $\fw, u \models \varphi$. Hence $\fw, v \models \C \varphi$ and $\fw, w \models \varphi \wedge \E \C \varphi \rightarrow \C \varphi$. The cases for 4. and 5. follow from the respective frame conditions.
\end{proof}

While knowledge validates the classical axioms, $\ICK$ as a whole inherits many of the properties of intuitionistic logic. For example, the \emph{Semantic Disjunction Property}, which states that whenever a disjunction is valid, then one of the disjuncts is valid, holds for $\ICK$ over all considered classes (with the exception of the class of S5 frames; see below). The proof of the following lemma was first presented in~\cite{marti_2017}, without considering the S4 case.

\begin{lemma}[Semantic Disjunction Property]\label{l: semantic disjunction property}
   Let $\varphi, \psi$ be $\LICK$-formulae. If $\varphi \vee \psi$ is valid over one of the classes of epistemic, reflexive or S4 frames, then $\varphi$ is valid or $\psi$ is valid over the respective class of frames.
\end{lemma}
\begin{proof}
    Suppose $\varphi \vee \psi$ is valid over one of the classes of epistemic frames, reflexive frames or S4 frames and suppose towards contradiction that neither $\varphi$ nor $\psi$ are valid over this class. Let $\fw =(\W^\fw, \leq^\fw, \rel^\fw, \V^\fw)$ and $\fv =(\W^\fv, \leq^\fv, \rel^\fv, \V^\fv)$ be epistemic models based on frames from the respective class and let $w \in \W^\fw$ and $v \in \W^\fv$ such that $\fw, w \not \models \varphi$ and $\fv, v \not \models \psi$. We assume without loss of generality that $\W^\fw \cap \W^\fv = \emptyset$ (otherwise rename worlds). Let $u$ be a fresh world not occurring in $\W^\fw$ or in $\W^\fv$. Define the model $\fw_G =(\W, \leq, \rel, \V)$ as follows.
    \begin{itemize} 
        \item $\W \coloneqq \{u\} \cup \W^\fw \cup \W^\fv$.
        \item Let $\leq_0 \coloneqq \{(u,w), (u,v)\} \cup \leq^\fw \cup \leq^\fv$ and define $\leq \coloneqq (\leq_0)^\ast$ (i.e. $\leq$ is the reflexive-transitive closure of $\leq_0$).
        \item For each $\agi \in \A$ let $R_\agi^0 \coloneqq R_\agi^\fw \cup R_\agi^\fv$ and define $R_\agi$ to be $R_\agi^0$ closed under triangle confluence with respect to $\leq$. Furthermore, if $\fw$ and $\fv$ are based on reflexive or S4 frames, then additionally $R_\agi$ shall be closed under reflexivity or reflexivity and transitivity, respectively.
        \item $\V: \W \longrightarrow \mathcal{P}(\Prop)$ is given by
        \begin{equation*}
            \V(s)  \coloneqq
            \begin{cases}
                 \V^\fw(s) &\text{ if } s \in \W^\fw; \\
                \V^\fv(s) &\text{ if } s \in \W^\fv; \\
                \emptyset & \text{ if } s= u. \\
            \end{cases}
        \end{equation*}
    \end{itemize}
    Note that $(\W, \leq, \rel)$ is an epistemic/reflexive/S4 frame, as $\leq$ is a partial order and each $R_\agi$ is triangle confluent and satisfies the relevant frame conditions by construction. Moreover, the valuation $\V$ is monotone in $\leq$ since $\V(u) \subseteq \V(w)$, $\V(u) \subseteq \V(v)$ and $\V(s) = \V^\fw(s)$ for $s \in \W^\fw$ and $\V(s) = \V^\fv(s)$ for $s \in \W^\fv$. Therefore $\fw_G$ is an epistemic model based on the respective type of frame. Note that for all formulae $\gamma$ and all $s \in \W^\fw$ holds that $\fw_G, s \models \gamma$ if and only if $\fw, s \models \gamma$ (since worlds `below' $s$ play no role in evaluating formulae at $s$). Similarly, for all formulae $\gamma$ and all $s \in \W^\fv$ holds that $\fw_G, s \models \gamma$ if and only if $\fv, s \models \gamma$. Since $\varphi \vee \psi$ is valid, it holds that $\fw_G, u \models \varphi \vee \psi$ and so $\fw_G, u \models \varphi$ or $\fw_G, u \models \psi$. In the first case persistence of truth implies that $\fw_G, w \models \varphi$ and therefore $\fw, w \models \varphi$. Similarly, in the second case persistence of truth implies that $\fw_G, v \models \psi$ and hence $\fv, v \models \psi$. Both cases thus lead to contradiction, implying that $\varphi$ is valid or $\psi$ is valid. 
\end{proof}

Let us now consider the case for S5. We have already observed that for S4 frames each node in the intuitionistic tree rooted at $w$ is a modal successor of $w$, due to triangle confluence and the S4 frame conditions. When adding symmetry, it further follows that every node of the intuitionistic rooted at $w$ belongs to the equivalence of $w$ under $R_\agi$. This implies that whenever $w \leq v$ then $w$ and $v$ have the exact same modal successors. Therefore when evaluating a formula of the form $K_\agi \varphi \rightarrow \psi$ at $w$, we do not have to take higher worlds into account. In other words, such implications can be evaluated classically.

\begin{lemma}\label{l: classical truth conditions for implication}
    Let $\fw=(\W, \leq, \rel, \V)$ be an S5 model, $w \in \W$, $\agi \in \A$ and $\varphi, \psi \in \LICK$. Then the following hold.
    \begin{center}
        $\fw, w \models \K_\agi \varphi \rightarrow \psi$ if and only if $\fw, w \not \models \K_\agi \varphi$ or $\fw, w \models \psi$.
    \end{center}
\end{lemma}
\begin{proof}
    For the left-to-right direction suppose $\fw, w \models \K_\agi \varphi \rightarrow \psi$. Then for all $v \in w^\uparrow$ if $\fw, v \models \K_\agi \varphi$, then $\fw, v \models \psi$. Since $w \in w^\uparrow$ we obtain that $\fw, w \not \models \K_\agi \varphi$ or $\fw, w \models \psi$. For the right-to-left direction we proceed by contraposition. Suppose $\fw, w \not \models \K_\agi \varphi \rightarrow \psi$. Then there exists $w \leq v$ such that $\fw, v \models \K_\agi \varphi$ and $\fw, v \not \models \psi$. Since $w \leq v$ the Monotonicity Lemma implies $\fw, w \not \models \psi$. Now suppose $w \Rel_\agi u$. Since $w \leq v $ and $v \Rel_\agi v$ by reflexivity of $R_\agi$, triangle confluence implies that $w \Rel_\agi v$. By symmetry of $R_\agi$ we have $v \Rel_\agi w$ and so by transitivity of $R_\agi$, $v \Rel_\agi u$. Thus $\fw, u \models \varphi$, implying that $\fw, w \models \K_\agi \varphi$. 
\end{proof}

From this observation  we obtain that $\mathbf{ICKS5}$ satisfies a `boxed' version of the law of excluded middle: for any proposition $p$ the formula $\K_\agi p \vee \neg \K_\agi p$ is valid. 

\begin{lemma}\label{c: classical truth conditions for implications}
    Let $\varphi \in \LICK$ and $\agi \in \A$. Then
    \begin{enumerate}
        \item $ \K_\agi \varphi \vee \neg \K_\agi \varphi \in \mathbf{ICKS5}$;
        \item $\C \varphi \vee \neg \C \varphi \in \mathbf{ICKS5}$;
        \item $ \neg \K_\agi \varphi \rightarrow \K_\agi \neg \K_\agi \varphi \in \mathbf{ICKS5}$.
    \end{enumerate}
\end{lemma}
\begin{proof}
    For 1. let $\fw=(\W, \leq \rel, \V)$ be an S5-model. Let $w \in \W$, $\agi \in \A$ and $\varphi \in \LICK$. If $\fw,w \models K_\agi \varphi$, then $\fw, w \models K_\agi \varphi \vee \neg K_\agi \varphi$. Otherwise, $\fw,w \not \models \K_\agi \varphi$. Then there exists a world $v \in R_\agi(w)$ with $\fw, v \not \models \varphi$. For arbitrary $u \in w^\uparrow$, reflexivity and triangle confluence give $w \mathrel{R_\agi} u$. By symmetry and transitivity $u \mathrel{R_\agi} v$. Hence, for any $u \in w^\uparrow$ there exists a world $v$ with $u \mathrel{R_\agi} v$ such that $\fw,v \not \models \varphi$. Thus $\fw,w \models \neg \K_\agi \varphi$ and so $\fw, w \models \K_\agi \varphi \vee \neg \K_\agi \varphi$. \smallskip
 
    \noindent For 2. let $\fw=(\W, \leq \rel, \V)$ be an S5-model, $w \in \W$ and $\varphi \in \LICK$. If $\fw, w \models \C \varphi$, then $\fw, w \models \C \varphi \vee \neg \C \varphi$. Otherwise $\fw, w \not \models \C \varphi$, which implies that there exists a world $v \in R^*(w)$ with $\fw, v \not \models \varphi$. For any $u \in w^\uparrow$ we show by induction on the length $n$ of the path from $w$ to $v$ that $v \in R^*(u)$ and hence that $\fw, v \not \models \C \varphi$. For $n=0$ we have $w = v$ and hence that $\fw, w \not \models \varphi$. As before we have $u \Rel_\agi w$ by reflexivity, symmetry and triangle confluence and so $w \in R^*(u)$. For $n > 0$ we have that $w \Rel_{\agi_1} u_1 \Rel_{\agi_2} \ldots u_{n-1} \Rel_{\agi_{n-1}} v$. By induction hypothesis $u_{n-1} \in R^*(u)$ and hence $v \in R^*(u)$. Thus $\fw, u \not \models \C \varphi$, implying that $\fw, w \models \neg \C \varphi$ and so $\fw, w \models \C \varphi \vee \neg \C \varphi$. \smallskip

   \noindent For 3. let $\fw=(\W, \rel, \V)$ be an S5 model. Let $w \in W$, $\agi \in \A$ and $\varphi \in \LICK$. Let $v \in w^\uparrow$ and suppose that $\fw,v \models \neg \K_\agi \varphi$. Then there exists a world $v' \in R_\agi(v)$ with $\fw, v' \not \models \varphi$. Let $u \in \W$ such that $v \mathrel{R_\agi} u$. We want to show that $\fw,u \models \neg \K_\agi p$. To that end consider an arbitrary $u' \in u^\uparrow$. By reflexivity, triangle confluence and symmetry $u' \mathrel{R_\agi} u$. Moreover, since $v \mathrel{R_\agi} u$, also $u \mathrel{R_\agi} v$ and hence $u' \mathrel{R_\agi} v$. Finally, since $v \mathrel{R_\agi} v'$, $u' \mathrel{R_\agi} v'$. Hence, for any $u' \in u^\uparrow$ there exists $v'$ with $u' \mathrel{R_\agi} v'$ and $\fw,v' \not \models \varphi$. Therefore $\fw, u \models \neg \K_\agi \varphi$ and so $\fw, v \models \K_\agi \neg \K_\agi \varphi$. All together $\fw,w \models \neg \K_\agi \varphi \rightarrow \K_\agi \neg \K_\agi \varphi$.
\end{proof}

Knowledge over S5 models does therefore behave `classically' and not `intuitionistically'. In particular, the (semantic) disjunction property fails!

\begin{lemma}\label{l: disjunction property S5 fails}
    Let $p \in \Prop$ and $\agi \in \A$. The formula $\K_\agi p \vee \neg \K_\agi p$ is valid over the class of S5 frames, however neither $\K_\agi p$ nor $\neg \K_\agi p$ are valid.
\end{lemma}
\begin{proof}
    That $\K_\agi p \vee \neg \K_\agi p$ is valid over the class of S5-frames was proven in Corollary \ref{c: classical truth conditions for implications}. That neither $\K_\agi p$ nor $\neg \K_\agi p$ are valid is immediate.
\end{proof}

The proof of Lemma \ref{l: semantic disjunction property} breaks down for S5 models in the observation that worlds `below' $s$ do not play a role in evaluating formulae at $s$, since $u \leq s$ implies $s \Rel_\agi u$. Hence the property that $\fw_G, s \models \gamma$ if and only if $\fw, s \models \gamma$ for all $s \in \W^\fw$ does not hold.

\subsection{Interpretation}\label{ss: interpretation}

We present our interpretation of $\ICK$ as a logic to reason about knowledge over growing information states. We give one particular interpretation of agents as constructive mathematicians, but remark that many other interpretations are compatible with $\ICK$.

In 1965, Kripke introduced Kripke models for intuitionistic logic~\cite{Kripke_1965}, which are partially ordered sets $(\W, \leq)$ equipped with a monotone valuation function. Kripke established an interpretation of such models as information orderings: each world corresponds to an information state and the $\leq$-relation models extensions of states with additional information (but without losing any of the previously available data). In this view, propositions absent from a world are not considered to be false, but rather unsupported by the available information and may thus become true at a later state. Implications $p \rightarrow q$ then express consequences of gaining information: formally, $p \rightarrow q$ is true at world $w$ if the classical implication holds in all worlds $v \in w^\uparrow$; the formula $p \rightarrow q$ therefore expresses that if information $p$ is gained, information $q$  is gained as well. When applied to the philosophy of intuitionism, we may regard each information state $w$ as the set of theorems of mathematics that have been proven up to a certain time point. States $v \in w^\uparrow$ model later stages where more theorems have been added. In such reading, the truth relation $\models$ should not be read as 'truth' but rather as `has been proven': $w \models \varphi$ expresses that the statement formalized by $\varphi$ has been proven at $w$. In this setting $w \models p \to q$ expresses that if $p$ is proven, then we obtain a proof of $q$ as well. In other words, a proof of $p \to q$ is a transformation of proofs of $p$ into proofs of $q$, which aligns with the BHK interpretation of $\mathsf{IPL}$. On the other hand $w \not \models \varphi$ expresses that $\varphi$ has not (yet) been proven. The lack of proof does not imply that $\varphi$ is false or unprovable, as at a later point $v \geq w$ a proof of $\varphi$ might be found and so $v \models \varphi$. Negations $\neg \varphi$ instead express `unprovability': $w \models \neg \varphi$ means that a proof has been found showing that proofs of $p$ can be transformation into proofs of $\bot$, which implies that $p$ cannot be proven and so that $v \not \models \varphi$ for all $v \in w^\uparrow$. It follows that the \emph{law of excluded middle} does not hold.

Within this interpretation the logic $\ICK$ may be regarded as an epistemic logic to reason about what agents (mathematicians) know about mathematics. The actual world represents the set of theorems that has been proven up to this point. However, most mathematicians are unaware of all proven mathematical theorems. Instead, they only know a subset of the actual world, namely those theorems they proved themselves or those where they read and understood a proof. This is formalized by the modal accessibility relation, which relates each state with all states that some agent $\agi$ considers possible, i.e. all states which are consistent with $\agi$'s information.  Whenever a new theorem is proven, the actual world $w$ is updated with this additional information and thus changes from $w$ to $v$ for some $v \in w^\uparrow$. If the new theorem is proven by an agent different from $\agi$, it is highly likely that $\agi$ itself is unaware of the existence of this theorem (or the fact that it has ben proven) and so $\agi$ is unaware that the actual world has changed. To ensure that the change of $w$ to $v$ does not fundamentally change $\agi$'s knowledge (especially when $\agi$ is unaware of the change), triangle confluence captures that all worlds considered possible by $\agi$ at $v$ were also considered possible by $\agi$ at $w$. The modality $K_\agi$ then captures $\agi$'s knowledge. The modality $\E$ (`everybody knows') expresses the shared knowledge of the mathematical community. $\C$, on the other hand, expresses those mathematical facts that are commonly known by all agents. This includes logical reasoning principles commonly used in proofs or basic results that are reasonable to be claimed as common knowledge (e.g. `$2$ is even').
Considering $\ICK$ over the class of all frames implies that only very limited idealizations are placed on $\agi$'s knowledge; the only assumption is that $\agi$ is capable of applying modus ponens and therefore knows $\psi$ whenever they knows $\varphi$ and that $\varphi$ implies $\psi$. In particular it is not assumed that $\agi$ is infallible, so they may have proven a theorem wrongly and thus believe it to be true even though there is no correct proof; or worse, they might give a flawed proof of a wrong statement and thus believe a statement that is wrong. Therefore $\agi$ might not consider the actual world to be a possibility. Note that without reflexivity, agents do not necessarily consider worlds `above' the actual world possible. This is crucial for the belief interpretation, as $\agi$ might have proven a wrong theorem and extensions of the current world might contradict their beliefs. The logic $\ICK$ over the class of all epistemic frames may therefore be regarded as a weak epistemic logic reasoning about mathematicians beliefs about mathematics. Adding additional frame conditions on epistemic frames then corresponds to idealizations of mathematicians. Reflexivity corresponds to the fact that every agent $\agi$ considers the actual world a possibility and therefore that mathematicians do not make mistakes and only provide correct proofs. Reflexivity combined with triangle confluence guarantees that at $w$ agent $\agi$ considers all extensions $v \in w^\uparrow$ to be possible, which is a natural assumption for agents that are infallible: since $\agi$ only provides correct proofs, their knowledge is a subset of the actual world and every extension of it should be consistent with their knowledge. $\ICK$ over reflexive frames is thus an epistemic logic reasoning about infallible mathematicians knowledge about mathematics. S4 frame conditions further idealize the agents by assuming that they are aware of their own knowledge (positive introspection). So far, adding additional frame conditions has led to more and more idealized agents. Interestingly enough, this is no longer the case when moving to S5 epistemic frames. Here, the combination of S5 frame conditions with triangle confluence leads to models where the entire intuitionistic tree belongs to the equivalence class of agent $\agi$, implying (due to monotonicity) that $\agi$ only knows those formulas that are already true at the lowermost world. In that sense $\ICK$ over S5 frames models hyper-skeptical agents: mathematicians who only believe the most basic theorems (such as logical tautologies) and distrust any theorem that has ever been proven at a later point. This is also shown by the fact that the `boxed law of excluded middle' ($K_\agi p \vee \neg K_\agi p$ is valid) holds in this logic, showing that for any proposition either $\agi$ knows it (namely if the proposition is contained at the root) or $\agi$ will never know the proposition (since they distrust any proof). Overall, $\ICK$ can be regarded as an epistemic logic for reasoning about knowledge states of mathematicians with various degrees of idealization. Importantly, all mathematicians are constructivists in the sense that they do not believe (or know) that any proposition is either true or false. The interpretation presented here can of course be adapted to account for agents which are not mathematicians, but who require constructive reasoning, which might be of interest to areas where such reasoning is studied. We leave such considerations for future work.

\subsection{Expressivity}\label{ss: expressivity}

We finish this section with a few remarks about the expressivity of $\ICK$. We do not provide any proofs but refer to the relevant papers where the discussed results are established.

First, it is possible to evaluate formulae of $\LICK$ over the class of birelational models, which contains models that are not triangle confluent. In this case the Monotonicity Property fails, since modal successors of $w$ and $v$ for $w \leq v$ are unrelated. Thus it is possible that $K_\agi \varphi$ is true at $w$ but false at $v$. To retain monotonicity we need to use `intuitionistic' truth conditions for the epistemic modalities when evaluating formulae on birelational models. Namely,
\begin{center}
    \begin{tabular}{l l l}
$\fw, w \models_i \K_\agi \varphi$ & iff & for all $v \in w^\uparrow$ and for all $u \in R_\agi(v)$, $\fw, u \models_\agi \varphi$\\
$\fw, w \models_i\C \varphi$ & iff & for all $v \in (\Tilde{R})^*(w)$, $\fw, v \models_\agi \varphi$\\
    \end{tabular}
\end{center}
with $\Tilde{R}$ defined as $\bigcup_{R_\agi \in \rel}(R_\agi \circ \leq)$ where $\circ$ denotes relation composition.\footnote{In other words, $w \mathrel{\Tilde{R}} v$ holds if there exists $u$ with $w \leq u$ and $u \Rel_\agi v$ for some $\agi \in \A$.} Note that the intuitionistic truth conditions take into account exactly the same worlds as the standard truth conditions introduced before when the model is closed under triangle confluence. It is therefore not surprising that the following lemma holds. For a proof, see for example~\cite{litak_2018, Zenger_2025}.

\begin{lemma}
    Let $\mathbf{ICK_B}$ denote the set of valid $\LICK$-formulae over the class of birelational frames (with respect to $\models_i$). Then $\mathbf{ICK_B} = \mathbf{ICK}$.
\end{lemma}

Second, $\ICK$ admits an alternative interpretation as an intuitionistic version of \emph{linear temporal logic} ($\mathsf{LTL}$). For this the language $\LICK$ contains only one agent whose modality is denoted by $\X$, where $\X \varphi$ expresses that $\varphi$ is true in the next time step. The common knowledge operator $\C$ then expresses arbitrary (finite) iterations of $\X$ and is thus interpreted as the temporal `henceforth' operator $\G$, where $\G \varphi$ expresses that $\varphi$ is true in all future time steps. For this interpretation it is natural to restrict the class of birelational models to those models where the unique modal relation is a function. A \emph{functional model} is therefore a birelational model $\fw = (\W, \leq, f, \V)$ where $f$ is a partial function. A \emph{total functional model} is a functional model where $f$ is a total function which models time. Due to the lack of triangle confluence, the intuitionistic truth conditions are used to retain the Monotonicity Property. Surprisingly, the resulting logic once again coincindes with $\mathbf{ICK}$, implying that $\LICK$ cannot distinguish between birelational, epistemic and functional frames. Similarly, $\LICK$ cannot distinguish between serial birelational frames (i.e. birelational models where the modal accessibility relation is serial), serial epistemic frames and total functional frames. For a semantic proof of the following lemma, see~\cite[Theorem 3.21]{Zenger_2025}. For a proof theoretic proof, see {\cite[Theorem 5.15]{afshari_intuitionistic_2024}}.

\begin{lemma}\label{l: functional and triangle frames}
    Consider a unimodal version of $\LICK$ and let $\mathbf{ILTL}$ denote the set of valid $\LICK$-formulae over the class of functional frames (with respect to $\models_i$). Then $\mathbf{ILTL} = \mathbf{ICK}$.
\end{lemma}

This result can be generalized to an arbitrary finite number $n$ of agents and functional models with $n$ functional relations.

\section{Axiomatizations}\label{s: axiomatization}

This section discusses Hilbert-style axiomatizations for $\ICK$. We begin by introducing four axiomatizations which capture the validities of $\ICK$ over the classes of all frames, of reflexive frames, of S4 frames and of S5 frames respectively. The axiomatizations are \emph{modular} in the sense that the basic axiomatization $\hick$ which captures validities over the class of all frames is the foundation for the other three systems, which can be obtained by adding additional axioms. After establishing basic properties about derivability in these systems, we then show soundness by a standard argument and completeness via canonical model constructions. Axiomatizations for $\ICK$ over the classes of all frames and of reflexive frames were already developed and proven sound and complete by Marti in his PhD thesis~\cite{marti_2017}.\footnote{Marti's axiomatizations use slightly different axioms and rules for common knowledge than our axiomatizations.} Here we extend Marti's work to the frame classes S4 and S5. Moreover we provide a modular approach to axiomatizing $\ICK$ and give uniform proofs of soundness and completeness. Additionally, since the constructed canonical model for each axiomatization is finite, we obtain the finite model property for all considered logics as a corollary of completeness and therefore decidability of the validity problem (c.f. Table~\ref{tab:validity problem}).

\subsection{Basic Definitions and Properties}

An axiomatization consists of a finite set of \emph{axiom schemes} and a finite set of \emph{inference rules}. An \emph{inference rule} takes the shape
\begin{equation*}
    \infer[\mathsf{r}]{\psi}{\varphi_1 & \ldots & \varphi_n}
\end{equation*}
where $\varphi_1, \ldots, \varphi_n$ are called the \emph{premises} and $\psi$ is called the \emph{conclusion}. In the presentation of a rule the formulae $\varphi_i$ and $\psi$ are metavariables which can be uniformly substituted to obtain \emph{instances} of the rule. 
An axiom scheme is a rule without premises, i.e. a formula. We call instances of an inference rule \emph{rule instances} and instances of axiom schemes \emph{axioms}. We assume familiarity with standard proof-theoretic notions (see~\cite[Chapter 2]{Zenger_2025}).

\begin{table}
\centering
        \begin{tabular}{|l l|}
        \hline
         $\mathsf{Int}$: &  Intuitionistic tautologies  \\
         $\mathsf{K_\agi}$: & $\K_\agi(\varphi \rightarrow \psi) \rightarrow (\K_\agi \varphi \rightarrow \K_\agi \psi)$  \\
          $\mathsf{\K_\agi^c}$:& $(\K_\agi \varphi \rightarrow \psi) \rightarrow (\neg \K_\agi \varphi \vee \psi)$  \\
          $\mathsf{T_\agi}$:&  $\K_\agi \varphi \rightarrow \varphi$   \\
         $\mathsf{S4_\agi}$:& $\K_\agi \varphi \rightarrow \K_\agi \K_\agi \varphi$  \\
         $\mathsf{S5_\agi}$ & $\neg \K_\agi \varphi \rightarrow \K_\agi \neg \K_\agi \varphi$ \\
         $\mathsf{Fix}$:& $\C \varphi \leftrightarrow (\varphi \wedge \E \C \varphi)$  \\
         \hline
        \end{tabular}
     \quad
     \begin{tabular}{|l l|}
        \hline
        $\mathsf{MP}$: $\infer{\psi}{\varphi & \varphi \rightarrow \psi}$ & $\mathsf{Nec_\agi}$: $\infer{\K_\agi \varphi}{\varphi}$ \\
         $\mathsf{Mon}$: $\infer{\C \varphi \rightarrow \C \psi}{\varphi \rightarrow \psi}$ &  $\mathsf{Ind}$: $\infer{\varphi \rightarrow \C \varphi}{\varphi \rightarrow \E \varphi}$\\
         \hline
        \end{tabular}
     \caption{Axioms and rules for $\ICK$.}  
     \label{t: axioms and rules of Hilbert system}
\end{table}

\begin{definition}
    Consider the inference rules and axioms depicted in Table~\ref{t: axioms and rules of Hilbert system}. Define the following axiom systems.
    \begin{enumerate}
        \item $\hick$ consists of the axioms $\mathsf{Int}$, $\mathsf{K_\agi}$ for $\agi \in \A$, $\mathsf{Fix}$ and the rules $\mathsf{MP}$, $\mathsf{Nec_\agi}$ for $\agi \in \A$, $\mathsf{Mon}$ and $\mathsf{Ind}$.
        \item $\mathrm{ICKT}$ extends $\mathrm{ICK}$ by the axioms $\mathsf{T_\agi}$ for $\agi \in \A$.
        \item $\mathrm{ICKS4}$ extends $\mathrm{ICKT}$ by the axioms $\mathsf{S4_\agi}$ for $\agi \in \A$.
        \item $\mathrm{ICKS5}$ extends $\mathrm{ICKS4}$ by the axioms $\mathsf{\K_\agi^c}$ and $\mathsf{S5_\agi}$ for $\agi \in \A$.
    \end{enumerate}
\end{definition}
The axiom scheme $\mathsf{Int}$ consists of all \emph{intuitionistic tautologies}, i.e. all formulae $\varphi \in \LICK$ which are obtained from a valid formula of intuitionistic propositional logic via a uniform substitution. For example the formula $p \rightarrow p$ is intuitionistically valid and so $\K_\agi \C q \rightarrow \K_\agi \C q$ is an intuitionistic tautology (obtained from substituting $p$ by $\K_\agi \C q$). We make use of intuitionistic tautologies instead of building our axiomatizations on top of an axiomatization for intuitionistic logic to shorten derivations. The axioms $\mathsf{K_a}$, $\mathsf{T_a}$, $\mathsf{S4_a}$ and $\mathsf{S5_\agi}$ are the standard axioms for knowledge from classical epistemic logic (see e.g.~\cite{ditmarsch_2017}). The axiom $\mathsf{K_a}$ expresses that $K_\agi$ distributes over $\to$ and so that $K_\agi$ is a \emph{normal modality}. The axiom $\mathsf{T_a}$ expresses \emph{factivity of knowledge}: knowledge implies truth. Axiom $\mathsf{S4_a}$ expresses \emph{positive introspection}: if $\agi$ knows $\varphi$, then $\agi$ knows that they know $\varphi$. Axiom $\mathsf{S5_\agi}$ on the other hand expresses \emph{negative introspection}: $\agi$ is aware of the limitation of their knowledge, i.e. they know what they don't know. Axiom $\mathsf{\K_\agi^c}$ is the `classicality of knowledge' axiom, which expresses that knowledge behaves classically in the following sense. Note that $(\neg \K_\agi \varphi \vee \psi) \rightarrow (\K_\agi \varphi \rightarrow \psi)$ is valid over the class of all frames. Thus in the presence of $\mathsf{\K_\agi^c}$  knowledge satisfies the classical equivalence $(\neg \K_\agi \varphi \vee \psi) \leftrightarrow (\K_\agi \varphi \rightarrow \psi)$. The axiom $\mathsf{Fix}$ characterizes $\C \varphi$ as a fixed point of the propositional function $\mathsf{x} \mapsto \varphi \wedge \E \mathsf{x}$ and expresses the equivalence $\C \varphi \leftrightarrow \varphi \wedge \E \C \varphi$. The rules $\mathsf{MP}$ and $\mathsf{Nec_a}$ are standard, while $\mathsf{Mon}$ characterizes $\C$ as a monotone operator. The rule $\mathsf{Ind}$ formalizes induction: the rule expresses that if $\varphi \rightarrow \E \varphi$ and $\varphi$ have been obtained, then we can derive $\C \varphi$.

\emph{Derivations} in our axiom systems are finite trees labelled by formulae according to the axioms and rules of the system. We use standard terminology regarding trees, such as root, node, leaf, parent, child, and so on. For precise definitions, see~\cite[Chapter 2]{Zenger_2025}.
    
\begin{definition}
    Let $\Gamma \cup \{\varphi\} \subseteq \LICK$ and $\hickp \in \{ \hick, \hickt, \hicks, \hickss\}$. A \emph{$\hickp$-derivation of $\varphi$ with assumptions in $\Gamma$} is a finite tree $\pi$ labelled by $\LICK$-formulae such that the following hold.
    \begin{enumerate}
    \item The root is labelled by $\varphi$.
        \item Every leaf is labelled by an axiom or by a formula $\psi \in \Gamma$.
        \item Every parent node is labelled by the conclusion of a rule instance of $\hickp$ with its children labelled by the corresponding premises.
        \item If a node $u \in \pi$ is labelled by the conclusion of a rule instance of $\mathsf{Nec}$, $\mathsf{Mon}$ or $\mathsf{Ind}$, then every leaf of the subtree of $\pi$ rooted at $u$ is labelled by an axiom.
    \end{enumerate}
    We write $\Gamma \vdash_\mathrm{P} \varphi$ if there exists a $\hickp$-derivation of $\varphi$ with assumptions in $\Gamma$ and call $\varphi$ \emph{derivable from $\Gamma$ in $\hickp$}. If $\Gamma = \emptyset$, then we write $\vdash_\mathrm{P} \varphi$ and call $\varphi$ simply \emph{derivable in $\hickp$}. If the proof system $\mathrm{P}$ is clear from context, we also write $\Gamma \vdash \varphi$.
\end{definition}

Note that the rules $\mathsf{Nec}$, $\mathsf{Mon}$ and $\mathsf{Ind}$ cannot be applied to assumptions. We denote sets of formulae by $\Gamma$, $\Delta$ and so on and we write $\Gamma, \Delta$ for $\Gamma \cup \Delta$ and $\Gamma, \varphi$ for $\Gamma \cup \{\varphi\}$. If $\Gamma= \emptyset$, define $\bigwedge \Gamma \coloneqq \top$ and $\bigvee \Gamma \coloneqq \bot$. Derivations (with or without assumptions) are denoted by $\pi$ or $\tau$. Note that since $\hickt, \hicks$ and $\hickss$ are extensions of $\hick$, every $\hick$-derivation is also a derivation in the extended systems. 
Before we prove soundness, we establish basic properties about derivability in our systems, including the Deduction Theorem. The following lemma is straightforward.

\begin{lemma}\label{l: derivable formulae}
    Let $\varphi, \psi \in \LICK$ and $\hickp \in \{\hick, \hickt, \hicks, \hickss\}$.
    \begin{enumerate}
    \item\label{l: derivable formulae item 1} If $\vdash_\hickp \varphi$, then $\vdash_\hickp \gamma \rightarrow \varphi$ for any $\gamma \in \LICK$.
    \item\label{l: derivable formulae item 2} If $\vdash_{\hickp} \varphi \rightarrow \psi$ and $\vdash_{\hickp} \psi \rightarrow \chi$, then $\vdash_{\hickp} \varphi \rightarrow \chi$.
    \item\label{l: derivable formulae item 3} $\vdash_\hickp (\varphi \rightarrow \psi) \rightarrow (\neg \psi \rightarrow \neg \varphi)$ (contraposition).\footnote{Note that the other direction, i.e. $(\neg \psi \rightarrow \neg \varphi) \rightarrow (\varphi \rightarrow \psi)$ is \emph{not} valid intuitionistically!}
    \end{enumerate}
\end{lemma}
\begin{proof}
    The proofs are standard. For example for 1. suppose that $\vdash_\hickp \varphi$ and let $\gamma \in \LICK$. Since $\varphi \rightarrow (\gamma \rightarrow \varphi)$ is an intuitionistic tautology, we obtain $\vdash_\hickp \gamma \rightarrow \varphi$ by $\mathsf{MP}$. The other cases are similar.
\end{proof}

\begin{lemma}[Assumption weakening] \label{l: assumption weakening}
    Let $\hickp \in \{\hick, \hickt, \hicks, \hickss\}$. If $\Gamma \vdash_\mathrm{P} \varphi$ and $\Gamma \subseteq \Delta$, then $\Delta \vdash_\mathrm{P} \varphi$.
\end{lemma}
\begin{proof}
    This follows immediately from the definition of a derivation with assumptions.
\end{proof}

\begin{theorem}[Deduction Theorem]\label{t: deduction theorem for IM}
    Let $\hickp \in \{\hick, \hickt, \hicks, \hickss\}$ and let $\Gamma \cup \{\varphi, \psi\} \subseteq \LICK$. Then $\Gamma, \varphi \vdash_\mathrm{P} \psi$ if and only if $\Gamma \vdash_\mathrm{P} \varphi \rightarrow \psi$.
\end{theorem}
\begin{proof}
    For the direction from right to left suppose $\Gamma \vdash_\mathrm{P} \varphi \rightarrow \psi$. Hence $\Gamma, \varphi \vdash_\mathrm{P} \varphi \rightarrow \psi$ by Lemma \ref{l: assumption weakening}. Since $\Gamma, \varphi \vdash_\mathrm{P} \varphi$, applying $\mathsf{MP}$ to $\varphi$ and $\varphi \rightarrow \psi$ yields $\Gamma, \varphi \vdash_\mathrm{P} \psi$. \smallskip
    
    For the direction from left to right we proceed by induction on the height $h(\pi)$ of the derivation $\pi$ witnessing $\Gamma, \varphi \vdash_\mathrm{P} \psi$. \smallskip 
    
    \noindent \textsc{$h(\pi) = 0$}: Then either $\psi$ is an instance of an axiom scheme of $\mathrm{P}$ or $\psi \in \Gamma \cup \{\varphi\}$. In the first case $\vdash_\mathrm{P} \psi$ and thus by Lemma \ref{l: derivable formulae}, Item \ref{l: derivable formulae 2 item 2}. $\vdash_\mathrm{P} \varphi \rightarrow \psi$. By Lemma \ref{l: assumption weakening}, $\Gamma \vdash_\mathrm{P} \varphi \rightarrow \psi$. In the second case first suppose $\psi = \varphi$. Then $ \varphi \rightarrow \psi$ is an intuitionistic tautology and therefore $\Gamma \vdash_\mathrm{P} \varphi \rightarrow \psi$. Otherwise $\psi \in \Gamma$, implying that $\Gamma \vdash \psi$. Since $\psi \rightarrow (\varphi \rightarrow \psi)$ is an intuitionistic tautology, $\Gamma \vdash_\mathrm{P} \psi \rightarrow (\varphi \rightarrow \psi)$. Applying $\mathsf{MP}$ yields $\Gamma \vdash_\mathrm{P} \varphi \rightarrow \psi$. \smallskip 
    
   \noindent $h(\pi) > 0$: Consider the last rule applied in $\pi$ (i.e. the rule instance where the conclusion labels the root). \smallskip
    
     \noindent \textsc{1.} Suppose the last rule applied in $\pi$ is an instance of $\mathsf{MP}$:
        \begin{equation*}
            \infer{\psi}{\gamma & \gamma \rightarrow \psi}
        \end{equation*}
        The induction hypothesis implies that there are derivations $\pi_0, \pi_1$ witnessing $\Gamma \vdash_\mathrm{P} \varphi \rightarrow \gamma$ and $\Gamma \vdash_\mathrm{P} \varphi \rightarrow (\gamma \rightarrow \psi)$. The following derivation shows $\Gamma \vdash \varphi \to \psi$:
        \begin{prooftree}
            \AxiomC{$\pi_0$}
            \noLine
            \UnaryInfC{$\varphi \rightarrow \gamma$}
            \AxiomC{$\pi_1$}
            \noLine
            \UnaryInfC{$ \varphi \rightarrow (\gamma \rightarrow \psi)$}
            \AxiomC{}
            \RightLabel{$\mathsf{Int}$}
            \UnaryInfC{$ (\varphi \rightarrow (\gamma \rightarrow \psi)) \rightarrow ((\varphi \rightarrow \gamma) \rightarrow (\varphi \rightarrow \psi))$}
            \RightLabel{$\mathsf{MP}$}
            \BinaryInfC{$ (\varphi \rightarrow \gamma) \rightarrow (\varphi \rightarrow \psi)$}
            \RightLabel{$\mathsf{MP}$}
            \BinaryInfC{$ \varphi \rightarrow \psi$}
        \end{prooftree}
  
    \noindent \textsc{2.} Suppose the last rule applied in $\pi$ is an instance of $\mathsf{r}$ for $\mathsf{r} \in \{\mathsf{Nec}, \mathsf{Mon}, \mathsf{Ind}\}$, with premise $\gamma$ and conclusion $\psi$. By definition of a derivation from assumptions, every leaf of $\pi$ is labelled by an axiom, implying that $\vdash_\mathrm{P} \psi$. Lemma \ref{l: derivable formulae}, Item \ref{l: derivable formulae 2 item 2} implies that $\vdash_\mathrm{P} \varphi \rightarrow \psi$ and so, by Lemma \ref{l: assumption weakening}, $\Gamma \vdash_\mathrm{P} \varphi \rightarrow \psi$. \qedhere
\end{proof}

\begin{corollary}\label{c: deduction theorem}
    For any finite set of formulae $\Gamma \cup \{\varphi\}$ the following holds:
    \begin{center}
        $\Gamma \vdash_\mathrm{P} \varphi$ if and only if $\vdash_\mathrm{P} \bigwedge \Gamma \rightarrow \varphi$.
    \end{center}
\end{corollary}
\begin{proof}
    By induction on $\lvert \Gamma \rvert$. For $\lvert \Gamma \rvert = 0$ we have that $\Gamma = \emptyset$ and so that $\vdash_\mathrm{P} \varphi$ holds if and only if $\vdash_\mathrm{P} \top \rightarrow \varphi$ holds (since $\varphi \rightarrow (\top \rightarrow \varphi)$ and $(\top \to \varphi) \to \varphi$ are intuitionistic tautologies). \smallskip \\
    Suppose $\lvert \Gamma \rvert = n+1$. Let $\Gamma =  \Gamma_0 \cup \Gamma_1$ where $\Gamma_0 = \{ \varphi_1, \ldots, \varphi_n \}$ and $\Gamma_1 = \{\varphi_{n+1}\}$. By the Deduction Theorem $\Gamma \vdash_\mathrm{P} \varphi$ if and only if $\Gamma_0 \vdash_\mathrm{P} \varphi_{n+1} \rightarrow \varphi$. By induction hypothesis
    \begin{equation*}
        \Gamma_0 \vdash_\mathrm{P} \varphi_{n+1} \rightarrow \varphi \text{ if and only if } \vdash_\mathrm{P} \bigwedge \Gamma_0 \rightarrow (\varphi_{n+1} \rightarrow \varphi).
    \end{equation*}
    It remains to show that
    \begin{equation*}\label{e: c: deduction theorem}
        \vdash_\mathrm{P} \bigwedge \Gamma_0 \rightarrow (\varphi_{n+1} \rightarrow \varphi) \text{ if and only if } \vdash \bigwedge \Gamma \rightarrow \varphi.
    \end{equation*}
   Note that $(p \rightarrow (q \rightarrow r)) \rightarrow((p \wedge q) \rightarrow r)$ is intuitionistically valid. Thus
    \begin{equation*}
         (\bigwedge \Gamma_0 \rightarrow (\varphi_{n+1} \rightarrow \varphi)) \rightarrow (\bigwedge \Gamma \rightarrow \varphi)
    \end{equation*}
    is an intuitionistic tautology and hence derivable. Applying $\mathsf{MP}$ yields
    \begin{equation*}
        \vdash_\mathrm{P}  \bigwedge \Gamma_0 \rightarrow (\varphi_{n+1} \rightarrow \varphi) \text{ implies }  \vdash \bigwedge \Gamma \rightarrow \varphi.
    \end{equation*}
    The direction from right-to-left is similar but uses $((p \wedge q) \rightarrow r) \rightarrow (p \rightarrow (q \rightarrow r))$ instead.  Hence,
    \begin{equation*}
        \Gamma \vdash_\mathrm{P} \varphi \text{ if and only if } \vdash_\mathrm{P} \bigwedge \Gamma \rightarrow \varphi.
    \end{equation*}
    which concludes the proof.
\end{proof}

We finish the subsection by establishing some derivable properties of knowledge in our systems. 

\begin{lemma}\label{l: derivable formulae 2}
   Let $\varphi, \psi \in \LICK$, $\agi \in \A$ and $\hickp \in \{\hick, \hickt, \hicks, \hickss\}$.
   \begin{enumerate}
       \item\label{l: derivable formulae 2 item 1} $\vdash_{\hickp} \K_\agi (\varphi \wedge \psi) \rightarrow (\K_\agi \varphi \wedge \K_\agi \psi)$ and $\vdash_{\hickp} (\K_\agi \varphi \wedge \K_\agi \psi) \rightarrow \K_\agi (\varphi \wedge \psi)$.
    \item\label{l: derivable formulae 2 item 2} $\vdash_{\hickp} (\K_\agi \varphi \vee \K_\agi \psi) \rightarrow \K_\agi (\varphi \vee \psi)$.
    \item\label{l: derivable formulae 2 item 3} If $\vdash_{\hickp} \varphi$, then $\vdash_{\hickp} \C \varphi$ (necessitation for $\C$).
    \item\label{l: derivable formulae 2 item 4} $\vdash_{\hickss} K_\agi \varphi \vee \neg K_\agi \varphi$.
   \end{enumerate}
\end{lemma}
\begin{proof}
    \noindent 1. Observe that $(\varphi \wedge \psi) \rightarrow \varphi$ and $(\varphi \wedge \psi) \rightarrow \psi$ are intuitionistic tautologies. Applying $\mathsf{Nec_\agi}$ yields $\vdash_\hickp \K_\agi ((\varphi \wedge \psi) \rightarrow \varphi)$ and $\vdash_\hickp \K_\agi ((\varphi \wedge \psi) \rightarrow \psi)$. By the $\mathsf{K_\agi}$-axiom and $\mathsf{MP}$ we obtain $\vdash_\hickp \K_\agi(\varphi \wedge \psi) \rightarrow \K_\agi \varphi$ and $\vdash_\hickp \K_\agi(\varphi \wedge \psi) \rightarrow \K_\agi \psi$. Since 
    \begin{equation*}
     (\K_\agi(\varphi \wedge \psi) \rightarrow \K_\agi \varphi) \rightarrow ((\K_\agi(\varphi \wedge \psi) \rightarrow \K_\agi \psi) \rightarrow (\K_\agi(\varphi \wedge \psi) \rightarrow (\K_\agi \varphi \wedge \K_\agi \psi)))
    \end{equation*}
    is an intuitionistic tautology, we obtain $\vdash_\hickp \K_\agi(\varphi \wedge \psi) \rightarrow (\K_\agi \varphi \wedge \K_\agi \psi)$ by applying $\mathsf{MP}$ twice. For the other direction observe that
    \begin{equation}\label{e: 1}
        \vdash_\hickp \varphi \rightarrow (\psi \rightarrow (\varphi \wedge \psi))
    \end{equation}
    as this formula is an intuitionistic tautology. Applying $\mathsf{Nec_\agi}$ yields
    \begin{equation}\label{e: 2}
        \vdash_\hickp \K_\agi (\varphi \rightarrow (\psi \rightarrow (\varphi \wedge \psi))).
    \end{equation}
    Note that the following formula is an instance of $\mathsf{K}$ and is hence derivable:
    \begin{equation}\label{e: 3}
        \vdash_\hickp \K_\agi (\varphi \rightarrow (\psi \rightarrow (\varphi \wedge \psi))) \rightarrow (\K_\agi \varphi \rightarrow \K_\agi (\psi \rightarrow (\varphi \wedge \psi))).
    \end{equation}
    Applying $\mathsf{MP}$ to (\ref{e: 2}) and (\ref{e: 3}) yields
    \begin{equation}\label{e: 4}
        \vdash_\hickp \K_\agi \varphi \rightarrow \K_\agi (\psi \rightarrow (\varphi \wedge \psi)).
    \end{equation}
    The following is an instance of $\mathsf{K_\agi}$:
    \begin{equation}\label{e: 5}
        \vdash_\hickp \K_\agi (\psi \rightarrow (\varphi \wedge \psi)) \rightarrow (\K_\agi \psi \rightarrow \K_\agi (\varphi \wedge \psi)).
    \end{equation}
    Item \ref{l: derivable formulae item 2} of Lemma \ref{l: derivable formulae} applied to (\ref{e: 4}) and (\ref{e: 5}) yields
    \begin{equation}\label{e: 6}
        \vdash_\hickp \K_\agi \varphi \rightarrow (\K_\agi \psi \rightarrow \K_\agi (\varphi \wedge \psi)).
    \end{equation}
    Observe that the following is an intuitionistic tautology and thus derivable:
    \begin{equation}\label{e: 7}
        \vdash_\hickp (\K_\agi \varphi \rightarrow (\K_\agi \psi \rightarrow \K_\agi (\varphi \wedge \psi))) \rightarrow ((\K_\agi \varphi \wedge \K_\agi \psi) \rightarrow \K_\agi (\varphi \wedge \psi))
    \end{equation}
    Applying $\mathsf{MP}$ to (\ref{e: 6}) and (\ref{e: 7}) thus yields
    \begin{equation}\label{e: 8}
       \vdash_\hickp (\K_\agi \varphi \wedge \K_\agi \psi) \rightarrow \K_\agi (\varphi \wedge \psi)
    \end{equation}
    which concludes the proof. \smallskip
    
    \noindent 2. Observe that $\varphi \rightarrow (\varphi \vee \psi)$ and $\psi \rightarrow (\varphi \vee \psi)$ are intuitionistic tautologies. Applying $\mathsf{Nec_\agi}$ yields $\vdash_\hickp \K_\agi(\varphi \rightarrow (\varphi \vee \psi))$ and $\vdash_\hickp \K_\agi(\psi \rightarrow (\varphi \vee \psi))$. By using the $\mathsf{K_\agi}$-axiom and $\mathsf{MP}$ we obtain $\vdash_\hickp \K_\agi \varphi \rightarrow \K_\agi (\varphi \vee \psi)$ and $\vdash_\hickp \K_\agi \psi \rightarrow \K_\agi (\varphi \vee \psi)$. Since
    \begin{equation*}
        (\K_\agi \varphi \rightarrow \K_\agi(\varphi \vee \psi)) \rightarrow ((\K_\agi \psi \rightarrow \K_\agi(\varphi \vee \psi)) \rightarrow ((\K_\agi \varphi \vee \K_\agi \psi) \rightarrow \K_\agi(\varphi \vee \psi)))
    \end{equation*}
    is an intuitionistic tautology, we obtain $\vdash_\hickp (\K_\agi \varphi \vee \K_\agi \psi) \rightarrow \K_\agi (\varphi \vee \psi)$ by applying $\mathsf{MP}$ twice. \smallskip 
    
    \noindent 3. Suppose that $\vdash_\hickp \varphi$. Applying $\mathsf{Nec_\agi}$ yields $\vdash_\hickp \K_\agi \varphi$. By Item \ref{l: derivable formulae item 1}. of Lemma \ref{l: derivable formulae} we obtain $\vdash_\hickp \varphi \rightarrow \K_\agi \varphi$. Since $\agi \in \A$ was arbitrary, we obtain $\vdash_\hickp \varphi \rightarrow \E \varphi$ by repeatedly using the (correct substitution instances of the) intuitionistic tautology 
    \begin{equation*}
        (p \rightarrow q) \rightarrow ((p \rightarrow r) \rightarrow (p \rightarrow (q \wedge r))
    \end{equation*}
    and $\mathsf{MP}$. Applying $\mathsf{Ind}$ yields $\vdash_\hickp \varphi \rightarrow \C \varphi$. Finally, by applying $\mathsf{MP}$ we obtain $\vdash_\hickp \C \varphi$. \smallskip

    \noindent 4. By axiom $\mathsf{\K_\agi^c}$ we have $\vdash_\hickss (\K_\agi \varphi \rightarrow \K_\agi \varphi) \rightarrow (\neg \K_\agi \varphi \vee \K_\agi \varphi)$. Since $\K_\agi \varphi \rightarrow \K_\agi \varphi$ is an intuitionistic tautology we obtain $\vdash_\hickss \neg \K_\agi \varphi \vee \K_\agi \varphi$ by $\mathsf{MP}$. Finally, $\vdash_\hickss {\K_\agi \varphi \vee \neg  \K_\agi \varphi}$ since $(\neg \K_\agi \varphi \vee \K_\agi \varphi) \to (\K_\agi \varphi \vee \neg \K_\agi \varphi)$ is an intuitionistic tautology.
\end{proof}




\subsection{Soundness}

We will now show that each given axiomatization is sound with respect to its corresponding class of frames: $\vdash_\hickp \varphi$ implies $\varphi \in \mathbf{P}$. The proof strategy is standard: we first show that axioms are valid and rules preserve validity. Soundness is then established by induction on the height of proofs in the usual way.

\begin{lemma}\label{l: axioms of ICK are valid}
    Every instance of an axiom scheme from Table \ref{t: axioms and rules of Hilbert system} is valid over the corresponding class of epistemic frames.
\end{lemma}
\begin{proof}
    The case for $\mathsf{Int}$ is omitted. The cases for $\mathsf{K_\agi}$, $\mathsf{T_\agi}$ $\mathsf{S4_\agi}$ and $\mathsf{Fix}$ were proven in Lemma~\ref{l: valid formulae}. The cases for $\mathsf{\K_\agi^c}$ and $\mathsf{S5_\agi}$ were shown in Lemma~\ref{c: classical truth conditions for implications}.
\end{proof}

\begin{lemma}\label{l: rules are valid}
    The rules $\mathsf{MP}$, $\mathsf{Nec}$, $\mathsf{Mon}$ and $\mathsf{Ind}$ preserve validity over the class of epistemic frames (and hence over the classes of reflexive, S4 and S5 epistemic frames as well): if the premises are valid, then the conclusion is valid, too.
\end{lemma}
\begin{proof}
 We show the contrapositive: if the conclusion is falsifiable, then so is one of the premises. Let $\fw=(\W, \leq, \rel, \V)$ be an epistemic model and let $w \in \W$ be a world. The cases for $\mathsf{MP}$ and $\mathsf{Nec}$ are standard and omitted. \smallskip
 
 
 \noindent \textsc{Rule $\mathsf{Mon}$.} Suppose that $\fw,w \not \models \C \varphi \rightarrow \C \psi$. Then there exists $v \in w^\uparrow$ with $\fw,v \models \C \varphi$ but $\fw,v \not \models \C \psi$. Hence there exists $u \in R^*(v)$ with $\fw,u \not \models \psi$. However, $\fw,u \models \varphi$ by assumption and so $\fw,u \not \models \varphi \rightarrow \psi$. Therefore $\varphi \rightarrow \psi$ is falsifiable. \smallskip
 
 \noindent \textsc{Rule $\mathsf{Ind}$.} Suppose that $\fw,w \not \models \varphi \rightarrow \C \varphi$. So there exists $v \in w^\uparrow$ with $\fw, v \models \varphi$ and $\fw, v \not \models \C \varphi$. Let $n$ be the least natural number such that there exist $u_0, \ldots, u_n \in \W$ with $u_0 = v$, for all $i \in [n]$ there exists $\agi_i \in \A$ such that $u_i \mathrel{R_{\agi_i}} u_{i+1}$, and $\fw, u_n \not \models \varphi$. By assumption $n > 0$. Then $\fw, u_{n-1} \models \varphi$ but $\fw, u_{n-1} \not \models \E \varphi$. Hence $\fw, u_{n-1} \not \models \varphi \rightarrow \E \varphi$, implying that $\varphi \rightarrow \E \varphi$ is falsifiable.
\end{proof}

\begin{theorem}[Soundness]\label{t: soundness}
    The axiomatizations $\hickp \in \{\hick, \hickt, \hicks, \hickss\}$ are sound: if $ \vdash_\mathrm{P} \varphi$, then $ \varphi$ is valid over the respective class of epistemic frames.
\end{theorem}
\begin{proof}
    By a standard induction on the height of derivations using Lemma~\ref{l: axioms of ICK are valid} and Lemma~\ref{l: rules are valid}. The details are omitted. \qedhere
\end{proof}

\subsection{Completeness}

We will now establish completeness of our axiomatizations. Completeness is shown via a canonical model construction, where we construct for each formula $\varphi$ a canonical model based on the correct type of frame. Our construction faces three main difficulties. First, since $\ICK$ is built on top of $\mathsf{IPL}$, we must adapt the standard canonical model construction for $\mathsf{IPL}$ (see e.g.~\cite{Bezhanishvili_?}) which uses \emph{prime theories} instead of \emph{maximally consistent sets} as the worlds of the canonical model. Second, we must account for the various frame conditions that our canonical models should satisfy. Third, the presence of common knowledge implies that our logics are not \emph{compact}. Therefore for the standard canonical model construction the \emph{Truth Lemma} fails. To solve this issue we restrict to a finite fragment of the language which contains all formulae necessary to evaluate $\varphi$ correctly. As a consequence our canonical models are finite and so we obtain the finite model property from the completeness proof.

We begin by introducing the auxiliary notions needed to define canonical models and establishing useful properties to prove the Truth Lemma. From now on let $\hickp$ stand for any of the axiomatizations in $\{\hick, \hickt, \hicks, \hickss\}$. Fix an arbitrary finite and closed set of formulae $\Sigma \subseteq \LICK$. Since knowledge over S5 frames behaves classically, whenever $K_\agi \varphi$ belongs to $\Sigma$, we must also take its negation $\neg K_\agi \varphi$ into account when dealing with $\hickss$. Therefore if $\hickp = \hickss$, we additionally assume $\Sigma$ to be negation closed.

\begin{definition}\label{d: sigma prime theory}\index{prime theory!for $\IM$}
    A \emph{$\Sigma$-prime theory} (relative to $\mathrm{P}$) is a set of formulae $\Gamma \subseteq \Sigma$ such that the following hold.
    \begin{enumerate}
        \item $\Gamma$ is \emph{consistent}: $\Gamma \not \vdash_\mathrm{P} \bot$.
        \item $\Gamma$ is \emph{deductively closed} with respect to $\Sigma$: if $\Gamma \vdash_\mathrm{P} \varphi$ and $\varphi \in \Sigma$, then $\varphi \in \Gamma$.
        \item $\Gamma$ satisfies the \emph{disjunction property}: if $\varphi \vee \psi \in \Gamma$, then $\varphi \in \Gamma$ or $\psi \in \Gamma$.
    \end{enumerate}
\end{definition}

$\Sigma$-prime theories will form the worlds of the canonical model. For the S5 case, we require prime theories to satisfiy an additional property which we call \emph{$\K$-maximal consistency}.

\begin{definition}
   Let $\Sigma$ be negation closed. A $\Sigma$-prime theory $\Gamma$ is \emph{$\K$-maximally consistent} if for any $\agi \in \A$ and any $\K_\agi \varphi \in \Sigma$, $\K_\agi \varphi \in \Gamma$ if and only if $\neg \K_\agi \varphi \not \in \Gamma$.
\end{definition}

$\K$-maximally consistent prime theories satisfy the property that for any $\K_\agi \varphi \in \Sigma$, either $\K_\agi \varphi$ or $\neg \K_\agi \varphi$ belongs to the prime theory. In that sense the theories are maximally consistent, since for any formula $\K_\agi \varphi \in \Sigma$ if $\Gamma$ is a $\K$-maximally consistent prime theory not containing $\K_\agi \varphi$ (or $\neg \K_\agi \varphi$, respectively), then $\Gamma \cup \{ \K_\agi \varphi \}$ ($\Gamma \cup \{ \neg \K_\agi \varphi \})$ is inconsistent. Observe that not every $\Sigma$-prime theory relative to $\hickss$ is $\K$-maximally consistent. For a simple counterexample consider $\Sigma = \{\bot, p, \K_\agi p, \neg \K_\agi p\}$ and note that $\Gamma = \{p\}$ is a $\Sigma$-prime theory but not $K$-maximally consistent. We now show that every consistent set of formulae $\Gamma \subseteq \Sigma$ can be extended into a prime theory.

\begin{lemma}[Lindenbaum]\label{l: Lindenbaum for IM}
    If $\Gamma \subseteq \Sigma$ and $\Gamma \not \vdash_\mathrm{P} \varphi$, then there exists a $\Sigma$-prime theory $\Delta$ (relative to $\mathrm{P}$) with $\Gamma \subseteq \Delta$ and $\Delta \not \vdash_\mathrm{P} \varphi$. If, furthermore, $\hickp = \hickss$ and $\Sigma$ is negation closed, then $\Delta$ is $\K$-maximally consistent.
\end{lemma}
\begin{proof}
Suppose $\Gamma \not \vdash_\mathrm{P} \varphi$ and consider an enumeration $\psi_0, \ldots, \psi_k$ of the formulae in $\Sigma$. We first show how to construct a set of formulae $\Delta$ by induction on $n \leq k$.
\begin{itemize}
    \item Define $\Delta_0 \coloneqq \Gamma$
    \item Define
    \begin{equation*}
        \Delta_{n+1} \coloneqq \begin{cases}
                               \Delta_n \cup \{\psi_n\}, \text{ if } \Delta_n, \psi_n \not \vdash_\mathrm{P} \varphi \\\
                               \Delta_n, \text{ otherwise }\\
                        \end{cases}
    \end{equation*}
\end{itemize}
Observe that $\Delta_n \subseteq \Delta_{n+1}$ for each $n \leq k$. Define $\Delta := \Delta_{k+1}$. \smallskip

A simple inductive proof shows that $\Delta_n \not \vdash_\mathrm{P} \varphi$ for each $n \leq k+1$. In fact, the base case holds by assumption and the induction step follows immediately by construction and the induction hypothesis. Therefore $\Delta \not \vdash_\mathrm{P} \varphi$, implying that $\Delta$ is consistent. Next, for showing that $\Delta$ is deductively closed, suppose that $\Delta \vdash_\mathrm{P} \chi$ for some $\chi \in \Sigma$. Then there exists $n \leq k$ such that $\chi = \psi_n$. Consider $\Delta_{n+1}$. If $\psi_n \in \Delta_{n+1}$, then $\psi_n \in \Delta$ and we are done. Otherwise, $\Delta_n, \psi_n \vdash_\mathrm{P} \varphi$, implying that $\Delta, \psi_n \vdash_\mathrm{P} \varphi$. By the Deduction Theorem $\Delta \vdash_\mathrm{P} \psi_n \rightarrow \varphi$. But since $\Delta \vdash_\mathrm{P} \psi_n$ by assumption, $\Delta \vdash_\mathrm{P} \varphi$ which is a contradiction. Hence $\psi_n \in \Delta_{n+1}$ and so $\psi_n \in \Delta$, implying that $\Delta$ is deductively closed with respect to $\Sigma$. For the disjunction property, suppose $\chi \vee \gamma \in \Delta$. Since $\Delta \subseteq \Sigma$, $\chi \vee \gamma \in \Sigma$ and since $\Sigma$ is closed also $\chi, \gamma \in \Sigma$. Therefore there are $n_1, n_2 \leq k$ with $\chi = \psi_{n_1}$ and $\gamma = \psi_{n_2}$. Suppose towards contradiction that $\psi_{n_1} \not \in \Delta$ and $\psi_{n_2} \not \in \Delta$. Thus $\Delta_{n_1}, \psi_{n_1} \vdash_\mathrm{P} \varphi$ and $\Delta_{n_2}, \psi_{n_2} \vdash_\mathrm{P} \varphi$. Therefore also $\Delta, \psi_{n_1} \vdash_\mathrm{P} \varphi$ and $\Delta, \psi_{n_2} \vdash_\mathrm{P} \varphi$. By the Deduction Theorem $\Delta \vdash_\mathrm{P} \psi_{n_1} \rightarrow \varphi$ and $\Delta \vdash_\mathrm{P} \psi_{n_2} \rightarrow \varphi$. Since $\psi_{n_1} \vee \psi_{n_2} \in \Delta$, also $\Delta \vdash_\mathrm{P} \psi_{n_1} \vee \psi_{n_2}$. Observe that
\begin{equation*}
    (\psi_{n_1} \rightarrow \varphi) \rightarrow ((\psi_{n_2} \rightarrow \varphi) \rightarrow ((\psi_{n_1} \vee \psi_{n_2}) \rightarrow \varphi))
\end{equation*}
is an intuitionistic tautology. Hence, by applying $\mathsf{MP}$ three times we obtain $\Delta \vdash_\mathrm{P} \varphi$, a contradiction. Therefore $\psi_{n_1} \in \Delta$ or $\psi_{n_2} \in \Delta$, implying that $\Delta$ satisfies the disjunction property and is hence a $\Sigma$-prime theory extending $\Gamma$, such that $\Delta \not \vdash_\mathrm{P} \varphi$. \smallskip

Now let $\Sigma$ be negation closed and $\hickp = \hickss$. Suppose towards contradiction that $\Delta$ is not $\K$-maximally consistent. Then there exists $\agi \in \A$ and $\K_\agi \psi \in \Sigma$ such that either (1) $\K_\agi \psi \in \Delta$ and $\neg \K_\agi \psi \in \Delta$ or (2) $\K_\agi \psi \not \in \Delta$ and $\neg \K_\agi \psi \not \in \Delta$. Clearly, (1) immediately gives $\Delta \vdash_\hickss \bot$, which is a contradiction. So suppose (2). First of all since $\Sigma$ is negation closed and $\K_\agi \psi \in \Sigma$, also $\neg \K_\agi \psi \in \Sigma$, implying that there are $n_1, n_2 \leq k$ with $\K_\agi \psi = \psi_{n_1}$ and $\neg \K_\agi \psi = \psi_{n_2}$. By assumption $\psi_{n_1} \not \in \Delta$, implying that $\Delta_{n_1}, \psi_{n_1} \vdash_\hickss \varphi$ and so $\Delta, \psi_{n_1} \vdash_\hickss \varphi$. Similarly, $\Delta, \psi_{n_2} \vdash_\hickss \varphi$ as well. By the Deduction Theorem $\Delta \vdash_\hickss \K_\agi \psi \rightarrow \varphi$ and $\Delta \vdash_\hickss \neg \K_\agi \psi \rightarrow \varphi$. By Lemma \ref{l: derivable formulae 2}, Item \ref{l: derivable formulae 2 item 4}, $\Delta \vdash_\hickss \K_\agi \psi \vee \neg \K_\agi \psi$ and so $\Delta \vdash_\hickss \varphi$ by the same argument as above; contradiction. Hence $\Delta$ is $\K$-maximally consistent.
\end{proof}

Next, we establish that prime theories locally behave like worlds in a model. The following lemma will thus prove useful in the proof of the Truth Lemma.

\begin{lemma}\label{l: properties of prime theories}
    Let $\Gamma$ be a $\Sigma$-prime theory (relative to $\mathrm{P}$).
    \begin{enumerate}
        \item\label{l: properties of prime theories, Item 1} If $\varphi \wedge \psi \in \Sigma$, then $\varphi \wedge \psi \in \Gamma$ if and only if $\varphi \in \Gamma$ and $\psi \in \Gamma$.
        \item\label{l: properties of prime theories, Item 2} If $\varphi \vee \psi \in \Sigma$, then $\varphi \vee \psi \in \Gamma$ if and only if $\varphi \in \Gamma$ or $\psi \in \Gamma$.
        \item\label{l: properties of prime theories, Item 3}  $\C \varphi \in \Gamma$ if and only if $\varphi \in \Gamma$ and $\{\K_\agi \C \varphi \mid \agi \in \A\} \subseteq  \Gamma$.
        \item\label{l: properties of prime theories, Item 4} If $\mathsf{T_\agi} \in \mathrm{P}$, then $\K_\agi \varphi \in \Gamma$ implies $\varphi \in \Gamma$.
    \end{enumerate}
\end{lemma}
\begin{proof}
 Let $\Gamma$ be a $\Sigma$-prime theory and let $\varphi \ast \psi \in \Sigma$ for $\ast \in \{\wedge, \vee\}$. First of all, observe that since $\Sigma$ is closed, $\varphi, \psi \in \Sigma$. \smallskip \\
  1. For the direction from left-to-right suppose $\varphi \wedge \psi \in \Gamma$. Then $\Gamma \vdash_\hickp \varphi \wedge \psi$. Since $(\varphi \wedge \psi) \rightarrow \varphi$ and $(\varphi \wedge \psi) \rightarrow \psi$ are intuitionistic tautologies, applying $\mathsf{MP}$ yields $\Gamma \vdash_\hickp \varphi$ and $\Gamma \vdash_\hickp \psi$. As $\Gamma$ is deductively closed with respect to $\Sigma$, $\varphi \in \Gamma$ and $\psi \in \Gamma$. For the other direction suppose $\varphi \in \Gamma$ and $\psi \in \Gamma$. Then $\Gamma \vdash_\hickp \varphi$ and $\Gamma \vdash_\hickp \psi$. Observe that $\varphi \rightarrow (\psi \rightarrow (\varphi \wedge \psi))$ is an intuitionistic tautology. Hence $\Gamma \vdash_\hickp \varphi \wedge \psi$ by two applications of $\mathsf{MP}$. Thus, as $\Gamma$ is deductively closed with respect to $\Sigma$, $\varphi \wedge \psi \in \Gamma$. \smallskip \\
2. The direction from left-to-right follows from $\Gamma$ being a $\Sigma$-prime theory and therefore satisfying the disjunction property. For the other direction suppose without loss of generality that $\varphi \in \Gamma$. Observe that $\varphi \rightarrow (\varphi \vee \psi)$ is an intuitionistic tautology. Therefore $\Gamma \vdash_\hickp \varphi \rightarrow (\varphi \vee \psi)$. Applying $\mathsf{MP}$ yields $\Gamma \vdash_\hickp \varphi \vee \psi$ and thus, since $\Gamma$ is deductively closed with respect to $\Sigma$, $\varphi \vee \psi \in \Gamma$. \smallskip \\
3. For the direction from left-to-right suppose $\C \varphi \in \Gamma$. Then $\Gamma \vdash_\hickp \C \varphi$ and so by axiom $\mathsf{Fix}$, $\Gamma \vdash_\hickp \varphi \wedge \E \C \varphi$. Thus $\Gamma \vdash_\hickp \varphi$, implying that $\varphi \in \Gamma$, and $\Gamma \vdash \E \C \varphi$. Since  $\E \C \varphi \to \K_\agi \C \varphi$ is an intuitionistic tautology, we obtain $\Gamma \vdash_\hickp \K_\agi \varphi$. Since $\K_\agi \varphi \in \Sigma$, we conclude that $\K_\agi \varphi \in \Gamma$ for all $\agi \in \A$. The other direction is similar.\smallskip \\
4. Suppose $\mathsf{T_\agi} \in \mathrm{P}$ and $\K_\agi \varphi \in \Gamma$. Then $\Gamma \vdash_\hickp \K_\agi \varphi \rightarrow \varphi$ and hence $\Gamma \vdash_\hickp \varphi$, implying that $\varphi \in \Gamma$. \qedhere

\end{proof}

For a set of formulae $\Gamma$ and $\agi \in \A$, define the following notation.
\begin{align*}
        \K_\agi \Gamma & := \{ \K_\agi \varphi \mid \varphi \in \Gamma \} &
         \neg \K_\agi \Gamma & \coloneqq \{\neg \K_\agi \varphi \mid \varphi \in \Gamma\} \\
        \K_\agi^{-1} \Gamma & := \{\varphi \mid \K_\agi \varphi \in \Gamma\} &
         \neg \K_\agi^{-1} \Gamma & \coloneqq \{\varphi \mid \neg \K_\agi \varphi \in \Gamma\}
\end{align*}
We denote by $\mathcal{T}_\Sigma^\mathrm{P}$ the set of $\Sigma$-prime theories relative to $\mathrm{P}$ for $\hickp \in \{\hick, \hickt, \hicks\}$ and by $\mathcal{T}_\Sigma^\hickss$ the set of $\K$-maximally consistent $\Sigma$-prime theories relative to $\hickss$.

\begin{definition}
    Let $\hickp \in \{\hick, \hickt, \hicks, \hickss\}$ and let $\Sigma$ be a non-empty, finite and closed (negation closed in case $\hickp = \hickss$) set of formulae. The \emph{canonical model} relative to $\Sigma$ and $\hickp$ is given by $\fw^{(\Sigma, \mathrm{P})} = (\W^{(\Sigma,\mathrm{P})}, \leq^{(\Sigma,\mathrm{P})}, \rel^{\Sigma, \hickp}, \V^{(\Sigma,\mathrm{P})})$ where
    \begin{itemize}
        \item $\W^{(\Sigma, \mathrm{P})} \coloneqq \mathcal{T}_\Sigma^\hickp$;
        \item $\Gamma \leq^{(\Sigma, \mathrm{P})} \Delta$ if and only if $\Gamma \subseteq \Delta$ for $\Gamma, \Delta \in \W^{(\Sigma, \mathrm{P})}$;
        \item $\rel^{(\Sigma, \hickp)} = \{R_\agi^{(\Sigma, \hickp)} \mid \agi \in \A\}$ where for each $\agi \in \A$, $R_\agi^{(\Sigma, \mathrm{P})} \subseteq \W^{(\Sigma, \mathrm{P})} \times W^{(\Sigma, \mathrm{P})}$ given by
        \begin{itemize}
            \item if $\hickp = \hick$ or $\hickp = \hickt$, then $\Gamma \Rel_\agi^{(\Sigma, \mathrm{P})} \Delta$ iff $\K_\agi^{-1} \Gamma \subseteq \Delta$;
            \item if $\hickp = \hicks$, then $\Gamma \Rel_\agi^{(\Sigma, \mathrm{P})} \Delta$ iff $\K_\agi \K_\agi^{-1} \Gamma \subseteq \Delta$;
            \item if $\hickp = \hickss$, then $\Gamma \Rel_\agi^{(\Sigma, \mathrm{P})} \Delta$ iff $\K_\agi \K_\agi^{-1} \Gamma = \K_\agi \K_\agi^{-1} \Delta$;
        \end{itemize}
        \item  $\V^{(\Sigma, \mathrm{P})}(\Gamma) := \Gamma \cap \Prop$ for $\Gamma \in \W^{(\Sigma, \mathrm{P})}$.
    \end{itemize}
\end{definition}

\begin{lemma}\label{l: canonical model}
    Let $\Sigma$ be a non-empty, finite and closed set of $\LICK$-formulae.
    \begin{itemize}
        \item $\fw^{(\Sigma, \hick)}$ is a finite epistemic model.
        \item  $\fw^{(\Sigma, \hickt)}$ is a finite reflexive model.
        \item  $\fw^{(\Sigma, \hicks)}$ is a finite S4 model.
        \item  If $\Sigma$ is negation closed, then $\fw^{(\Sigma, \hickss)}$ is a finite S5 model.
    \end{itemize}
\end{lemma}
\begin{proof}
    For any $\hickp \in \{\hick, \hickt, \hicks, \hickss\}$, note that since $\Sigma$ is finite, $\mathcal{T}_\Sigma^\hickp$ is finite as well, implying that $\fw^{(\Sigma, \hickp)}$ is finite. Furthermore, $\W^{(\Sigma, \hickp)}$ is non-empty by the Lindenbaum Lemma and Theorem \ref{t: soundness}. That $(\W^{(\Sigma, \hickp)}, \leq^{(\Sigma, \hickp)}$) is a poset follows immediately from $\subseteq$ being reflexive, transitive and antisymmetric. Moreover, the monotonicity of $\V^{(\Sigma, \hickp)}$ follows directly from construction. It remains to show that each $R_\agi^{(\Sigma, \hickp)}$ is triangle confluent and satisfies the correct frame conditions. For the frame conditions note that for each $\agi \in \A$ the relation $R_\agi^{(\Sigma, \hickt)}$ is reflexive: since $\mathsf{T_\agi} \in \hickt$ we have for each $\Gamma \in \W^{(\Sigma, \hickt)}$ that $\Gamma \vdash_{\hickt} \K_\agi \varphi \rightarrow \varphi$. Thus since $\Gamma$ is deductively closed with respect to $\Sigma$ we have $\K_\agi^{-1} \Gamma \subseteq \Gamma$ and so $\Gamma \in R_\agi^{(\Sigma, \hickt)}(\Gamma)$. For $\hicks$ clearly $\K_\agi \K_\agi^{-1} \Gamma \subseteq \Gamma$ for each $\agi \in \A$ and each $\Gamma \in \W^{(\Sigma, \hicks)}$ and so each $R_\agi^{(\Sigma, \hicks)}$ is reflexive. Transitivity follows from construction. Finally, for $\hickss$ it follows directly that each $R_\agi^{(\Sigma, \hickss)}$ is an equivalence relation since $=$ is one. For triangle confluence suppose $\Gamma, \Delta, \Omega \in \W^{(\Sigma, \hickp)}$ for $\hickp \in \{\hick, \hickt\}$ with $\Gamma \leq^{(\Sigma, \hickp)} \Delta$ and $\Delta \Rel_\agi^{(\Sigma, \hickp)} \Omega$. Then $\Gamma \subseteq \Delta$ and $\K_\agi^{-1} \Delta \subseteq \Omega$, which implies that $\K_\agi^{-1} \Gamma \subseteq \Omega$ as well. Hence $\Gamma \Rel_\agi^{(\Sigma, \hickp)} \Omega$. The case for $\hicks$ follows by a similar argument. Finally suppose $\Gamma, \Delta, \Omega \in \W^{(\Sigma, \hickss)}$ and $\Gamma \leq^{(\Sigma, \hickss)} \Delta$ and $\Delta \Rel_\agi^{(\Sigma, \hickss)} \Omega$. Then $\Gamma \subseteq \Delta$ and $\K_\agi \K_\agi^{-1} \Delta = \K_\agi \K_\agi^{-1} \Omega$. From $\Gamma \subseteq \Delta$ we obtain that $\K_\agi \K_\agi^{-1} \Gamma \subseteq \K_\agi \K_\agi^{-1} \Delta$. Suppose towards contradiction that there exists $\K_\agi \varphi \in \Delta$ such that $\K_\agi \varphi \not \in \Gamma$. Since $\Gamma$ is $\K$-maximally consistent we have $\neg \K_\agi \varphi \in \Gamma$ and so $\neg \K_\agi \varphi \in \Delta$. Therefore $\Delta \vdash_\hickss \bot$, a contradiction. Hence $\K_\agi \K_\agi^{-1} \Gamma = \K_\agi \K_\agi^{-1} \Delta = \K_\agi \K_\agi^{-1} \Omega$, implying that $\Gamma \Rel_\agi^{(\Sigma, \hickss)} \Omega$.
\end{proof}

The last step to obtain completeness is to establish the Truth Lemma. For the common knowledge case, we will need to use formal induction over the set of prime theories reachable from the current world. The following formulae will help us achieve this. For $\Gamma \in \W^{(\Sigma, \hickp)}$, let the \emph{reachable component} of $\Gamma$ be the set $R^{(\Sigma, \hickp)^*}(\Gamma)$ where $R^{(\Sigma, \hickp)} = \bigcup_{\agi \in \A} R_\agi^{(\Sigma, \hickp)}$. Note that $R^{(\Sigma, \hickp)^*}(\Gamma)$ is finite.

\begin{definition}\label{d: characteristic formula}
    Let $\Gamma \in \W^{(\Sigma, \hickp)}$. The \emph{characteristic formula} of $\Gamma$ is the formula
    \begin{equation*}
        \chi(\Gamma) \coloneqq \bigwedge \Gamma.
    \end{equation*}
    Furthermore, the \emph{reachable component formula} is the formula
    \begin{equation*}
        \gamma(\Gamma) \coloneqq \bigvee_{\Delta \in R^{(\Sigma, \hickp)^*}(\Gamma)} \chi(\Delta).
    \end{equation*}
\end{definition}

The characteristic formula $\chi(\Gamma)$ characterizes the prime theory $\Gamma$. Note that due to the lack of the duals of $\K_\agi$ and $\C$ in the language of $\ICK$ we do not require to encode any of the \emph{negative} information in the characteristic formula of a prime theory. The reachable component formula $\gamma(\Gamma)$ characterizes the reachable component of $\Gamma$. Together these formulae will help us to encode induction over the worlds in the reachable component of $\Gamma$. The following lemma establishes some useful properties of the characteristic formula and the reachable component formula which will be readily used in the proof of the Truth Lemma. 

\begin{lemma}\label{l: properties characteristic formula IM}
    Let $\Gamma \in \W^{(\Sigma, \hickp)}$.  The following hold.
    \begin{enumerate}
        \item $\Gamma \vdash_\mathrm{P} \chi(\Gamma)$.
        \item For any $\Delta \in (R^{(\Sigma,\hickp)})^*(\Gamma)$, $\Delta \vdash_\mathrm{P} \gamma(\Gamma)$.
        \item If $\psi \in \LICK$ such that $\psi \in \Delta$ for every $\Delta \in R^{(\Sigma, \hickp)^*}(\Gamma)$, then $\vdash_\mathrm{P} \gamma(\Gamma) \rightarrow \psi$.
    \end{enumerate}
\end{lemma}
\begin{proof}
    1. is trivial. For 2. note that $\chi(\Delta)$ is one of the disjuncts of $\gamma(\Gamma)$, implying that $\chi(\Delta) \rightarrow \gamma(\Gamma)$ is an intuitionistic tautology. Hence, 1. implies that $\Delta \vdash \gamma(\Gamma)$. For 3. by assumption $\psi$ is a conjunct of each $\chi(\Delta)$ occurring as a disjunct of $\gamma(\Gamma)$, which immediately implies that $\gamma(\Gamma) \rightarrow \psi$ is an intuitionistic tautology and hence derivable.
\end{proof}

\begin{lemma}[Truth Lemma] \label{l: truth lemma}
Let $\hickp \in \{\hick, \hickt, \hicks, \hickss\}$, $\Sigma$ a finite, non-empty and closed set of formulae (negation closed if $\hickp = \hickss$), and $\fw^{(\Sigma, \hickp)}$ the canonical model relative to $\Sigma$ and $\hickp$. Then for any $\Gamma \in \W^{(\Sigma, \hickp)}$ and any formula $\varphi \in \Sigma$ the following holds.
\begin{center}
    $\varphi \in \Gamma$ if and only if $\fw^{(\Sigma, \hickp)}, \Gamma \models \varphi$.
\end{center}
\end{lemma}
\begin{proof}
    Let $\hickp \in \{\hick, \hickt, \hicks, \hickss\}$. We proceed by induction on the structure of $\varphi$. The case where $\varphi = \bot$ is trivial, and the case where $\varphi = p$ for $p \in \Prop$ follows immediately from the definition. For the induction step, the cases where $\varphi = \psi \wedge \chi$ or $\varphi = \psi \vee \chi$ follow immediately from Lemma \ref{l: properties of prime theories}, the induction hypothesis and $\Sigma$ being closed. \smallskip
    
    \noindent \textsc{Case for $\rightarrow$.} Suppose $\varphi = \psi \rightarrow \chi$. For the left-to-right direction suppose $\psi \rightarrow \chi \in \Gamma$. Let $\Delta \in \W^{(\Sigma, \hickp)}$ with $\Gamma \subseteq \Delta$ and suppose $\fw^{(\Sigma, \hickp)}, \Delta \models \psi$. By induction hypothesis $\psi \in \Delta$ implying that $\Delta \vdash_\mathrm{P} \psi$. Since $\Delta$ extends $\Gamma$, we have $\psi \rightarrow \chi \in \Delta$ and thus $\Delta \vdash_\mathrm{P} \psi \rightarrow \chi$. Applying $\mathsf{MP}$ yields $\Delta \vdash_\mathrm{P} \chi$. As $\Delta$ is deductively closed with respect to $\Sigma$, $\chi \in \Delta$. The induction hypothesis yields $\fw^{(\Sigma, \hickp)}, \Delta \models \chi$. $\Delta$ was an arbitrary $\Sigma$-prime theory extending $\Gamma$, therefore $\fw^{(\Sigma, \hickp)}, \Gamma \models \psi \rightarrow \chi$. 
    For the right-to-left direction we proceed by contraposition. Suppose $\psi \rightarrow \chi \not \in \Gamma$. Then $\Gamma \not \vdash_\mathrm{P} \psi \rightarrow \chi$. The Deduction Theorem yields $\Gamma, \psi \not \vdash_\mathrm{P} \chi$. By the Lindenbaum Lemma there exists $\Delta \in \W^{(\Sigma, \hickp)}$ with $\Gamma \cup \{\psi\} \subseteq \Delta$ and $\Delta \not \vdash_\mathrm{P} \chi$. Hence, $\psi \in \Delta$ and $\chi \not \in \Delta$. The induction hypothesis yields $\fw^{(\Sigma, \hickp)}, \Delta \models \psi$ and $\fw^{(\Sigma, \hickp)}, \Delta \not \models \chi$. As $\Gamma \leq^{(\Sigma, \hickp)} \Delta$, we have $\fw^{(\Sigma, \hickp)}, \Gamma \not \models \psi \rightarrow \chi$. \smallskip
    
    \noindent \textsc{Case for $\K_\agi$.} Suppose $\varphi = \K_\agi \psi$. We distinguish the following cases depending on $\hickp$. \smallskip
   
    \setlength{\leftskip}{0.5cm} \noindent \textsc{Case for $\hickp \in \{\hick, \hickt\}$}.  For the left-to-right direction suppose $\K_\agi \psi \in \Gamma$. Let $\Delta \in \W^{(\Sigma, \hickp)}$ be any $\Sigma$-prime theory such that $\Gamma \Rel_\agi^{(\Sigma, \hickp)} \Delta$ holds. By definition $\K_\agi^{-1} \Gamma \subseteq \Delta$. Hence, $\psi \in \Delta$. By induction hypothesis $\fw^{(\Sigma, \hickp)}, \Delta \models \psi$. As $\Delta$ was arbitrary, we conclude that $\fw^{(\Sigma, \hickp)}, \Gamma \models \K_\agi \psi$. For the right-to-left direction suppose $\fw^{(\Sigma, \hickp)}, \Gamma \models \K_\agi \psi$. Note that
    \begin{equation}\label{e: truth lemma1}
        \K_\agi^{-1} \Gamma \vdash_\mathrm{P} \psi
    \end{equation}
     as otherwise, by the Lindenbaum Lemma, there would exist a $\Sigma$-prime theory $\Delta$ with $\K_\agi^{-1} \Gamma \subseteq \Delta$ and $\Delta \not \vdash_\mathrm{P} \psi$. Hence $\psi \not \in \Delta$, and so by induction hypothesis $\fw^{(\Sigma, \hickp)}, \Delta \not \models \psi$, contradicting that $\fw^{(\Sigma, \hickp)}, \Gamma \models \K_\agi \psi$. Therefore (\ref{e: truth lemma1}) holds. Since $\K_\agi^{-1}\Gamma$ is a finite set, from (\ref{e: truth lemma1}) and Corollary \ref{c: deduction theorem} we obtain
    \begin{equation}\label{e: truthlemma2}
        \vdash_\mathrm{P} \bigwedge \K_\agi^{-1} \Gamma \rightarrow \psi.
    \end{equation}
    By $\mathsf{Nec_\agi}$ we obtain $\vdash_\mathrm{P} \K_\agi (\bigwedge \K_\agi^{-1} \Gamma \rightarrow \psi) $ and therefore $\vdash_\mathrm{P} \K_\agi (\bigwedge \K_\agi^{-1} \Gamma) \rightarrow \K_\agi \psi$ by the axiom $\mathsf{K_\agi}$ and $\mathsf{MP}$. We claim that
    \begin{equation}\label{e: truthlemma3}
         \K_\agi \K_\agi^{-1} \Gamma \vdash_\mathrm{P} \K_\agi (\bigwedge \K_\agi^{-1} \Gamma).
    \end{equation}
    If true, then applying $\mathsf{MP}$ to $\vdash_\mathrm{P} \K_\agi (\bigwedge \K_\agi^{-1} \Gamma) \rightarrow \K_\agi \psi$ and (\ref{e: truthlemma3}) gives $\K_\agi \K_\agi^{-1} \Gamma \vdash \K_\agi \psi$. Lemma \ref{l: assumption weakening} then yields $\Gamma \vdash \K_\agi \psi$ implying that $\K_\agi \psi \in \Gamma$. We prove (\ref{e: truthlemma3}) by induction on the size of $\K_\agi^{-1} \Gamma$. For $\lvert \K_\agi^{-1} \Gamma \rvert = 0$, the statement reduces to $\emptyset \vdash \K_\agi \top$ which is trivially true. For $\lvert \K_\agi^{-1} \Gamma \rvert = k+1$ let $\K_\agi^{-1}\Gamma_0 = \{ \varphi_1, \ldots, \varphi_k\}$ and let $\K_\agi^{-1}\Gamma = \K_\agi^{-1} \Gamma_0 \cup \{ \varphi_{k+1}\}$. By induction hypothesis
    \begin{equation*}
        \K_\agi \K_\agi^{-1} \Gamma_0 \vdash_\mathrm{P} \K_\agi (\bigwedge \K_\agi^{-1} \Gamma_0)
    \end{equation*}
    Hence also $\K_\agi \K_\agi^{-1} \Gamma \vdash_\mathrm{P} \K_\agi (\bigwedge \K_\agi^{-1} \Gamma_0)$. Since $\K_\agi \K_\agi^{-1} \Gamma \vdash_\mathrm{P} \K_\agi \varphi_{k+1}$ we also obtain
    \begin{equation*}
        \K_\agi \K_\agi^{-1} \Gamma \vdash_\mathrm{P} \K_\agi (\bigwedge \K_\agi^{-1} \Gamma_0) \wedge \K_\agi \varphi_{k+1}.
    \end{equation*}
     Lemma \ref{l: derivable formulae}, Item \ref{l: derivable formulae item 2}. yields
     \begin{equation*}
         \K_\agi \K_\agi^{-1} \Gamma \vdash_\mathrm{P} \K_\agi (\bigwedge \K_\agi^{-1} \Gamma)
     \end{equation*}
     which finishes the proof of the claim. \smallskip

     \noindent \textsc{Case for $\hickp = \hicks$.} For the left-to-right direction suppose $\K_\agi \psi \in \Gamma$. Let $\Delta \in \W^{(\Sigma, \hicks)}$ be any $\Sigma$-prime theory such that $\Gamma \Rel_\agi^{(\Sigma, \hicks)} \Delta$ holds. By definition $\K_\agi \K_\agi^{-1} \Gamma \subseteq \Delta$. Hence $\K_\agi \psi \in \Delta$. Since $\Delta \vdash_\hicks \K_\agi \psi \rightarrow \psi$, applying $\mathsf{MP}$ gives $\Delta \vdash_\hicks \psi$ and so $\psi \in \Delta$. The induction hypothesis yields $\fw^{(\Sigma, \hicks)}, \Delta \models \psi$. As $\Delta$ was arbitrary $\fw^{(\Sigma, \hicks)}, \Gamma \models \K_\agi \psi$. For the right-to-left direction suppose $\fw^{(\Sigma, \hicks)}, \Gamma \models \K_\agi \psi$. Note that
    \begin{equation}\label{e: completeness S4 K1}
       \K_\agi \K_\agi^{-1} \Gamma \vdash_\hicks \psi
    \end{equation}
    as otherwise there would exist, by the Lindenbaum Lemma, a prime theory $\Delta \supseteq \K_\agi \K_\agi^{-1} \Gamma$ with $\Delta \not \vdash_\hicks \psi$. Hence $\psi \not \in \Delta$ and so by induction hypothesis $\fw^{(\Sigma, \hicks)}, \Delta \not \models \psi$. By construction $\Gamma \Rel_\agi^{(\Sigma, \hicks)} \Delta$, implying that $\fw^{(\Sigma, \hicks)}, \Gamma \not \models \K_\agi \psi$; contradiction. Hence (\ref{e: completeness S4 K1}) holds. By Corollary \ref{c: deduction theorem} we have
    \begin{equation}\label{e: completeness S4 K2}
        \vdash_\hicks \bigwedge \K_\agi \K_\agi^{-1} \Gamma \rightarrow \psi.
    \end{equation}
    By $\mathsf{Nec_\agi}$, $\mathsf{K_\agi}$ and $\mathsf{MP}$ we obtain
    \begin{equation}\label{e: completeness S4 K3}
        \vdash_\hicks \K_\agi \bigwedge \K_\agi \K_\agi^{-1} \Gamma \rightarrow \K_\agi \psi.
    \end{equation}
   We claim that $\vdash_\hicks \bigwedge \K_\agi \K_\agi \K_\agi^{-1} \Gamma \rightarrow \K_\agi \bigwedge \K_\agi \K_\agi^{-1} \Gamma$ and we prove the claim by induction on $\lvert \K_\agi^{-1} \Gamma \rvert$. For ${\lvert \K_\agi^{-1} \Gamma \rvert = 0}$ the statement reduces to $\vdash_\hicks  \top \rightarrow \K_\agi \top$, which is easy to prove. For ${\lvert \K_\agi^{-1} \Gamma \rvert = k+1}$, let $\K_\agi^{-1} \Gamma = \K_\agi^{-1}\Gamma_0 \cup \{\varphi_{k+1}\}$ where $ \K_\agi^{-1} \Gamma_0= \{\varphi_1, \ldots,\varphi_k\}$. By induction hypothesis
   \begin{equation}\label{e: completeness S4 K4}
       \vdash_\hicks \bigwedge \K_\agi \K_\agi \K_\agi^{-1}\Gamma_0 \rightarrow \K_\agi \bigwedge \K_\agi \K_\agi^{-1} \Gamma_0.
   \end{equation}
   Since $(p_1 \rightarrow q_1) \rightarrow ((p_2 \rightarrow q_2) \rightarrow ((p_1 \wedge p_2) \rightarrow (q_1 \wedge q_2)))$ is an intuitionistic tautology and $\vdash_\hicks \K_\agi \K_\agi \varphi_{k+1} \rightarrow \K_\agi \K_\agi \varphi_{k+1}$, we obtain
   \begin{equation}\label{e: completeness S4 K5}
       \vdash_\hicks \bigwedge \K_\agi \K_\agi \K_\agi^{-1}\Gamma \rightarrow ((\K_\agi \bigwedge \K_\agi \K_\agi^{-1} \Gamma_0) \wedge \K_\agi \K_\agi \varphi_{k+1})
   \end{equation}
 By Lemma \ref{l: derivable formulae 2}, Item \ref{l: derivable formulae 2 item 1} we have
 \begin{equation}\label{e: completeness S4 K6}
     \vdash_\hicks ((\K_\agi \bigwedge \K_\agi \K_\agi^{-1} \Gamma_0) \wedge \K_\agi \K_\agi \varphi_{k+1}) \rightarrow \K_\agi \bigwedge \K_\agi \K_\agi^{-1} \Gamma.
 \end{equation}
 Hence we we obtain $\vdash_\hicks \bigwedge \K_\agi \K_\agi \K_\agi^{-1} \Gamma \rightarrow \K_\agi \bigwedge \K_\agi \K_\agi^{-1} \Gamma$ by Lemma \ref{l: derivable formulae}, Item \ref{l: derivable formulae item 2}  as claimed. Now since $\Gamma \vdash_\hicks \K_\agi \gamma$ for each $\K_\agi \gamma \in \Gamma$, axiom $\mathsf{S4_\agi}$ implies that $\Gamma \vdash_\hicks \K_\agi \K_\agi \gamma$ for each $\K_\agi \gamma \in \Gamma$ and so, using standard arguments, $\Gamma \vdash_\hicks \bigwedge \K_\agi \K_\agi \K_\agi^{-1} \Gamma$. Hence $\Gamma \vdash_\hicks \K_\agi \bigwedge \K_\agi \K_\agi^{-1} \Gamma$ and so we obtain $\Gamma \vdash_\hicks \K_\agi \psi$ from (\ref{e: completeness S4 K3}). Therefore $\K_\agi \psi \in \Gamma$. \smallskip

 \noindent \textsc{Case for $\hickss$.} For the direction from left-to-right suppose $\K_\agi \psi \in \Gamma$ and let $\Delta \in \W^{(\Sigma, \hickss)}$ be any world such that $\Gamma \Rel_\agi^{(\Sigma, \hickss)} \Delta$. By definition $\K_\agi \K_\agi^{-1} \Gamma = \K_\agi \K_\agi^{-1} \Delta$, implying that $\K_\agi \psi \in \Delta$. Thus $\psi \in \Delta$ by axiom $\mathsf{T_\agi}$ and so $\fw^{(\Sigma, \hickss)}, \Delta \models \psi$ by induction hypothesis. As $\Delta$ was arbitrary we conclude $\fw^{(\Sigma, \hickss)}, \Gamma \models \K_\agi \psi$. For the right-to-left direction suppose $\fw^{(\Sigma, \hickss)}, \Gamma \models \K_\agi \psi$. We claim that
 \begin{equation}\label{e: completeness S5 K1}
     \K_\agi \K_\agi^{-1} \Gamma, \neg \K_\agi \neg \K_\agi^{-1} \Gamma \vdash_\hickss \psi.
 \end{equation}
 In fact if (\ref{e: completeness S5 K1}) would not hold, then by the Lindenbaum Lemma there would exist a $\K$-maximally consistent prime theory $\Delta$ such that (1) $ \K_\agi \K_\agi^{-1} \Gamma \cup \neg \K_\agi \neg \K_\agi^{-1} \Gamma \subseteq \Delta$ and (2) $\Delta \not \vdash_\hickss \psi$. From (2) follows by the induction hypothesis that $\fw^{(\Sigma, \hickss)} \not \models \psi$. From (1) follows that $\Gamma \Rel_\agi^{(\Sigma, \hickss)} \Delta$: By construction $\K_\agi \K_\agi^{-1} \Gamma \subseteq \Delta$. Suppose $\K_\agi \chi \in \Delta$. Since $\Gamma$ is $\K$-maximally consistent we have either (i) $\K_\agi \chi \in \Gamma$ or (ii) $\neg \K_\agi \chi \in \Gamma$. Note that (ii) implies that $\neg \K_\agi \chi \in \Delta$ which implies that $\Delta \vdash_\hickss \bot$. Hence (i) holds, i.e. $\K_\agi \K_\agi^{-1} \Gamma = \K_\agi \K_\agi^{-1} \Delta$. Thus (1) and (2) imply that $\fw^{(\Sigma, \hickss)}, \Gamma \not \models \K_\agi \psi$; a contradiction. We conclude that (\ref{e: completeness S5 K1}) holds. By Corollary \ref{c: deduction theorem} we have
 \begin{equation}\label{e: completeness S5 K2}
     \vdash_\hickss \bigwedge \Omega \rightarrow \psi
 \end{equation}
 where $\Omega = \K_\agi \K_\agi^{-1} \Gamma \cup \neg \K_\agi \neg \K_\agi^{-1} \Gamma$. By $\mathsf{Nec_\agi}$, $\mathsf{K_\agi}$ and $\mathsf{MP}$ we obtain
 \begin{equation}\label{e: completeness S5 K3}
     \vdash_\hickss \K_\agi \bigwedge \Omega \rightarrow \K_\agi \psi.
 \end{equation}
We can then show that 
 \begin{equation}\label{e: completeness S5 K4}
     \vdash_\hickss \bigwedge \K_\agi \Omega \rightarrow \K_\agi \bigwedge \Omega
 \end{equation}
 by induction on $\lvert \Omega \rvert$. The proof is similar to the corresponding case for $\hicks$ and omitted. Now since $\Gamma \vdash_\hickss \K_\agi \gamma$ for each $\K_\agi \gamma \in \Gamma$, axiom $\mathsf{S4_\agi}$ implies $\Gamma \vdash_\hickss \K_\agi \K_\agi \gamma$ for each $\K_\agi \gamma \in \Gamma$. Similarly since $\Gamma \vdash_\hickss \neg \K_\agi \gamma$ for each $\neg \K_\agi \gamma \in \Gamma$, axiom $\mathsf{S5_\agi}$ implies $\Gamma \vdash_\hickss \K_\agi \neg \K_\agi \gamma$ for each $\neg \K_\agi \gamma \in \Gamma$. Thus $\Gamma \vdash_\hickss \bigwedge \K_\agi \Omega$ and so $\Gamma \vdash_\hickss \K_\agi \bigwedge \Omega$. Therefore $\Gamma \vdash_\hickss \K_\agi \psi$ which implies that $\K_\agi \psi \in \Gamma$.
 \smallskip
     
   \setlength{\leftskip}{0cm}  \noindent \textsc{Case for $\C$.} Suppose $\varphi = \C \psi$. We distinguish the following cases depending on $\hickp$. \smallskip

   \setlength{\leftskip}{0.5cm} \noindent \textsc{Case for $\hickp \in \{\hick, \hickt\}$.} For the left-to-right direction suppose $\C \psi \in \Gamma$. Let $\Delta \in R^{(\Sigma, \hickp)^*}(\Gamma)$. Then there are prime theories $\Delta_0, \ldots, \Delta_n$ with $\Delta_0 = \Gamma$, $ \Delta_n = \Delta$ and for each $k \in [n]$ there exists $\agi_k \in \A$ with $\Delta_k \Rel_\agi^{(\Sigma, \hickp)} \Delta_{k+1}$. We prove that $\psi \in \Delta_n$ and $\C \psi \in \Delta_n$ by induction on $n$. For $n=0$, $\C \psi \in \Delta_0 = \Gamma$ by assumption. Moreover, Lemma~\ref{l: properties of prime theories}, Item~\ref{l: properties of prime theories, Item 3}. implies that $\varphi \in \Gamma$. For $n = k+1$ we have that $\psi \in \Delta_k$ and $\C \psi \in \Delta_k$ by induction hypothesis. By Lemma \ref{l: properties of prime theories}, Item \ref{l: properties of prime theories, Item 3}. $\K_\agi \C \psi \in \Delta_k$ for all $\agi \in \A$ and therefore $\K_{\agi_k} \C \psi \in \Delta_k$. Hence, $\C \psi \in \Delta_{k+1}$, implying that also $\psi \in \Delta_{k+1}$. Therefore $\psi \in \Delta$ for any $\Delta \in R^{(\Sigma, \hickp)^*}(\Gamma)$. By induction hypothesis $\fw^{(\Sigma, \hickp)}, \Delta \models \psi$ for any such $\Delta$ and thus $\fw^{(\Sigma, \hickp)}, \Gamma \models \C \psi$. 
     
     For the right-to-left direction suppose that $\fw^{(\Sigma, \hickp)}, \Gamma \models \C \psi$. For any $\Delta \in R^{(\Sigma, \hickp)^*}(\Gamma)$ thus holds that $\fw^{(\Sigma, \hickp)}, \Delta \models \psi$. By induction hypothesis $\psi \in \Delta$. In order to show that $\Gamma \vdash_\mathrm{P} \C \psi$, we claim that $\vdash_\mathrm{P} \gamma(\Gamma) \rightarrow \C \gamma(\Gamma)$. By the presence of the rule $\mathsf{Ind}$, it suffices to prove that $\vdash_\mathrm{P} \gamma(\Gamma) \rightarrow \E \gamma(\Gamma)$. We first show that $\vdash_\mathrm{P} \gamma(\Gamma) \rightarrow \K_\agi \gamma(\Gamma)$ for $\agi \in \A$. To that end note that for any $\Delta \in R^{(\Sigma, \hickp)^*}(\Gamma)$,
        \begin{equation*}
            \K_\agi^{-1} \Delta \vdash_\mathrm{P} \gamma(\Gamma).
        \end{equation*}
        In fact if $ \K_\agi^{-1} \Delta \not \vdash_\mathrm{P} \gamma(\Gamma)$ for some $\Delta \in R^{(\Sigma, \hickp)^*}(\Gamma)$, then by the Lindenbaum Lemma there exists a prime theory $\Omega$ such that $\K_\agi^{-1}\Delta \subseteq \Omega$ and $\Omega \not \vdash_\mathrm{P} \gamma(\Gamma)$. By construction $\Delta \Rel_\agi^{(\Sigma, \hickp)} \Omega$, implying that $\Omega \in R^{(\Sigma, \hickp)^*}(\Gamma)$. By Lemma \ref{l: properties characteristic formula IM}, $\Omega \vdash_\mathrm{P} \gamma(\Gamma)$ which is a contradiction. Thus $\K_\agi^{-1} \Delta \vdash_\mathrm{P} \gamma(\Gamma)$. By Corollary \ref{c: deduction theorem} and $\mathsf{Nec_\agi}$ we obtain
        \begin{equation*}
            \vdash_\mathrm{P} \K_\agi (\bigwedge \K_\agi^{-1} \Delta \rightarrow \gamma(\Gamma)).
        \end{equation*}
    Using the $\mathsf{K_\agi}$-axiom and $\mathsf{MP}$ yields
    \begin{equation*}
        \vdash_\mathrm{P} \K_\agi \bigwedge \K_\agi^{-1} \Delta \rightarrow \K_\agi \gamma(\Gamma).
    \end{equation*}
    By (\ref{e: truthlemma3}),
    \begin{equation*}
        \K_\agi \K_\agi^{-1} \Delta \vdash_\mathrm{P} \K_\agi \gamma(\Gamma).
    \end{equation*}
    Hence, by Lemma \ref{l: assumption weakening} we obtain $\Delta \vdash_\mathrm{P} \K_\agi \gamma(\Gamma)$. By Corollary \ref{c: deduction theorem} $\vdash_\mathrm{P} \chi(\Delta) \rightarrow \K_\agi \gamma(\Gamma)$. As $\Delta$ was arbitrary, we obtain $\vdash_\mathrm{P} \gamma(\Gamma) \rightarrow \K_\agi \gamma(\Gamma)$ by repeatedly using the (correct substitution instance of the) intuitionistic tautology $(p \rightarrow r) \rightarrow ((q \rightarrow r) \rightarrow ((p \vee q) \rightarrow r)$. Now since $\agi \in \A$ was arbitrary we obtain $\vdash_\mathrm{P} \gamma(\Gamma) \rightarrow \E \gamma(\Gamma)$ by repeatedly using the (correct substitution instance of the) intuitionistic tautology $(p \rightarrow q) \rightarrow ((p \rightarrow r) \rightarrow (p \rightarrow (q \wedge r)))$. Thus applying $\mathsf{Ind}$ to $\gamma(\Gamma) \rightarrow \E \gamma(\Gamma)$ yields $\vdash_\mathrm{P} \gamma(\Gamma) \rightarrow \C \gamma(\Gamma)$. \smallskip
    
    \noindent By Lemma \ref{l: properties characteristic formula IM}, $\Gamma \vdash_\mathrm{P} \gamma(\Gamma)$ and so we obtain
    \begin{equation*}
        \Gamma \vdash_\mathrm{P} \C \gamma(\Gamma).
    \end{equation*}
   Since $\psi \in \Delta$ for every $\Delta \in R^{(\Sigma, \hickp)^*}(\Gamma)$, Lemma \ref{l: properties characteristic formula IM} implies $\vdash_\mathrm{P} \gamma(\Gamma) \rightarrow \psi$. Applying $\mathsf{Mon}$ gives
   \begin{equation*}
       \vdash_\mathrm{P} \C \gamma(\Gamma) \rightarrow \C \psi.
   \end{equation*}
   Thus applying $\mathsf{MP}$ yields
   \begin{equation*}
       \Gamma \vdash_\mathrm{P} \C \psi
   \end{equation*}
   and hence $\C \psi \in \Gamma$. \smallskip

   \noindent \textsc{Case for $\hicks$.} The left-to-right direction is similar to the case for $\hickp \in \{\hick, \hickt\}$. The right-to-left direction is very similar, but uses the same construction as used in the case for $\K_\agi$, subcase $\hicks$, in this lemma. Suppose $\fw^{(\Sigma, \hicks)}, \Gamma \models \C \psi$. For any $\Delta \in R^{(\Sigma, \hicks)^*}(\Gamma)$ thus holds that $\fw^{(\Sigma, \hicks)}, \Delta \models \psi$. By induction hypothesis $\psi \in \Delta$. In order to show that $\Gamma \vdash_\hicks \C \psi$ we claim that $\vdash_\hicks \gamma(\Gamma) \rightarrow \C \gamma(\Gamma)$. By the presence of the rule $\mathsf{Ind}$ it suffices to prove that $\vdash_\hicks \gamma(\Gamma) \rightarrow \E \gamma(\Gamma)$. We first show that $\vdash_\hicks \gamma(\Gamma) \rightarrow \K_\agi \gamma(\Gamma)$ for $\agi \in \A$. To that end note that for any $\Delta \in R^{(\Sigma, \hicks)^*}(\Gamma)$,
 \begin{equation*}
     \K_\agi \K_\agi^{-1} \Delta \vdash_\hicks \gamma(\Gamma).
 \end{equation*}
 In fact if $\K_\agi \K_\agi^{-1} \Delta \not \vdash_\hicks \gamma(\Gamma)$ for some $\Delta \in R^{(\Sigma, \hicks)^*}(\Gamma)$, then by the Lindenbaum Lemma there exists a prime theory $\Omega$ such that $\K_\agi \K_\agi^{-1} \Delta \subseteq \Omega$ and $\Omega \not \vdash_\hicks \gamma(\Gamma)$. By construction $\Omega \in R^{(\Sigma, \hicks)^*}(\Gamma)$, contradicting Lemma \ref{l: properties characteristic formula IM}. Hence $ \K_\agi \K_\agi^{-1} \Delta \vdash_\hicks \gamma(\Gamma)$. Using the same argument as in the case for $\K_\agi$ starting from (\ref{e: completeness S4 K1}) we obtain
 \begin{equation*}\label{e: completeness S4 C1}
     \Delta \vdash_\hicks \K_\agi \gamma(\Gamma).
 \end{equation*}
 Corollary \ref{c: deduction theorem} yields $\vdash_\hicks \chi(\Delta) \rightarrow \K_\agi \gamma(\Gamma)$ and so $\vdash_\hicks \gamma(\Gamma) \rightarrow \K_\agi \gamma(\Gamma)$. We obtain $\vdash_\hicks \gamma(\Gamma) \rightarrow \E \gamma(\Gamma)$ by the same argument as in the previous case and so $\vdash_\hicks \gamma(\Gamma) \rightarrow \C \gamma(\Gamma)$ by $\mathsf{Ind}$. From this we deduce $\Gamma \vdash_\hicks \C \psi$ by the same argument as before. \smallskip

 \noindent \textsc{Case for $\hickss$.} The left-to-right direction is identical to the previous cases. For the right-to-left direction suppose $\fw^{(\Sigma, \hickss)}, \Gamma \models \C \psi$. For any $\Delta \in R^{(\Sigma, \hickss)^*}(\Gamma)$ thus holds that $\fw^{(\Sigma, \hickss)}, \Delta \models \psi$. By induction hypothesis $\psi \in \Delta$. In order to show that $\Gamma \vdash_\hickss \C \psi$ we claim that $\vdash_\hickss \gamma(\Gamma) \rightarrow \C \gamma(\Gamma)$. By the presence of the rule $\mathsf{Ind}$ it suffices to prove that $\vdash_\hickss \gamma(\Gamma) \rightarrow \E \gamma(\Gamma)$. We first show that $\vdash_\hickss \gamma(\Gamma) \rightarrow \K_\agi \gamma(\Gamma)$ for $\agi \in \A$. To that end note that for any $\Delta \in R^{(\Sigma, \hickss)^*}(\Gamma)$,
 \begin{equation*}
     \K_\agi \K_\agi^{-1} \Delta, \neg \K_\agi \neg \K_\agi^{-1} \vdash \gamma(\Gamma)
 \end{equation*}
 as otherwise there would exists a $\K$-maximally consistent prime theory $\Omega$ with $\K_\agi \K_\agi^{-1} \Delta \cup \neg \K_\agi \neg \K_\agi^{-1} \Delta \subseteq \Omega$ and $\Omega \not \vdash_\hickss \gamma(\Gamma)$. Note that $\Delta \Rel_\agi^{(\Sigma, \hickss)} \Omega$ and so we obtain that $\Omega \in R^{(\Sigma, \hickss)^*}(\Gamma)$, contradicting Lemma \ref{l: properties characteristic formula IM}. We obtain
 \begin{equation*}
     \Delta \vdash \K_\agi \gamma(\Gamma)
 \end{equation*}
 by the same arguments as used for the corresponding case starting from (\ref{e: completeness S5 K1}). From this we deduce $\Gamma \vdash_\hickss \C \psi$ as before.
\end{proof}


Observe that all axioms and rules of our axiomatizations are actively used in the previous proofs. In particular for the S5 case, axiom $\mathsf{K_\agi^c}$ is used in the proof of the Lindenbaum Lemma to show that the constructed prime theory is $K$-maximally consistent and the axiom $\mathsf{S5_\agi}$ is used in the proof of the Truth Lemma. Completeness now readily follows.

\begin{theorem}[Completeness]\label{t: completeness of Hilbert style system} The following statements hold.
    \begin{enumerate}
        \item $\hick$ is complete with respect to the class of epistemic frames: if $\varphi$ is valid over the class of epistemic frames, then $\vdash_\hick \varphi$.
        \item $\hickt$ is complete with respect to the class of reflexive frames: if $\varphi$ is valid over the class of reflexive frames, then $\vdash_\hickt \varphi$.
        \item $\hicks$ is complete with respect to the class of S4 frames: if $\varphi$ is valid over the class of S4 frames, then $\vdash_\hicks \varphi$.
        \item$\hickss$ is complete with respect to the class of S5 frames: if $\varphi$ is valid over the class of S5 frames, then $\vdash_\hickss \varphi$.
    \end{enumerate}
\end{theorem}
\begin{proof}
    Let $\hickp \in \{\hick, \hickt, \hicks, \hickss \}$ and $\Sigma = \Cl(\varphi)$ ($\Sigma = \Cl^\neg(\varphi)$ if $\hickp = \hickss$). We proceed by contraposition. Suppose $\not \vdash_\mathrm{P} \varphi$. By the Lindenbaum Lemma there exists a $\Sigma$-prime theory $\Gamma$ relative to $\hickp$ (where $\Gamma$ is additionaly $\K$-maximally consistent if $\hickp = \hickss$), such that $\Gamma \not \vdash_\mathrm{P} \varphi$. By the Truth Lemma, $\varphi$ is falsified at world $\Gamma$ of the canonical model relative to $\Sigma$ and $\mathrm{P}$. Since the canonical model relative to $\Sigma$ and $\hickp$ is an epistemic/reflexive/S4/S5 model depending on $\hickp$, the statements of the theorem hold.
\end{proof}

As a corollary we obtain that all four logics have the finite model property, since the canonical models are finite. As usual, the finite model property combined with a finite sound and complete axiomatization yields that the validity problem of the logic (see Table \ref{tab:validity problem}) is decidable.

\begin{table}[t]
    \centering
    \begin{tabular}{|l l|}
    \hline
        \textbf{Input:} &  A formula $\varphi \in \LICK$ and a logic $\mathbf{L} \in \{\mathbf{ICK}, \mathbf{ICKT}, \mathbf{ICKS4}, \mathbf{ICKS5}\}$. \\
         \textbf{Question:} &  Is $\varphi \in \mathbf{L}$? \\
         \hline
    \end{tabular}
    \caption{The validity problem}
    \label{tab:validity problem}
\end{table}

\begin{theorem}[Finite model property and decidability]\label{t: fmp and decidability for K and T}
    The logics $\mathbf{ICK}$, $\mathbf{ICKT}$, $\mathbf{ICKS4}$ and $\mathbf{ICKS5}$ have the finite model property. Their validity problems are therefore decidable.
\end{theorem}

Moreover, from soundness, completeness and the semantic disjunction property (c.f. Lemma~\ref{l: semantic disjunction property}) we obtain the \emph{syntactic disjunction property}.

\begin{theorem}[Syntactic Disjunction Property]
    Let $\hickp \in \{\hick, \hickt, \hicks\}$ and $\varphi,\psi$ be formulae. If $\vdash_\hickp \varphi \vee \psi$, then $\vdash_\hickp \varphi$ or $\vdash_\hickp \psi$.
\end{theorem}

\section{A Modal Variant of Kuroda's Translation}\label{s: translation}

 Glivenko's Theorem~\cite{Glivenko_1929} shows that classical propositional logic ($\mathsf{CPC}$) can be embedded into $\mathsf{IPC}$ via a \emph{double-negation translation}, which maps every propositional formula $\varphi$ onto $\neg \neg \varphi$. The formula $\neg \neg \varphi$ is classically equivalent to $\varphi$ (but not intuitionistically) and furthermore satisfies the property that $\varphi$ is classically valid if and only if $\neg \neg \varphi$ is intuitionistically valid. A similar translation embedding classical first-order predicate logic into intuitionistic first-order predicate logic is Kuroda's translation~\cite{kuroda_1951}, which prefixes every formula with $\neg \neg$ and additional adds $\neg \neg$ after every universal quantifier. This section shows that classical common knowledge logic ($\mathsf{CK}$) over S5 can be embedded into $\mathsf{ICK}$ over S5. This will be achieved by constructing a modal variant of Kuroda's translation, which assigns to each formula $\varphi$ of $\LICK$ a formula $tr(\varphi)$ by adding $\neg \neg$ in front of $\varphi$ and after every knowledge and common knowledge operator. We will then show that for any formula $\varphi$,
    \begin{enumerate}
        \item $\varphi$ is classically equivalent to its translation $tr(\varphi)$ and
        \item  $\varphi$ is classically valid if and only $tr(\varphi)$ is intuitionistically valid,
    \end{enumerate}
    which entails that $\mathsf{CK}$ over S5 can be regarded as a fragment of $\mathsf{ICK}$ over S5. Intuitively, our translation preserves validity due to the following reasons. First of all, since $\mathbf{ICKS5}$ has the finite model property, formulae are valid over the class of S5 frames if and only if they are valid over the class of finite S5 frames. In every finite S5 frame, every world $w$ lies below a \emph{maximal world} $v$ with respect to $\leq$, meaning that $w \leq v$ and there exists no world $u$ with $v < u$. Observe that in maximal worlds, implications are evaluated classically and so such worlds (when ignoring the modalities) are classical worlds. The double negation $\neg \neg$ which our translation places in front of every formula then `pushes' formulae to the maximal worlds of the model. More precisely $w \models \neg \neg \varphi$ if and only if $v \models \varphi$ for any maximal world above $w$. The formula $\varphi$ can then be evaluated in the (maximal) classical worlds. However, a maximal world might be $R_\agi$-related to a world `inside' the model. Therefore after every modality we need to place another double negation to `push' the formula back to a maximal world.

    We begin by introducing the logic $\mathbf{CKS5}$. Then we formally introduce our translation and prove Item 1. above. Importantly, note that every double-negated formula $\neg\neg \varphi$ is classically equivalent to $\varphi$, whence Item 1. is easily inferred. For 2. we make the above intuition precise. We first show how to obtain a classical epistemic model from an intuitionistic one and, vice versa, how to obtain an intuitionistic epistemic model from a classical one. In the latter case we show that both models satisfy the same formulae, while in the former we show that the classical model satisfies the translated formula $tr(\varphi)$ in a world $w$ if and only if the intuitionistic model satisfies $\varphi$ in the corresponding world. Together this will yield Item 2. Our proof makes essential use of the S5 frame conditions and can therefore not be directly generalized to the other considered logics. It is unclear to us whether $\mathsf{CK}$ over the classes of all frames, reflexive frames or S4 frames can be embedded into its intuitionistic counterpart and we leave this question as an open problem. 
    
\subsection{Classical Common Knowledge}

The \emph{language} of $\mathsf{CK}$ is $\mathcal{L}_{\ICK}$, where formulae are defined as for $\mathsf{ICK}$. Formulae are evaluated on classical epistemic models. In order to distinguish these models from the epistemic models used for $\mathsf{ICK}$, we will refer to the former as classical models and to the latter as intuitionistic models.
    
\begin{definition}
A \emph{classical S5 frame} is a tuple $\mathcal{F} = (\W, \rel)$ where
        \begin{itemize}
            \item $\W$ is a non-empty set of \emph{worlds};
            \item $\rel = \{ R_\agi \mid \agi \in \A\}$ where $R_\agi \subseteq \W \times \W$ is an equivalence relation for each $\agi \in \A$;
        \end{itemize}
A \emph{classical S5 model} is a tuple $\fw =(\mathcal{F}, \V)$ where $\mathcal{F}=(\W, \rel)$ is a classical frame and $\V: \W \longrightarrow \mathcal{P}(\Prop)$ is a valuation function. The model  $\fw =(\mathcal{F}, \V)$ is said to be \emph{based} on $\mathcal{F}$.
\end{definition}

In difference to intuitionistic models, classical models do not feature the intuitionistic order $\leq$. Formulae are evaluated on classical S5 models as follows. As before we let $R \coloneqq \bigcup_{R_\agi \in \rel} R_\agi$ and $R^*$ be the reflexive transitive closure of $R$.

\begin{definition}
    The truth relation $\models_c$ between worlds $w$ of a classical S5 model $\fw$ and formulae $\varphi, \psi$ is defined inductively as follows.
\begin{center}
    \begin{tabular}{l l l}
        $\fw, w \not \models_c \bot$ &&\\
        $\fw, w \models_c p$ & iff & $p \in \V(w)$ \\
        $\fw, w \models_c \varphi \wedge \psi$ & iff & $\fw, w \models_c \varphi$ and $\fw, w \models_c \psi$\\
        $\fw, w \models_c \varphi \vee \psi$ & iff & $\fw, w \models_c \varphi$ or $\fw, w \models_c \psi$ \\
        $\fw, w \models_c \varphi \rightarrow \psi$ & iff & $\fw, w \not \models_c \varphi$ or $\fw, w \models_c \psi$ \\
        $\fw, w \models_c \K_\agi \varphi$ & iff & for all $v \in R_\agi(w)$, $\fw, v \models_c \varphi$ \\
        $\fw, w \models_c \C \varphi$ & iff & for all $v \in R^*(w)$, $\fw, v \models_c \varphi$.
    \end{tabular}
\end{center}
A formula $\varphi$ is called \emph{true at world $w$} if $\fw, w \models \varphi$ . Furthermore, $\varphi$ is called \emph{true in $\fw$}, written $\fw \models \varphi$, if $\fw, w \models \varphi$ for all $w \in \W$. Let $\fclass$ be the class of classical S5 frames. Then $\varphi$ is called \emph{satisfiable} over $\fclass$ if there exists a classical model $\fw$ based on a frame in $\fclass$ and a world $w$ such that $\fw, w \models \varphi$; otherwise $\varphi$ is called \emph{unsatisfiable} over $\fclass$. A formula $\varphi$ is called \emph{valid} over $\fclass$ if $\fw \models \varphi$ for all classical models $\fw$ based on a $\fclass$-frame; otherwise $\varphi$ is called \emph{invalid} over $\fclass$.
\end{definition}

It is easily checked that $\fw, w \models \neg \varphi$ where $\neg \varphi = \varphi \rightarrow \bot$ if and only if $\fw, w \not \models \varphi$ and therefore that $\varphi$ is equivalent to $\neg \neg \varphi$.

\begin{definition}
    The logic $\mathbf{CKS5}$ is defined to be the set of valid $\LICK$-formulae over the class of classical S5 frames.
\end{definition}

\subsection{The Translation}
    
The following definition introduces our modal variant $tr$ of Kuroda's translation which prefixes every formula with $\neg\neg$ and adds $\neg\neg$ also after every modality.

\begin{definition}
    Define the function $\tau: \mathcal{L}_{\ICK} \longrightarrow \mathcal{L}_{\ICK}$ by induction on $\varphi$ as follows.
    \begin{align*}
        \tau(\bot) & \coloneqq \bot  & \tau(p) & \coloneqq p \text{ for } p \in \Prop\\
        \tau(\varphi \wedge \psi) & \coloneqq \tau (\varphi) \wedge \tau(\psi) & \tau(\varphi \vee \psi) & \coloneqq \tau (\varphi) \vee \tau(\psi)\\
        \tau(\varphi \rightarrow \psi) & \coloneqq \tau (\varphi) \rightarrow \tau(\psi) & \tau(\C \varphi) & \coloneqq \C \neg \neg \tau(\varphi) \\
        \tau(\K_\agi \varphi) & \coloneqq \K_\agi \neg \neg \tau(\varphi) \text{ for } \agi \in \A 
    \end{align*}

Then define $tr: \mathcal{L}_{\ICK} \longrightarrow \mathcal{L}_{\ICK}$ by $tr(\varphi) \coloneqq  \neg \neg \tau(\varphi).$
\end{definition}

In the following we show that $tr$ satisfies Properties 1. and 2. stated above.

\begin{lemma}\label{l: tr is classicaly equivalent}
    For any $\mathcal{L}_\ICK$-formula $\varphi$ the following holds:
    \begin{center}
        $\varphi  \leftrightarrow tr(\varphi) \in \mathbf{CKS5}$.
    \end{center} 
\end{lemma}
\begin{proof}
     We first prove by induction on the structure of $\varphi$ that for any classical S5 model $\fw=(\W, \rel, \V)$ and any world $w \in \W$, $\fw,w \models_c \varphi$ if and only if $\fw,w \models_c \tau(\varphi)$. The base cases for $\varphi = \bot$ and $\varphi = p$ for $p \in \Prop$ are trivial. The cases for $\varphi = \psi \ast \chi$ where $\ast \in \{\wedge, \vee, \rightarrow\}$ follow immediately from the induction hypothesis. Suppose $\varphi = \K_\agi \psi$ for $\agi \in \A$. By definition $\tau(\varphi) = \K_\agi \neg \neg \tau(\psi)$. Then $\fw,w \models_c K_\agi \psi$ if and only if $M,v \models_c \psi$ for all $v \in R_\agi(w)$ if and only if (by induction hypothesis) $\fw,v \models_c \tau(\psi)$ for all $v \in R_\agi(w)$ if and only if $\fw,v \models_c \neg \neg \tau(\psi)$ for all $v \in R_\agi(w)$ if and only if $\fw,w \models_c \K_\agi \neg \neg \tau(\varphi)$. The case for $\varphi = \C \psi$ is similar.

    By definition, $tr(\varphi) = \neg \neg \tau(\varphi)$. We obtain the following: $\fw,w \models_c \varphi$ if and only if $\fw,w \models_c \tau(\varphi)$ if and only if $\fw,w \models_c \neg \neg \tau(\varphi)$. Therefore $\varphi \leftrightarrow tr(\varphi) \in \mathbf{CKS5}$.
\end{proof}

Recall that by Theorem  \ref{t: fmp and decidability for K and T} the logic $\mathbf{ICKS5}$ has the finite model property. Therefore to check whether a formula is valid over the class $\mathscr{F}$ of intuitionistic S5 frames it suffices to consider the subclass of $\mathscr{F}$ which consists of all finite frames in $\mathscr{F}$. We will now introduce two simple model constructions needed to prove that $tr$ is a translation. First, we define \emph{reduced models} which are classical S5 models obtained from restricting finite intuitionistic S5 models to their maximal (classical) worlds. Let ${\mathcal{F}=(\W, \leq, \rel)}$ be a finite intuitionistic S5 frame. Then $w \in \W$ is called a \emph{maximal world} if $w \leq v$ implies $w = v$ for all $v \in \W$. We write $\W^{max}$ for the set of maximal worlds of $\fw$.

\begin{definition}
    Let $\fw=(\W, \leq, \rel, \V)$ be an intuitionistic S5 model based on a finite frame. The \emph{$\fw$-reduced model} $\fw^r =(\W^r, \rel^r, \V^r)$ is defined as follows.
    \begin{itemize}
        \item $\W^r \coloneqq \W^{max}$
        \item $\rel^r \coloneqq \{R_\agi^r \mid R_\agi \in \rel\}$ where for each $\agi \in \A$: $R_\agi^r \coloneqq R_\agi \cap (\W^r \times \W^r)$
       
        \item $\V^r \coloneqq V \cap (\W^r \times \mathcal{P}(\Prop))$
    \end{itemize}
\end{definition}

\begin{lemma}
    Let $\fw=(\W, \leq, \rel, \V)$ be a finite intuitionistic S5 model. The $\fw$-reduced model is a classical S5 model.
\end{lemma}
\begin{proof}
    Note that since $\fw$ is based on a finite frame, maximal worlds exist which implies that $\W^r \not = \emptyset$. It is then immediate to check that $R_\agi^r$ for $\agi \in \A$ and $\V^r$ are well-defined. Since $R_\agi^r$ is simply the restriction of $R_\agi$ to $\W^r$, it further follows that $R_\agi^r$ is reflexive, transitive and symmetric since $R_\agi$ is.
\end{proof}

\begin{proposition}\label{l: tau_0 preserves truth}
     Let $\fw=(\W, \leq, \rel, \V)$ be a finite intuitionistic S5 model, let $\fw^r =(\W^r, \rel^r, \V^r)$ be the $\fw$-reduced model and let $w \in  \W^r$. For any $\mathcal{L}_\ICK$-formula $\varphi$ the following holds:
     \begin{center}
         $\fw,w \models \tau(\varphi)$ if and only if $\fw^r, w \models_c \varphi$. 
     \end{center} 
\end{proposition}
\begin{proof}
    We proceed by induction on the structure of $\varphi$. The base cases where $\varphi = \bot$ and $\varphi = p$ for $p \in \Prop$ are trivial. The cases where $\varphi = \psi \ast \gamma$ for $\ast \in \{\wedge, \vee\}$ follow immediately from the induction hypothesis. Suppose $\varphi = \psi \rightarrow \gamma$. Then $\tau(\varphi) = \tau(\psi) \rightarrow \tau(\gamma)$. Since $w$ is a maximal world, observe that $\fw,w \models \tau(\psi) \rightarrow \tau(\gamma)$ if and only if $\fw,w \not \models \tau(\psi)$ or $\fw,w \models \tau(\gamma)$. Therefore the claim follows immediately from the induction hypothesis.\smallskip 
    
    \noindent \textsc{Case for $\varphi = \K_\agi \psi$.} By definition, $\tau(\varphi) = \K_\agi \neg \neg \tau(\psi)$. Suppose first that $\fw^r, w \models_c \K_\agi \psi$. Then for all $v \in R_\agi^r(w)$, $\fw^r,v \models_c \psi$. Suppose $w \mathrel{R_\agi} u$. For any maximal world $v \in u^\uparrow$, we have $v \in R_\agi(u)$ by reflexivity and triangle confluence and so $v \in R_\agi(w)$ by transitivity. Therefore $v \in R_\agi^r(w)$, implying that $\fw^r, v \models_c \psi$. By induction hypothesis, $\fw,v \models \tau(\psi)$ and so $\fw,u \models \neg \neg \tau(\psi)$, which implies that $\fw,w \models \K_\agi \neg \neg \tau(\psi)$. 
    
    For the other direction suppose that $\fw, w \models \K_\agi \neg \neg \tau(\psi)$. Then for all $v \in R_\agi(w)$, $\fw,v \models \neg \neg \tau(\psi)$. Let $v \in \R_\agi^r(w)$. Then $v \in R_\agi(w)$ and so ${\fw, v \models \neg \neg \tau(\psi)}$. Since $v$ is maximal it follows that $\fw,v \models \tau(\psi)$. By induction hypothesis $\fw^r, v \models_c \psi$. Hence $\fw^r, w \models_c \K_\agi \psi$.\smallskip

    \noindent \textsc{Case for $\varphi = \C \psi$.}  By definition, $\tau(\varphi) = \C \neg \neg \tau(\psi)$. Let $w \in \W$, $u \in R^*(w)$ and let $u_0, \ldots, u_n \in \W$ with $u_0 = w$, $u_n = u$ and for each $ i \in [n]$ there is $\agi_i \in \A$ with $u_i \mathrel{R_{\agi_i}} u_{i+1}$. We first show by induction on $n$ that for any maximal world $v \in u^\uparrow$ there are $v_0, \ldots, v_n \in \W^r$ such that $v_0 = w$ and for all $i \in [n+1]$, $u_i \leq v_i$ and if $i \in [n]$ and $u_i \mathrel{R_{\agi_i}} u_{i+1}$, then $v_i \mathrel{R_{\agi_i}^r} v_{i+1}$. For $n=0$ let $v_0 = w$. For $n>0$ the induction hypothesis yields that there are $v_0, \ldots, v_{n-1} \in \W^r$ with $v_0 = w$, for all $ i \in [n]$, $u_i \leq v_i$ and if $i \in [n-1]$ and $u_i \mathrel{R_{\agi_i}} u_{i+1}$, then $v_i \mathrel{R_{\agi_i}^r} v_{i+1}$. Let $v_n \in \W$ be a maximal world with $u_n \leq v_n$. By reflexivity $v_n \mathrel{R_{\agi_{n-1}}} v_n$. By triangle confluence $u_n \mathrel{R_{\agi_{n-1}}} v_n$ and so by transitivity $u_{n-1} \mathrel{R_{\agi_{n-1}}} v_n$. Since $u_{n-1} \leq v_{n-1}$, by reflexivity and triangle confluence $u_{n-1} \mathrel{R_{\agi_{n-1}}} v_{n-1}$. By symmetry and transitivity $v_{n-1} \mathrel{R_{\agi_{n-1}}} v_n$. Hence, since $v_{n-1}, v_n \in \W^r$, we have $v_{n-1} \mathrel{R_{\agi_{n-1}}^r} v_n$.

    Now suppose that $\fw^r, w \models_c \C \psi$. Then for all $v \in (R^r)^*(w)$, ${\fw^r,v \models_c \psi}$. Let $u \in R^*(w)$. By the proof above, for any maximal world $v \in u^\uparrow$ we have $v \in (R^r)^*(w)$ and so $\fw^r, v \models_c \psi$ by assumption. By induction hypothesis ${\fw,v \models \tau(\psi)}$. Therefore $\fw,u \models \neg \neg \tau(\psi)$, implying that $\fw,w \models \C \neg \neg \tau(\psi)$. 

    For the other direction suppose that $\fw,w \models \C \neg \neg \tau(\varphi)$. Then for all $u \in R^*(w)$, $\fw,u \models \neg \neg \tau (\psi)$. Let $v \in (R^r)^*(w)$. Then $v \in R^*(w)$ and so $\fw,v \models \neg \neg \tau(\psi)$. Since $v$ is a maximal world, $\fw,v \models \tau(\psi)$. By induction hypothesis $\fw^r, v \models_c \psi$. Therefore $\fw^r, w \models_c \C \psi$.    
\end{proof}

Next, we consider the inverse construction and introduce \emph{induced models}, which are intuitionistic S5  models obtained from classical S5 models by equipping them with an intuitionistic order.

\begin{definition}
    Let $\fw=(\W, \rel, \V)$ be a classical S5 model. The \emph{$\fw$-induced model} $\fw^i =(\W^i, \leq^i, \rel^i, \V^i)$ is defined as follows.
    \begin{itemize}
        \item $\W^i \coloneqq \W$
        \item $\leq^i \coloneqq \{(w,w) \mid w \in \W^i\}$
        \item $\rel^i \coloneqq \{R_\agi^i \mid R_\agi \in \rel\}$ where for each $\agi \in \A$: $R_\agi^i \coloneqq R_\agi$
        \item $\V^i \coloneqq \V$ 
    \end{itemize}
\end{definition}

\begin{lemma}
    Let $\fw$ be a classical S5 model. The $\fw$-induced model is an intuititionistic S5 model.
\end{lemma}

\begin{proof}
    Note that $\leq^i$ is a partial order on $\W^i$, $R_\agi^i$ is triangle confluent for each $\agi \in \A$ and $\V^i$ is monotone in $\leq^i$. Moreover, clearly each $R_\agi^i$ is an equivalence relation.
\end{proof}

\begin{lemma}\label{l: classical model induces intuitionistic model}
    If $\fw = (\W, \rel, \V) $ is a classical S5 model and $w \in \W$, then for any formula $\varphi \in \LICK$, $\fw, w \models_c \varphi$ if and only if $\fw^i, w \models \varphi$.
\end{lemma}
\begin{proof}
    By definition of $\leq$, $\fw^i, w \models \varphi \rightarrow \psi$ if and only if $\fw^i, w \not \models \varphi$ or $\fw^i, w \models \psi$.
\end{proof}

Finally, we are able to prove the main result:

\begin{theorem}[Kuroda's translation]\label{t: negative translation}
    For any $\LICK$-formula $\varphi$ the following hold.
    \begin{enumerate}
        \item $\varphi \leftrightarrow tr(\varphi) \in \mathbf{CKS5}$.
        \item $\varphi \in \mathbf{CKS5}$ if and only if $tr(\varphi) \in \mathbf{ICKS5}$.
    \end{enumerate}
\end{theorem}
\begin{proof}
     1. is proven in Lemma \ref{l: tr is classicaly equivalent}. For 2. let $\varphi \in \mathcal{L}_\ICK$ be any formula. First suppose that $tr(\varphi) \in \mathbf{ICKS5}$. Then $tr(\varphi)$ is valid over the class of intuitionistic S5 frames. Therefore, $tr(\varphi)$ is valid over the class of classical S5 frames by Lemma \ref{l: classical model induces intuitionistic model} and so $\varphi \in \mathbf{CKS5}$ by 1. For the other direction suppose that $\varphi \in \mathbf{CKS5}$. Let $\fw$ be an arbitrary finite intuitionistic S5 model and let $w$ be an arbitrary world. Let $\fw^r$ be the $\fw$-reduced model. For any maximal world $u \in w^\uparrow$ we have by assumption that $\fw^r, u \models_c \varphi$. By Lemma~\ref{l: tau_0 preserves truth}, $\fw, u \models \tau(\varphi)$. Therefore $\fw,w \models \neg \neg \tau(\varphi)$, i.e. $\fw,w \models tr(\varphi)$. Therefore $tr(\varphi)$ is valid over the class of finite intuitionistic S5 frames. Theorem~\ref{t: fmp and decidability for K and T} yields that $tr(\varphi) \in  \mathbf{ICKS5}$.
\end{proof}

\section{Cyclic Proof Systems}\label{s: cyclic proof systems}

This section introduces four sequent calculi for $\ICK$ with the aim to capture the validities of $\ICK$ over the classes of all epistemic frames, reflexive frames, S4 frames and S5 frames. In difference to the calculus presented in~\cite{jager-intuitionistic_2016} we use non-wellfounded and cyclic sequent calculi to capture common knowledge as the greatest fixed point of the function $\mathsf{x} \mapsto \varphi \wedge \E \mathsf{x}$. \emph{Non-wellfounded proofs} deviate from standard sequent calculus proofs by allowing proof trees to contain infinitely long branches. As such they can be regarded as a formalizations of proofs by infinite descent, which are generally accepted to be equally strong as inductive proofs~\cite{brotherston}. To distinguish valid from invalid infinitary reasoning, infinite branches are required to satisfy a \emph{global correctness criterion}, which captures \emph{good formula traces} and roughly states that infinite branches are generated by infinitely often unfolding a given common knowledge formula on the right side of the sequent. A major advantage of non-wellfounded proofs over finite inductive proofs is that the former allows to retain analyticity of the rules: every rule satisfies the property that all formulae occurring in the premises belong to the closure of the formulae in the conclusion. This property is typically lost when using induction rules, due to the need to guess the right inductive hypothesis. Thus sequent calculi with induction rules often contain either the cut rule or the induction rule itself is non-analytic. For example, the calculus presented in~\cite{jager-intuitionistic_2016} uses unrestricted applications of cut in the completeness proof. 

In order to retain finite proofs, we will also consider \emph{cyclic proofs}, which are essentially finite representations of non-wellfounded proofs. Cyclic proofs are finite trees containing back-edges or cycles that point from some of their leaves to internal nodes. The basic idea is that such cyclic proofs can be unfolded over its cycles back into non-wellfounded proofs. To guarantee that the correctness condition on infinite branches is retained, cyclic proofs use a finite analogue for their cyclic branches: every branch leading to a cyclic leaf must satisfy a \emph{local correctness criterion}. Cyclic proofs thus combine the advantages of both approaches: finite proofs and analytic rules. As such they are amenable for automated proof-search~\cite{rooduijn_analytic_2022} and computing interpolants~\cite{Afshari_lyndon_2022, Shamkanov_circular_2014}.

We begin by introducing the four sequent calculi for $\ICK$ and its variations. Each calculus shares the same set of basic rules for the propositional connectives and axioms, as well as \emph{unfolding} rules for common knowledge. Additionally, each calculus contains one or several rules for the knowledge operators, which differ depending on which frame condition ought to be captured. We then define the notion of non-wellfounded and of cyclic proof for each calculus. The cyclic calculus for $\ICK$ over the class of epistemic models is denoted $\cick$ and is the extension of a standard multi-conclusion sequent calculus for $\mathsf{IPL}$ with rules for the modalities and common knowledge. The cyclic calculi $\cickt$, $\cicks$ and $\cickss$ for $\ICK$ over the classes of reflexive, S4 and S5 frames, respectively, are modular extensions of $\cick$. Soundness is established via an indirect argument. Completeness for the non-wellfounded calculi is proven via proof-search: given a sequent $\sigma$, we show how to define a \emph{proof-search tree} which represents a systematic search for a proof of $\sigma$. This proof-search tree forms the arena of a two-player game, played between \emph{Prover} and \emph{Refuter}. Intuitively, Prover tries to show that $\sigma$ has a proof and Refuter tries to show that $\sigma$ has a countermodel. We will make this intuition precise by showing that from winning strategies of Prover we can find a subtree of the proof-search tree that is a proof of $\sigma$ and from winning strategies of Refuter we can find a subtree from which a countermodel of $\sigma$ can be constructed. Since the game is \emph{determined}, meaning that exactly one of the two players has a winning strategy, completeness follows. Our construction is robust in the sense that it suffices to adapt Prover's turns to obtain countermodels satisfying different frame conditions and so a uniform argument for completeness can be given. We then also obtain completeness of the cyclic calculi by showing how to translate non-wellfounded into cyclic proofs. Our calculi and soundness and completeness proofs are based on the joint work with Afshari, Grotenhuis and Leigh~\cite{afshari_intuitionistic_2024} about intuitionistic master modality, which is a unimodal version of $\ICK$. There we developed our modular approach towards proof-search and used it to establish Lemma~\ref{l: functional and triangle frames} by showing how Prover's turn can be adapted to obtain functional and triangle confluent countermodels. Here, we illustrate how to obtain countermodels satisfying different frame conditions.

The only case for which our method does not work is when dealing with S5 frame conditions. To handle those the cyclic calculus $\cickss$ requires the cut rule, which obstructs our proof-search argument. However, we can instead prove completeness for the cyclic calculus directly via a canonical model construction which restricts applications of cut to \emph{analytic cuts} (meaning that only formulae of the negation closure of the root sequent can function as cut formulae). Thus we retain analyticity. The combination of cyclic proofs with analytic cuts was first explored in a joint publication with Rooduijn~\cite{rooduijn_analytic_2022} for classical S5 common knowledge logic, from which many ideas here are drawn.

\subsection{Basic Definitions}

In order to capture \emph{good formula traces} in cyclic proofs our calculi employ a simple type of \emph{formula annotation}. An \emph{annotated formula} is a tuple $(\varphi, \an)$, written $\varphi^\an$, where $\varphi \in \LICK$ and $\an \in \{\u, \f\}$. The annotation $\u$ designates that the formula is \emph{unfocused} and $\f$ that the formula is \emph{in focus}. When depicting annotated formulae we will usually omit brackets, e.g. the formula $\varphi \rightarrow \psi^\u$ should be read as $(\varphi \rightarrow \psi)^\u$. Formulae without annotation will be called \emph{plain}. Finite sets of annotated formulae are denoted by $\Gamma, \Delta, \Sigma, \Pi$ and $\Omega$, using subscripts or superscripts when necessary. For a set of annotated formulae $\Gamma$ define
\[
    \Gamma^- \coloneqq \{\varphi \mid \varphi^\an \in \Gamma\}
    \text{ and }
    \Gamma^\u \coloneqq \{\varphi^\u \mid \varphi^\an \in \Gamma\}.
\]
The use of annotations will become clear in the soundness proof below. 

\begin{definition}
    A \emph{sequent}\index{sequent!for $\ICK$} is an ordered pair $\langle \Gamma, \Delta \rangle$ of finite sets of annotated formulae, written as $\Gamma \Rightarrow \Delta$, such that
    \begin{enumerate}
    \item Every formula in $\Gamma$ is unfocused.
    \item At most one formula in $\Delta$ is in focus.
    \item If a formula $\varphi$ is in focus, then $\varphi = \C \psi$ or $\varphi = \K_\agi \C \psi$ for some formula $\psi$ and $\agi \in \A$.
\end{enumerate}
\end{definition}

We denote sequents by $\sigma$ and write $\Gamma_\sigma$ and $\Delta_\sigma$ for the left and right side of $\sigma$, respectively. The \emph{interpretaion} of a sequent $\sigma$ is the formula $\sigma^I \coloneqq \bigwedge \Gamma_\sigma^- \rightarrow \bigvee \Delta_\sigma^-$. Given an epistemic model $\fw$ and a world $w$, we write $\fw, w \models \sigma$ if $\fw, w \models \sigma^I$. A sequent $\sigma$ is valid over a class $\mathscr{F}$ of frames if $\sigma^I$ is valid and falsifiable otherwise. The notions of satisfiability and unsatisfiability are defined similarly. Observe that the annotations convey no semantic meaning. The \emph{closure} $\Cl(\sigma)$ of a sequent $\sigma$ is defined as $\Cl(\sigma) \coloneqq \Cl(\Gamma_\sigma) \cup \Cl(\Delta_\sigma)$ and the \emph{negation closure} as $\Cl^\neg (\sigma) \coloneqq \Cl^\neg(\Gamma_\sigma) \cup \Cl^\neg (\Delta_\sigma)$.

\begin{table}[t]
    \centering
        \begin{tabular}{|c c|}
        \hline
        & \\
         $\infer[\mathsf{id}]{\Gamma, \varphi^\u \Rightarrow \varphi^\an, \Delta}{}$ 
         & $\infer[\bot]{\Gamma, \bot^\u \Rightarrow \Delta}{}$\\
         & \\
         $\infer[\wedge \mathsf{L}]{\Gamma, \varphi \wedge \psi^\u \Rightarrow \Delta}{\Gamma, \varphi^\u, \psi^\u \Rightarrow \Delta}$ 
         & $\infer[\wedge \mathsf{R}]{\Gamma\Rightarrow \varphi \wedge \psi^\u ,\Delta}{\Gamma \Rightarrow \varphi^\u, \Delta & \Gamma \Rightarrow \psi^\u, \Delta}$\\
         & \\
         $\infer[\vee \mathsf{L}]{\Gamma, \varphi \vee \psi^\u\Rightarrow \Delta}{\Gamma, \varphi^\u \Rightarrow \Delta & \Gamma, \psi^\u \Rightarrow \Delta}$ & $\infer[\vee \mathsf{R}]{\Gamma \Rightarrow \varphi \vee \psi^\u, \Delta}{\Gamma \Rightarrow \varphi^\u, \psi^\u, \Delta}$\\
         & \\
         $\infer[{\to} \mathsf{L}]{\Gamma, \varphi \rightarrow \psi^\u \Rightarrow \Delta}{\Gamma \Rightarrow \varphi^\u, \Delta & \Gamma, \psi^\u \Rightarrow \Delta}$ 
         & $\infer[{\to} \mathsf{R}]{\Gamma \Rightarrow \varphi \rightarrow \psi^\u, \Delta}{\Gamma, \varphi^\u \Rightarrow \psi^\u}$\\
         & \\
         $\infer[\mathsf{u}]{\Gamma \Rightarrow \varphi^\f, \Delta}{\Gamma \Rightarrow \varphi^\u, \Delta}$ & $\infer[\mathsf{f}]{\Gamma \Rightarrow \varphi^\u, \Delta}{\Gamma \Rightarrow \varphi^\f, \Delta}$\\
         & \\
         \hline
        \end{tabular}
    \caption{The basic rules. The symbols $\Gamma$ and $\Delta$ range over finite sets of annotated formulae which may be empty and $\an \in \{\u, \f\}$.}
    \label{d: basic rules}
\end{table}

The basic rules our calculi use are depicted in Table~\ref{d: basic rules}. The rules for $\wedge, \vee, \rightarrow$ as well as $\mathsf{id}$ and $\bot$ form an annotated version of a standard multi-conclusion sequent calculus for intuitionistic propositional logic taken from~\cite{Negri_2001}. The annotations do not play a role for these rules, since only formulae $\C \psi$ or $\K_\agi \C \psi$ can be in focus. The rules $\mathsf{u}$ and $\mathsf{f}$ are called the \emph{focus rules} and are used to change the focus annotation of a formula. Table \ref{d: additional rules for ICK} contains rules for the modalities. The rule $\mathsf{\K_\agi}$ is a standard modal rule, while the rules $\mathsf{T_\agi}$, $\mathsf{S4_\agi}$ and $\mathsf{S5_\agi}$ capture the frame conditions reflexivity, transitivity and symmetry, respectively (see e.g.~\cite{rooduijn_analytic_2022}). The rule $\mathsf{\K_\agi {\rightarrow}}$ expresses that implications with a boxed antecedent behave classically and is a special case of the classical right-implication rule. The rules $ \mathsf{\C L}$ and $ \mathsf{\C R}$ are \emph{unfolding rules} and replace a common knowledge formula $\C \varphi$ - when read bottom-up - with its unfolding $\varphi \wedge \E \C \varphi$. Note that the annotations play a role in the rules $ \mathsf{\C R}$, $\mathsf{\K_\agi}$, $\mathsf{S4_\agi}$ and $\mathsf{S5_\agi}$. Finally there is the cut rule. All rules apart from $\mathsf{u}$ and $\mathsf{f}$ are called \emph{logical rules}.

For every rule in Table \ref{d: basic rules} as well the rules $\mathsf{CL}$, $\mathsf{CR}$ and $\mathsf{\K_\agi{\rightarrow}}$, the distinguished formula in the conclusion is called \emph{principal} and the distinguished formulae in the premises are its \emph{residuals}. The rule $\mathsf{cut}$ has no principal formulae, while both distinguished formulae in the premises are residual. These distinguished formulae are called the \emph{cut formulae}. For the knowledge rules for agent $\agi$ (i.e. $\mathsf{\K_\agi}$, $\mathsf{T_\agi}$, $\mathsf{S4_\agi}$ and $\mathsf{S5_\agi}$) every formula in the conclusion is principal and every formula in the premise is the residual of the corresponding principal formula. Formulae in $\Sigma, \Pi$ have no residuals. In every rule instance, any formula that is neither principal nor residual is called a \emph{side formula}.

\begin{example}
    For an instance of ${\to}\mathsf{L}$ of the form
    \begin{equation*}
        \infer[{\to} \mathsf{L}]{\Gamma, \varphi \rightarrow \psi^\u\Rightarrow \Delta}{\Gamma \Rightarrow \varphi^\u, \Delta & \Gamma, \psi^\u \Rightarrow \Delta} 
    \end{equation*}
    the principal formula is $\varphi\to \psi^\u$ in the conclusion, its residuals are  $\varphi^\u$ in the left premise and $ \psi^\u$ in the right premise. The formulae in $\Gamma$ and $\Delta$ (both in the conclusion and the premises) are side formulae. For an instance of $\mathsf{\K_\agi}$ of the form
    \begin{equation*}
        \infer[\mathsf{\K_\agi}]{\Pi, \K_\agi \psi_1^\u, \ldots, \K_\agi \psi_k^\u \Rightarrow \K_\agi \varphi^\f, \Sigma}{\psi_1^\u, \ldots , \psi_k^\u \Rightarrow \varphi^\f}
    \end{equation*}
    all formulae in the conclusion are principal. The formulae in $\Pi$ and $\Sigma$ have no residuals. The residual of $\K_\agi \psi_\agi^\u$ for $1 \leq i \leq k$ is the formula $\psi_\agi^\u$ in the premise and the residual of $\K_\agi \varphi^\f$ is the formula $\varphi^\f$ in the premise. There are no side formulae.
\end{example}

\begin{table}[t]
    \centering
        \begin{tabular}{|c c|}
        \hline
        & \\
        $\infer[\mathsf{\K_\agi}]{\Pi, \K_\agi \Gamma \Rightarrow \K_\agi \varphi^\an, \Sigma}{\Gamma \Rightarrow \varphi^\an}$ & $\infer[\mathsf{cut}]{\Gamma \Rightarrow \Delta}{\Gamma, \varphi^\u \Rightarrow \Delta & \Gamma \Rightarrow \varphi^\u, \Delta}$\\
         & \\
        $\infer[\mathsf{T_\agi}]{\Gamma, \K_\agi \varphi^\u \Rightarrow \Delta}{\Gamma, \varphi^\u \Rightarrow \Delta}$ &  $\infer[\mathsf{S4_\agi}]{\Pi, \K_\agi \Gamma \Rightarrow \K_\agi \varphi^\an, \Sigma}{ \K_\agi \Gamma \Rightarrow \varphi^\an}$\\
        & \\
        $\infer[\mathsf{S5_\agi}]{\Pi, \K_\agi \Gamma \Rightarrow \K_\agi \varphi^\an, \K_\agi \Delta, \Sigma}{ \K_\agi \Gamma \Rightarrow \varphi^\an, \K_\agi \Delta}$ &  $\infer[\mathsf{\K_\agi{\rightarrow}}]{\Gamma \Rightarrow \K_\agi \varphi \rightarrow \psi^\u, \Delta}{\Gamma, \K_\agi \varphi^\u \Rightarrow \psi^\u, \Delta}$\\
        & \\
        $\infer[ \mathsf{\C L}]{\Gamma, \C \varphi^\u \Rightarrow \Delta}{\Gamma, \varphi^\u, \{\K_\agi \C \varphi^\u\}_{\agi \in \A} \Rightarrow \Delta}$ & $\infer[ \mathsf{\C R}]{\Gamma \Rightarrow \C \varphi^\an,\Delta}{\Gamma \Rightarrow \varphi^\u, \Delta & \{\Gamma \Rightarrow \K_\agi \C \varphi^\an, \Delta\}_{\agi \in \A}}$\\
        & \\
        \hline
        \end{tabular}
    \caption{The additional rules. The symbols $\Gamma, \Delta, \Sigma$ and $\Pi$ range over finite sets of annotated formulae which may be empty, $\agi \in \A$ and $\an \in \{\u,\f\}$.}
    \label{d: additional rules for ICK}
\end{table}

We call a rule $\mathsf{r}$ of the form
 \begin{equation*}
     \infer[\mathsf{r}]{\sigma}{\sigma_1 & \ldots & \sigma_n}
 \end{equation*}
\emph{analytic} if for all $i \in [n+1]$, $\Cl(\sigma_i) \subseteq \Cl(\sigma)$. For the S5 case we instead use the negation closure of the sequents for analyticity. Note that all rules in Table~\ref{d: basic rules} and Table~\ref{d: additional rules for ICK} with the exception of $\mathsf{cut}$ are analytic. 

The condition that sequents have at most one formula in focus imposes restrictions on rule applications. For example the focus rule $\f$ can only be applied (bottom-up) to a sequent $\Gamma \Rightarrow \varphi^\u, \Delta$ if $\Delta = \Delta^\u$. Additionally $\varphi$ must be of the right shape, i.e. $\varphi = \C \psi$ or $\varphi = \K_\agi \C \psi$ for $\agi \in \A$ and $\psi \in \LICK$. The following lemma illustrates another restriction.

\begin{lemma}\label{l: proper application of rules}
    If in an instance of $\mathsf{\C R}$ the principal formula is in focus, then the left premise has no formula in focus.
\end{lemma}
\begin{proof}
    Suppose towards contradiction that there is an application of $\mathsf{\C R}$ in which the principal formula is in focus and the left premise has a formula in focus, too.
    \begin{prooftree}
        \AxiomC{$\Gamma \Rightarrow \varphi^\u, \Delta$}
        \AxiomC{$\Gamma \Rightarrow \K_\agi \C \varphi^\f, \Delta$}
        \RightLabel{$\mathsf{\C R}$}
        \BinaryInfC{$\Gamma \Rightarrow \C \varphi^\f, \Delta$}
    \end{prooftree}
    Since $\varphi^\u$ is by definition not in focus, the formula in focus must occur in $\Delta$. Since the right premise cannot contain two formulae in focus, the formula in focus in $\Delta$ must be $\K_\agi \C \varphi^\f$. But then the conclusion sequent contains two formulae in focus, a contradiction.
\end{proof}

We will usually construct proofs bottom-up. That is if $\sigma$ is a sequent and $\mathsf{r}$ a rule, then `$\mathsf{r}$ is applied to $\sigma$' means that we consider a rule instance of $\mathsf{r}$ with conclusion $\sigma$.

\subsection{Non-wellfounded proofs}

In order to obtain completeness for our cyclic calculi, we will make a detour to non-wellfounded calculi, which are introduced in this section. The completeness of the cyclic calculus for S5 will be proven directly, whence we do not introduce a non-wellfounded calculus for that case.

\begin{definition}
    Consider the rules depicted in Table \ref{d: basic rules} and Table \ref{d: additional rules for ICK}. 
    \begin{enumerate}
        \item The calculus $\nick$ consists of the basic rules from Table \ref{d: basic rules} as well as the rules $\mathsf{\K_\agi}$ for $\agi \in \A$, $\mathsf{\C L}$ and $ \mathsf{\C R}$.
        \item The calculus $\nickt$ is the extension of $\nick$ with the rules $\mathsf{T_\agi}$ for $\agi \in \A$.
        \item The calculus $\nicks$ is the calculus  $\nickt$ with the rules $\mathsf{\K_\agi}$ replaced by the rules $\mathsf{S4_\agi}$ for $\agi \in \A$.
    \end{enumerate}
\end{definition}

Let $\hickp \in \{\nick, \nickt, \nicks\}$. A \emph{non-wellfounded pre-proof} of a sequent $\sigma$ in $\hickp$ is a finite or countably infinite tree $\pi$ whose nodes are labeled by sequents according to the rules of $\hickp$\footnote{By this we mean that whenever $u$ is a parent node in $\pi$ with $v_1, \ldots, v_k$ its children, then the sequents labeling $u, v_1, \ldots, v_k$ form an instance of a rule of $\mathsf{P}$.} with the root labeled by $\sigma$. We call $\pi$ a $\hickp$-pre-proof. Note that pre-proofs are finite branching, since each rule has finitely many premises. Therefore, by K\H{o}nig's Lemma, infinite pre-proofs contain infinite branches. 

A \emph{path} through $\pi$ is a finite or infinite sequence $\rho = \rho(0), \rho(1), \ldots$ of nodes of $\pi$ such that for every $i < \omega$ the node $\rho(i+1)$ is a child of $\rho(i)$ (if $\rho(i), \rho(i+1)$ exist). A \emph{branch} is a maximal path, i.e. a branch is a path starting at the root which either ends in a leaf or is infinite. If there is no danger of confusion, we will tacitly identify a node in a pre-proof with the sequent labeling it and thereby paths and branches of a pre-proof with sequences of sequents. In order for a pre-proof to be a \emph{proof} it must satisfy a \emph{global correctness criterion} on its infinite branches: every infinite branch must contain a \emph{good suffix}. Given a branch $\rho$ a suffix $\rho'$ of $\rho$ is called \emph{good} if every sequent in $\rho'$ has a formula in focus and $\rho'$ passes through infinitely many applications of $\mathsf{\C R}$ where the principal formula is in focus.

\begin{definition}
    An \emph{non-wellfounded proof} in $\hickp$ of a sequent $\sigma$ is a non-wellfounded $\hickp$-pre-proof $\pi$ of $\sigma$, such that every leaf of $\pi$ is labeled by an axiom and every infinite branch $\rho$ of $\pi$ has a good suffix $\rho'$. In this case we write $\vdash_\hickp \sigma$.
\end{definition}

Readers familiar with non-wellfounded proof theory might notice that this is a rather unusual definition for a non-wellfounded proof. The definition does not refer to \emph{formula traces} and instead uses the formula annotations to capture good $\nu$-traces. It should be noted that using the annotations is not necessary for obtaining a sound notion of non-wellfounded proof. One could instead follow the more traditional way of using formula traces to characterize the good branches in a non-wellfounded pre-proof. The reason behind these design choices is that the non-wellfounded proofs will only serve as a tool for proving completeness for cyclic proofs, and our definition simplifies the argument showing that non-wellfounded proofs can be translated into cyclic proofs.

An illustration of a non-wellfounded proof is provided in Example \ref{e: non-wellfounded and cyclic proof for IM} below. The following lemma establishes a straightforward yet crucial property of good suffixes.

\begin{lemma}\label{lemma good suffix contains inf many box appl}
    Let $\pi$ be a $\hickp$-pre-proof, $\rho$ an infinite branch of $\pi$ and $\rho'$ a suffix of $\rho$. If $\rho'$ is good, then it contains infinitely many applications of the rule $\mathsf{\K_\agi}$ (or, $\mathsf{S4_\agi}$, if $\hickp = \nicks$) for some $\agi \in \A$.
\end{lemma}
\begin{proof}
    We only consider the case for $\mathsf{K_\agi}$, as the case for $\mathsf{S4_\agi}$ is argued in the same way. Let $\rho$ be an infinite branch of a $\hickp$-pre-proof $\pi$ and let $\rho'$ be a good suffix of $\rho$. Suppose towards contradiction that $\rho'$ contains only finitely many applications of $\mathsf{\K_\agi}$ for any $\agi \in \A$. Let $\rho''$ be the suffix of $\rho'$ after the last application of $\mathsf{\K_\agi}$ for any $\agi \in \A$. Since $\rho'$ is good, every sequent in $\rho''$ has a formula in focus. This implies that the focus rules cannot be applied in $\rho''$. Furthermore, $\rho''$ contains infinitely many applications of $\mathsf{\C R}$ where the principal formula is in focus. Let $n$ be the least natural number such that $\rho''(n)$ is the conclusion of an application of $\mathsf{\C R}$ with the principal formula in focus. By Lemma \ref{l: proper application of rules}, $\rho''(n+1)$ must be the right premise. Thus the formula in focus is of the form $\K_\agi \C \varphi^\f$ for some formula $\varphi$ and some $\agi \in \A$. By inspection of the rules of $\hickp$, the only rules that can be applied to $\K_\agi \C \varphi^\f$ are $\mathsf{\K_\agi}$ and $\mathsf{u}$. Since there are no applications of $\mathsf{\K_\agi}$ or $\mathsf{u}$ in $\rho''$, the formula in focus must be a side formula in any rule application after $\rho''(n+1)$, implying that $\rho''$ passes only through finitely many applications of $\mathsf{\C R}$ where the principal formula is in focus; a contradiction.
\end{proof}
%
\subsection{Cyclic proofs}

This subsection introduces cyclic sequent calculi. Roughly, a cyclic proof is a finite representation of a non-wellfounded proof, where each infinite branch is pruned after encountering a suitable repetition, such that a non-wellfounded proof can be obtained by unravelling the cyclic proof over its repetitions. We introduce the following calculi.

\begin{definition}
    Consider the rules in Table \ref{d: basic rules} and Table \ref{d: additional rules for ICK}.
    \begin{enumerate}
        \item The calculi $\cick$, $\cickt$ and $\cicks$ consist of the same rules as their corresponding calculi $\nick$, $\nickt$ and $\nicks$.
        \item The calculus $\cickss$ is the calculus $\cicks$ with the rules $\mathsf{S4_\agi}$ replaced by the rules  $\mathsf{S5_\agi}$ for $\agi \in \A$ and, furthermore, extended by the rules ${\rightarrow}\mathsf{\K_\agi}$ for $\agi \in \A$ and by $\mathsf{cut}$
    \end{enumerate}
\end{definition}

 Let $\hickp \in \{\cick, \cickt, \cicks, \cickss\}$. A \emph{cyclic pre-proof} of a sequent $\sigma$ in $\hickp$ is a finite tree $\pi$ whose nodes are labeled by sequents according to the rules of $\hickp$ with the root labeled by $\sigma$. We call $\pi$ a cyclic $\hickp$-pre-proof, where we drop `$\hickp$' in case the specific calculus is not of relevance. Note that cyclic pre-proofs are simply non-wellfounded pre-proofs which are finite.

In order to distinguish cyclic pre-proofs from cyclic proofs, we require to formulate a \emph{local correctness criterion} for those branches in a cyclic proof that do not end in an axiom. This criterion is essentially a finite version of the global correctness criterion for non-wellfounded proofs.

\begin{definition}
    A path $\rho$ in a cyclic pre-proof is \emph{successful} if the following hold.
    \begin{enumerate}
        \item Every sequent in $\rho$ has a formula in focus.
        \item The path $\rho$ passes through at least one instance of $\mathsf{\C R}$ where the principal formula is in focus.
    \end{enumerate}
\end{definition}

Given a cyclic pre-proof $\pi$, a pair of nodes $(u,v)$ of $\pi$ is called a \emph{repetition} if $u \not = v$, there exists a path from $u$ to $v$ and both nodes are labeled by the same sequent. A repetition $(u,v)$ is \emph{successful}\index{successful repetition} if the path from $u$ to $v$ is successful.

\begin{definition}
   Let $\hickp \in \{\cick, \cickt, \cicks, \cickss\}$. A \emph{cyclic proof} in $\hickp$ of a sequent $\sigma$ is a cyclic $\hickp$-pre-proof $\pi$ of $\sigma$ such that every leaf $l$ of $\pi$ is either labeled by an axiom or there exists a node $c(l)$ in $\pi$ such that $(c(l),l)$ is a successful repetition. We write $\vdash_\hickp \sigma$ if $\sigma$ has a cyclic $\hickp$-proof.
\end{definition}

Given a cyclic proof $\pi$, leafs labeled by axioms are called \emph{axiomatic leafs} and all other leafs are called \emph{non-axiomatic leafs}. If $v$ is a non-axiomatic leaf and $(u,v)$ form a successful repetition, then $u$ is called the \emph{companion} of $v$. Note that $v$ may have multiple companions.

\begin{example}\label{e: non-wellfounded and cyclic proof for IM}
    Figure \ref{fig:example-proof} shows a cyclic proof (in any considered calculus) for the induction axiom $(\varphi \wedge \C(\varphi \rightarrow \K_\agi \varphi)) \rightarrow \C \varphi$ for common knowledge for the case where $\LICK$ contains only one agent. The case for multiple agents is a straightforward generalization. We write $\gamma$ for the formula $\varphi \rightarrow \K_\agi \varphi$. The left and middle leaves are axioms. The right leaf is non-axiomatic and forms a repetition with the node `below' indicated by the arrow. It is immediate to check that the path from the lower node to the right leaf is successful. A non-wellfounded proof may be obtained from the cyclic proof by `unfolding' it infinitely often over the successful repetition. Intuitively the unfolding is obtained by `gluing' the subtree rooted at the companion on top of the non-axiomatic leaf co-recursively. Note that the resulting tree contains one infinite branch and infinitely many finite branches. Since the repetition is successful, the suffix of the infinite branch which starts at the lowermost companion is good. 
\end{example}

\begin{figure}[t]

\fbox{
\begin{tikzpicture}
\node[] (a) [] {
\AxiomC{}
    \RightLabel{$\mathsf{id}$}
    \UnaryInfC{$\varphi^\u, \C \gamma^\u \Rightarrow \varphi^\u$}
    \AxiomC{}
    \RightLabel{$\mathsf{id}$}
    \UnaryInfC{$\varphi^\u, \K_\agi \C \gamma^\u \Rightarrow \varphi^\u, \K_\agi \C \varphi^\f$}
    \AxiomC{$\varphi^\u, \C \gamma^\u \Rightarrow \C \varphi^\f$}
    \RightLabel{$\mathsf{\K_\agi}$}
    \UnaryInfC{$\varphi^\u, \K_\agi \varphi^\u, \K_\agi \C \gamma^\u \Rightarrow \K_\agi \C \varphi^\f$}
    \RightLabel{${\to} \mathsf{L}$}
    \BinaryInfC{$\varphi^\u,\gamma^\u, \K_\agi \C \gamma^\u \Rightarrow \K_\agi \C \varphi^\f$}
    \RightLabel{$ \mathsf{\C L}$}
    \UnaryInfC{$\varphi^\u, \C \gamma^\u \Rightarrow \K_\agi \C \varphi^\f$}
    \RightLabel{$ \mathsf{\C R}$}
    \insertBetweenHyps{\hskip -24pt}
    \BinaryInfC{$\varphi^\u, \C\gamma^\u \Rightarrow \C \varphi^\f$}
    \RightLabel{$\mathsf{f}$}
    \UnaryInfC{$\varphi^\u, \C\gamma^\u \Rightarrow \C \varphi^\u$}
    \RightLabel{${\land}\mathsf{L}$}
    \UnaryInfC{$\varphi \wedge \C\gamma^\u \Rightarrow  \C \varphi^\u$}
    \RightLabel{${\to}\mathsf{R}$}
    \UnaryInfC{$\Rightarrow (\varphi \wedge \C\gamma) \rightarrow \C \varphi^\u$}
      \DisplayProof};

\draw[->,rounded corners=.25cm,dashed] (3,2.3) -- (3,2.6) -- (7.3,2.6) -- (7.3,-0.4) -- (-0.8,-0.4);
\end{tikzpicture}
}
\caption{A cyclic proof for the induction axiom.\label{fig:example-proof}}
\end{figure}


    

The defining feature of $\cickss$ is the rule $\mathsf{{\rightarrow}\K_\agi}$. It formalizes the observation from Lemma~\ref{l: classical truth conditions for implication} that implications of the form $\K_\agi \varphi \rightarrow \psi$ behave classically. Observe that the rule is only applicable if the principal formula is of the form $\K_\agi \varphi \rightarrow \psi$. It is therefore a special case of the classical right implication rule.

\begin{example}
    Recall that $\neg \varphi$ is a shortcut for $\varphi \rightarrow \bot$.  The sequent $\Rightarrow \K_\agi p \vee \neg \K_\agi p^\u$ has a proof in $\cickss$, depicted below.
    \begin{prooftree}
        \AxiomC{}
        \RightLabel{$\mathsf{id}$}
        \UnaryInfC{$\K_\agi p^\u \Rightarrow \K_\agi p^\u, \bot^\u$}
        \RightLabel{$\mathsf{{\rightarrow}\K_\agi}$}
        \UnaryInfC{$\Rightarrow \K_\agi p^\u, \K_\agi p \rightarrow \bot^\u$}
        \RightLabel{$\vee \mathsf{R}$}
        \UnaryInfC{$\Rightarrow \K_\agi p \vee (\K_\agi p \rightarrow \bot)^\u$}
    \end{prooftree}
    Note that with the standard intuitionistic right implication rule, this proof is not possible: if ${\rightarrow}\mathsf{R}$ is applied (bottom-up) to the sequent $\Rightarrow \K_\agi p^\u, \K_\agi p \rightarrow \bot^\u$, the premise is $\K_\agi p^\u \Rightarrow \bot^\u$, which is not provable.
\end{example}

 Recall that, like $ \K_\agi p \vee \neg \K_\agi p$, also $\C p  \vee \neg \C p$ is valid over S5 frames. However, the rule $\mathsf{{\rightarrow}\K_\agi}$ cannot be applied to formulae of the form $\C \varphi \rightarrow \psi$. The following example shows how to obtain a proof of $\Rightarrow \C p  \vee \neg \C p^\u$ using the cut rule.

\begin{example}\label{e: law of excluded middle for C}
   The sequent  $\Rightarrow \C p  \vee \neg \C p^\u$ has a proof in $\cickss$. For simplicity we assume that there is only one agent in the language, whose knowledge operator is $\K_\agi$. The proof is generalized in a straightforward way for multiple agents. The following is a cyclic proof of $\Rightarrow \C p  \vee \neg \C p^\u$,
    \begin{prooftree}
        \AxiomC{$\pi$}
        \AxiomC{$\tau$}
        \RightLabel{$ \mathsf{\C R}$}
        \BinaryInfC{$\Rightarrow \C p^\u, \neg \C p^\u$}
        \RightLabel{$\vee \mathsf{R}$}
        \UnaryInfC{$\Rightarrow \C p \vee \neg \C p^\u$}
    \end{prooftree}
    
    where $\pi$ is the following proof,
 
    \begin{prooftree}
    \small
        \AxiomC{}
        \RightLabel{$\mathsf{id}$}
        \UnaryInfC{$ p^\u, \K_\agi \C p^\u \Rightarrow \K_\agi \C p^\u, \bot^\u$}
        \RightLabel{$ \mathsf{\C L}$}
        \UnaryInfC{$\C p^\u \Rightarrow \K_\agi \C p^\u, \bot^\u$}
        \AxiomC{}
        \RightLabel{$\mathsf{id}$}
        \UnaryInfC{$\bot^\u, \C p^\u \Rightarrow \bot^\u$}
        \RightLabel{$\rightarrow \mathsf{L}$}
        \BinaryInfC{$\neg \K_\agi \C p^\u, \C p^\u \Rightarrow \bot^\u$}
        \RightLabel{$\rightarrow \mathsf{R}$}
        \UnaryInfC{$\neg \K_\agi \C p^\u \Rightarrow p^\u, \neg \C p^\u$}
        \AxiomC{}
        \RightLabel{$\mathsf{id}$}
        \UnaryInfC{$p^\u, \K_\agi \C p^\u \Rightarrow p^\u, \neg \C p^\u, \bot^\u$}
        \RightLabel{$\mathsf{\C L}$}
        \UnaryInfC{$\C p^\u \Rightarrow p^\u, \neg \C p^\u, \bot^\u$}
        \RightLabel{$\mathsf{T_\agi}$}
        \UnaryInfC{$\K_\agi \C p^\u \Rightarrow p^\u, \neg \C p^\u, \bot^\u$}
        \RightLabel{$\mathsf{{\rightarrow}\K_\agi}$}
        \UnaryInfC{$\Rightarrow p^\u, \neg \C p^\u, \neg \K_\agi \C p^\u $}
        \RightLabel{$\mathsf{cut}$}
        \BinaryInfC{$\Rightarrow p^\u, \neg \C p^\u$}
    \end{prooftree}

and $\tau$ is the following proof,

    \begin{prooftree}
    \small
        \AxiomC{}
        \RightLabel{$\mathsf{id}$}
        \UnaryInfC{$p^\u, \K_\agi \C p^\u \Rightarrow \K_\agi \C p^\u, \bot^\u$}
        \RightLabel{$\mathsf{\C L}$}
        \UnaryInfC{$\C p^\u \Rightarrow \K_\agi \C p^\u, \bot^\u$}
        \AxiomC{}
        \RightLabel{$\mathsf{id}$}
        \UnaryInfC{$\bot^\u, \C p^\u \Rightarrow \bot^\u$}
        \RightLabel{$\rightarrow \mathsf{L}$}
        \BinaryInfC{$\neg \K_\agi \C p^\u, \C p^\u \Rightarrow \bot^\u$}
        \RightLabel{$\rightarrow \mathsf{R}$}
        \UnaryInfC{$\neg \K_\agi \C p^\u \Rightarrow \K_\agi \C p^\u, \neg \C p^\u$}
        \AxiomC{}
        \RightLabel{$\mathsf{id}$}
        \UnaryInfC{$\K_\agi \C p^\u \Rightarrow \K_\agi \C p^\u, \neg \K_\agi p^\u$}
        \RightLabel{$\mathsf{{\rightarrow}\K_\agi}$}
        \UnaryInfC{$\Rightarrow \neg \K_\agi \C p^\u, \K_\agi \C p^\u, \neg \C p^\u$}
        \RightLabel{$\mathsf{cut}$}
        \BinaryInfC{$\Rightarrow\K_\agi \C p^\u, \neg \C p^\u$.}
    \end{prooftree}
\end{example}

Example \ref{e: law of excluded middle for C} hints towards the fact that $\cickss$ is not cut-free complete. This does not come as a surprise, given that it is well-known that classical S5 modal logic does not have a cut-free (plain) sequent calculus. A famous counterexample in the classical realm is the formula $p \rightarrow \Box \Diamond p$. In the intuitionistic setting, $\Box$ and $\Diamond$ are not interdefinable, therefore our language cannot express $\Diamond$. However, we may still use the formula as a counterexample, by identifying $\Box$ with $\K_\agi$ and $\Diamond$ with $\neg \K_\agi \neg$. First of all, an easy computation shows that $\neg \K_\agi \neg$ is not evaluated as a $\Diamond$, but instead its truth conditions are as follows:
\begin{center}
    \begin{tabular}{l l l}
    $\fw, w \models \neg \K_\agi \neg \varphi$ & iff & there exist $u,u' \in W$ with $u \in R_\agi(w)$, $u' \in u^\uparrow$\\
    & & and $\fw, u' \models \varphi$.\\
    \end{tabular}
\end{center}
Therefore the following holds.

\begin{lemma}
    The formula $p \rightarrow \K_\agi \neg \K_\agi \neg p$ is valid over the class of S5 frames.
\end{lemma}
\begin{proof}
    Let $\fw=(\W, \leq, \rel, \V)$ be an arbitrary model based on an S5 frame and $w \in \W$ an arbitrary world. Let $w \leq v$ and suppose that $\fw,v \models p$. Let $u \in R_\agi(v)$. We want to show that $\fw, u \models \neg \K_\agi \neg p$. By the previous observation it suffices to find worlds $s,t$ such that $s \in R_\agi(u)$, $t \in s^\uparrow$ and $\fw,t \models p$. By symmetry $v \in R_\agi(u)$. Since $v \in v^\uparrow$ and $\fw,v \models p$, choosing $s = t = v$ shows that $\fw, u \models \neg \K_\agi p$ and hence that $\fw,v \models \K_\agi \neg \K_\agi \neg p$. Hence $\fw,w \models p \rightarrow \K_\agi \neg \K_\agi \neg p$, implying that $p \rightarrow \K_\agi \neg \K_\agi \neg p$ is valid.
\end{proof}

\begin{proposition}\label{p: cicks5 not cutfree complete}
    $\CICK_\mathsf{S5}$ is not cut-free complete.
\end{proposition}
\begin{proof}[Proof sketch]
    Observe that the sequent  $\sigma = (p^\u \Rightarrow \K_\agi \neg \K_\agi \neg p^\u)$ contains no formula with a common knowledge operator. Therefore, the focus rules cannot be applied, implying that any (cut-free) proof of $\sigma$ must be a proof where every branch ends in an axiom. Applying rules bottom-up, by inspection of the rules, the only rule applicable to $\sigma$ is $\mathsf{S5_\agi}$. There are two possible such applications. The first application has the sequent $\Rightarrow \neg \K_\agi \neg p^\u$ as premise and the second has the sequent $\Rightarrow \neg \K_\agi \neg p^\u, \K_\agi \neg \K_\agi \neg p^\u$ as premise (this case happens if $\K_\agi \neg \K_\agi \neg p^\u$ in the conclusion is assumed to be contained in $\K_\agi \Delta$; see the definition of the rule $\mathsf{S5_\agi}$). It is routine to check that both sequents cannot be proven.\footnote{For the interested reader who is unwilling to try out the many possible and sensible candidates for a proof, let us remark here that both possible sequents in the premise are invalid. Thus the soundness result presented in the next section will provide an alternative and simpler proof.}\qedhere
\end{proof}

Note that the cut formulae used in the cyclic proof from Example \ref{e: law of excluded middle for C} belong to the negation closure of the root sequent. Similarly, the sequent $p^\u \Rightarrow K_\agi \neg K_\agi \neg p^\u$ is provable by using $\neg K_\agi \neg p^\u$ as cut formula, which also belongs to the negation closure of the root sequent. Following this observation we will show that $\cickss$ is complete when the cut-rule is restricted to \emph{analytic cuts}, i.e. instances of $\mathsf{cut}$ where the cut formula belongs to the negation closure of the root sequent. Therefore we will still obtain analytic completeness for $\cickss$.

\subsection{Soundness}\label{c: IM, section soundness of cim}

We now turn towards proving our cyclic calculi sound.\footnote{The non-wellfounded calculi are only used as a tool to establish completeness of the cyclic calculi. Still, let us remark here that they are also sound, which follows from soundness of the cyclic calculi and {Lemma~\ref{l: non-wellfounded proof implies cyclic proof}}.} Since cyclic proofs may contain non-axiomatic leafs, a more complex argument than an induction on the height of a proof is required to establish soundness.
Intuitively, the cyclic calculus is sound because success of repetitions ensures that along every repetition, \emph{progress} is made in the form of a modal step: as a result, when proving $\C\varphi$ we are essentially proving $\E^n\varphi$ for every $n<\omega$. To prove this formally, we argue by contradiction. We first assign to each invalid sequent with a focused formula a \emph{measure} in the form of a natural number. Then, assuming there is a cyclic proof $\pi$ of an invalid sequent, we show that there must be a successful repetition $(u,v)$ in $\pi$ of an invalid sequent whose measure strictly decreases along the path $(u,v)$ which contradicts the fact that measures are well-defined. The following proof is routine (c.f. Lemma \ref{l: axioms of ICK are valid}).

\begin{lemma}\label{l: local soundness}
 All rules $\mathsf{r}$ of the cyclic calculus $\hickp \in \{\cick, \cickt, \cicks, \cickss\}$ preserve validity (over the respective class of frames $\mathscr{F}$): if $\mathsf{r}$ is of the form
    \begin{equation*}
        \infer[\mathsf{r}]{\sigma}{\sigma_1 & \dotsm & \sigma_n}
    \end{equation*}
    and $\sigma$ is invalid over $\mathscr{F}$, then there is an $i \in [n+1]$ such that $\sigma_i$ is invalid over $\mathscr{F}$.
\end{lemma}

Let $\sigma$ be a sequent that has a formula in focus, i.e., $\Delta_\sigma$ contains a formula of the form $\K_\agi^j \C \varphi^\f$ for $j \in \{0,1\}$. Denote by $\sigma(n)$ the sequent $\Gamma_\sigma \Rightarrow \Delta_\sigma, \K_\agi^j \E^n \varphi^\u$, i.e., the sequent expanding the right side of $\sigma$ by the formula $\K_\agi^j \E^n \varphi^\u$. 
\begin{lemma}\label{l: invalid sequent with formula in focus}
    If $\sigma$ has a formula in focus and is invalid (over a class of frames $\mathscr{F}$), then there exists a natural number $n$, such that $\sigma(n)$ is invalid over $\mathscr{F}$.
\end{lemma}
\begin{proof}
Let $\mathscr{F}$ be one of the considered classes of frames and suppose $\sigma$ is an invalid sequent with a formula in focus. Then there exists a formula $\K_\agi^j \C \varphi^\f \in \Delta_\sigma$ for $j \in \{0,1\}$ and a pointed $\mathscr{F}$-model $(\fw,w)$ with $\fw,w  \models \bigwedge \Gamma_\sigma^-$ and $\fw,w \not \models \bigvee \Delta_\sigma^-$. So in particular $\fw,w \not \models \K_\agi^j \C \varphi$, implying that there exist worlds $u,v \in \W$ with $w \Rel_\agi^j v$ and $v \Rel^\ast u$ such that $\fw, u \not \models \varphi$. By definition there exists a natural number $n$ such that $v \Rel^n u$ and hence $\fw, w \not \models \K_\agi^j \E^n \varphi$. Thus $\sigma(n)$ is invalid over $\mathscr{F}$.
\end{proof}

As a consequence, every invalid sequent $\sigma$ (over a fixed class of frames $\mathscr{F}$) with a formula in focus can be associated a measure:
\begin{equation*}
    \mu(\sigma) \coloneqq \min \{n < \omega \mid \sigma(n) \text{ is invalid}\}.
\end{equation*}
We do not explicitly display the class of frames $\mathscr{F}$ in the notation of the measure. This will not cause any confusion, as the class of frames we work with will always be clear. We may now prove a strenghtening of Lemma \ref{l: local soundness}.

\begin{lemma}\label{l: global soundness}
    Suppose
    \begin{equation*}
        \infer[\mathsf{r}]{\sigma}{\sigma_1 & \dotsm & \sigma_n}
    \end{equation*}
    is an instance of a rule $\mathsf{r} \in \hickp$ for $\hickp \in \{\cick, \cickt, \cicks, \cickss\}$. If $\sigma$ is invalid (over the respective class of frames $\mathscr{F}$), then there is an $i \in [n+1]$ such that $\sigma_i$ is invalid. 
    If both $\sigma$ and $\sigma_i$ have a formula in focus, then, moreover,
    \[
        \mu(\sigma_i) \leq \mu(\sigma),
    \]
    where the inequality is strict if $\mathsf{r} = \mathsf{\C R}$ and the principal formula is in focus.
\end{lemma}
\begin{proof}
    By Lemma \ref{l: local soundness} it suffices to only consider the case where both the conclusion and at least one premise have a formula in focus. We first treat the case that the formula in focus is not principal. Then $\mathsf{r}\notin \{{\to}\R, \mathsf{\K_\agi}, \mathsf{S4_\agi}\}$, as this would contradict the existence of a premise with a focused formula. By inspection of the rules, note that then \emph{every} premise must have a formula in focus, and so the following is a correct rule instance of $\mathsf{r}$.
    \begin{equation*}
        \infer[\mathsf{r}]{\sigma(\mu(\sigma))}{\sigma_1(\mu(\sigma)) & \dotsm & \sigma_n(\mu(\sigma))}
    \end{equation*}
    By Lemma~\ref{l: local soundness}, since $\sigma(\mu(\sigma))$ is invalid, there exists a premise $\sigma_i(\mu(\sigma))$ that is invalid. Hence $\sigma_i$ is invalid and $\mu(\sigma_i) \leq \mu(\sigma)$. \smallskip
    
    Now suppose that the formula in focus is principal in $\mathsf{r}$. Then $\mathsf{r} \in \{\mathsf{\K_\agi}, \mathsf{S4_\agi}, \mathsf{S5_\agi}, \mathsf{\C\mathsf{R}}\}$. We check all cases. \smallskip

    \noindent \textsc{Case for $\mathsf{\K_\agi}$.} Suppose $\mathscr{F}$ is the class of epistemic frames and $\mathsf{r} = \mathsf{\K_\agi}$. Then the conclusion $\sigma$ is of the form $\Pi, \K_\agi \Gamma \Rightarrow \K_\agi \C \varphi^\f, \Sigma$ and the single premise $\sigma_1$ is of the form $\Gamma \Rightarrow \C \varphi^\f$.  By assumption there exists an epistemic model $\fw = (\W, \leq, \rel, \V)$ and a world $w \in \W$ with $\fw, w \not \models \sigma(\mu(\sigma))$. Thus, there exists $v \in w^\uparrow$ with $\fw, v  \models \bigwedge \K_\agi \Gamma^-$ and ${\fw, v \not \models \K_\agi \E^{\mu(\sigma)} \varphi}$. This implies that there exists a world $u \in R_\agi(v)$ with ${\fw, u \not \models \E^{\mu(\sigma)} \varphi}$. Note that $\fw, u \models \bigwedge \Gamma^-$. Therefore $(\fw, u)$ falsifies $\sigma_1(\mu(\sigma))$, implying that $\sigma_1$ is invalid and  $\mu(\sigma_1) \leq \mu(\sigma)$. Note that this also covers the case for the class of reflexive frames. \smallskip

    \noindent \textsc{Case for $\mathsf{S5_\agi}$.} Suppose $\mathscr{F}$ is the class of S5 frames and $\mathsf{r} = \mathsf{S5_\agi}$. Then the conclusion $\sigma$ is of the form:
			\[
			\Pi, \K_\agi \Gamma \Rightarrow \K_\agi \C \psi^\f, \K_\agi \Delta, \Sigma.
			\]
			Since $\sigma$ is invalid, there is an S5 model $\fw = (\W, \leq, \rel, \V)$ and a world $ w \in \W$  such that $\fw, w \not \models \sigma(n)$ where $n= \mu(\sigma)$. Thus there exists $w' \in w^\uparrow$ such that $\fw, w' \models \bigwedge \Gamma_\sigma^-$ and $\fw, w' \not \models \bigvee \Delta_\sigma^- \vee \K_\agi \E^n \psi$. In particular it holds that
			\[
			\fw, w' \not \models \K_\agi \E^{n} \psi.
			\]
			It follows that there is a world $v \in R_\agi(w')$ such that $\fw, v \not \models \E^n \psi$. Clearly this also means that $\fw, v \not \models \C \psi$.  We claim that, in fact,
			\[
			\fw,v \not \models (\K_\agi \Gamma \Rightarrow \C \psi^\f, \E^n \psi^u,  \K_\agi\Delta)^I,
			\]
			which implies that the premise $\sigma_1$ is invalid and that $\mu(\sigma_1) \leq \mu(\sigma)$.
			
			Since $R_\agi$ is transitive, it holds for all $\varphi$ with $\fw, w' \models \K_\agi \varphi$, that  $\fw, v \models \K_\agi \varphi$. Hence $M, v \models \K_\agi \varphi$ for each $\K_\agi \varphi^\u \in \K_\agi \Gamma$. Moreover, suppose that $\K_\agi \gamma^\u \in \K_\agi \Delta$. Then ${\fw, w' \not \models \K_\agi \gamma}$. Thus there is a state $u \in \W$ such that $w' \mathrel{R_\agi} u$ and $\fw, u \not \models \gamma$. By symmetry and transitivity, we get $v \mathrel{R_\agi} u$, whence $\fw, v \not \models \K_\agi \gamma$, as required. \smallskip

            \noindent \textsc{Case for $\mathsf{S4_\agi}$}. Suppose $\mathscr{F}$ is the class of all S4 frames and $\mathsf{r} = \mathsf{S4_\agi}$. Observe that $\mathsf{S4_\agi}$ is a special case of $\mathsf{S5_\agi}$, namely it consists of all instances of $\mathsf{S5_\agi}$ where $\K_\agi \Delta = \emptyset$. In the case above, the symmetry of $R_\agi$ is exclusively used to deal with formulae in $\K_\agi \Delta$. Hence, we obtain the result by following the previous case using a pointed S4 model instead of a pointed S5 model, and skipping the part dealing with $\K_\agi \Delta$. \smallskip

            \noindent \textsc{Case for $\mathsf{\C R}$.} Suppose $\mathscr{F}$ is the class of all epistemic frames and $\mathsf{r} = \mathsf{CR}$. Then the conclusion $\sigma$ is of the form $\Gamma \Rightarrow \C \varphi^\f, \Delta$ with premises $\sigma_0$ given by $\Gamma \Rightarrow \varphi^\u, \Delta$, and $\sigma_i$ for $1 \leq i \leq \lvert \A \rvert$ given by $\Gamma \Rightarrow \K_{\agi_i}\C \varphi^\f, \Delta$, respectively.  By assumption there exists an epistemic model $\fw = (\W, \leq, \rel, \V)$ and a world $w \in \W$ that falsifies $\sigma(\mu(\sigma))$. Hence there exists ${v \in w^\uparrow}$ such that $\fw,v\models \bigwedge \Gamma^-$ and $\fw,v\not\models \C \varphi \lor \E^{\mu(\sigma)} \varphi\lor \bigvee\Delta^{-}$. If $\mu(\sigma) = 0$, then ${\fw,v \not \models \varphi}$, so $(\fw,v)$ falsifies the left premise $\sigma_0$. By Lemma~\ref{l: proper application of rules}, $\sigma_0$ does not have a formula in focus, and so the statement of the lemma holds. If $\mu(\sigma) > 0$, then $\fw,v \not \models \K_{\agi}\E^{\mu(\sigma) - 1} \varphi$ for some $\agi \in \A$. Hence $(\fw,v)$ falsifies $\sigma_{i}(\mu(\sigma)-1)$. So $\sigma_i$ is invalid and we have  $\mu(\sigma_i) < \mu(\sigma)$. Note that this covers all classes of frames we consider.
\end{proof}

\begin{theorem}[Soundness]\label{t: CIM soundness}
   Let $\hickp \in \{\cick, \cickt, \cicks, \cickss\}$. If there is a $\hickp$-proof of a sequent $\sigma$, then $\sigma$ is valid over the class of frames associated with $\hickp$.
\end{theorem}
\begin{proof}
    Let $\pi$ be a $\hickp$-proof of $\sigma$ and suppose for contradiction that $\sigma$ is invalid over the respective class of frames $\mathscr{F}$. By repeatedly applying Lemma~\ref{l: global soundness} we obtain a path of invalid sequents
    \[
        \rho = \sigma_1,\sigma_2 \ldots, \sigma_n
    \]
    through $\pi$ such that $\sigma=\sigma_1$ and  $\sigma_n$ is a leaf. As $\sigma_n$ cannot be an axiom and $\pi$ is a proof, there exists $i \in [n]$ such that $(\sigma_i, \sigma_n)$ is a successful repetition. Then the path from $\sigma_i$ to $\sigma_n$ always has a formula in focus and passes through at least one instance of $\mathsf{\C R}$ in which the formula in focus is principal. Hence, by construction, we have $\mu(\sigma_n) < \mu(\sigma_i)$. Since $\sigma_n=\sigma_i$, this contradicts the fact that the measure $\mu$ is well-defined. 
\end{proof}

\subsection{Completeness}\label{c: IM, section completeness of the cyclic calculus cIM}

We now turn towards proving completeness of our calculi. For $\cick$, $\cickt$ and $\cicks$ we proceed in two steps. First, we set up a general framework for proving completeness via proof-search games in Subsection~\ref{ss: proof-search games}, from which we infer completeness with respect to non-wellfounded proofs in Subsection~\ref{ss: completeness for epistemic, reflexive and s4 frames}. Crucially, our framework is robust and can be successfully applied for all three systems, with only minor changes required. Second, we then show how to transform arbitrary non-wellfounded proofs into cyclic proofs in Subsection~\ref{ss: translating non-wellfounded into cyclic proofs}, whence obtaining completeness of our calculi with respect to cyclic proofs as well. For $\cickss$ however, the proof-search strategy does not seem to work due to the presence of the cut rule. Fortunately, using analytic cuts allows us to establish (analytic) completeness with respect to cyclic proofs directly via a canonical model construction, which will be done in Subsection~\ref{ss: completeness for s5 frames}.

\subsubsection{Proof-search games}\label{ss: proof-search games}

Each sequent $\sigma$ is associated with a \emph{proof-search tree} which forms the arena of a two-player game between \emph{Prover}, whose winning strategies establish proofs of $\sigma$, and \emph{Refuter}, whose winning strategies describe countermodels for $\sigma$. 
Completeness then becomes a corollary of determinacy of the game. In the following subsections we refer to the rules $\mathsf{{\rightarrow} R}$, $\mathsf{\K_\agi}$ and $\mathsf{S4_\agi}$ as \emph{non-invertible rules} and to all other rules of $\hickp \in \{\nick, \nickt, \nicks\}$ as \emph{invertible rules}.\footnote{This is a slight abuse of terminology, since $\mathsf{{\rightarrow} L}$ is not invertible. However, when building proof-search trees we will restrict rule instances of `invertible' rules to so-called \emph{preserving} applications (see below) and preserving applications of $\mathsf{{\rightarrow} L}$ are in fact invertible.} Moreover all rules apart from the focus rules are called \emph{logical rules}.

A proof-search tree for $\sigma$ is built by applying rules bottom-up to $\sigma$. The invertible rules are applied first until a \emph{saturated} sequent is obtained. 
\begin{definition}
    Let $\hickp \in \{\nick, \nickt, \nicks\}$. A sequent $\Gamma \Rightarrow \Delta$ is \emph{$\hickp$-saturated} if the following hold.
    \begin{enumerate}
        \item If $\varphi\wedge \psi^\u  \in \Gamma$, then $\varphi^\u  \in \Gamma$ and $\psi^\u  \in \Gamma$.
        \item If $\varphi \vee \psi^\u  \in \Gamma$, then $\varphi^\u  \in \Gamma$ or $\psi^\u  \in \Gamma$.
        \item If $\varphi \rightarrow \psi^\u  \in \Gamma$, then $\varphi^\u  \in \Delta$ or $\psi^\u  \in \Gamma$.
        \item If $\C \varphi^\u  \in \Gamma$, then $\varphi^\u  \in \Gamma$ and $\{\K_\agi \C \varphi^\u \mid \agi \in \A\} \subseteq \Gamma$.
        \item If $\varphi \wedge \psi^\u \in \Delta$, then $\varphi^\u \in \Delta$ or $\psi^\u \in \Delta$.
        \item If $\varphi \vee \psi^\u \in \Delta$, then $\varphi^\u \in \Delta$ and  $\psi^\u \in \Delta$.
        \item If $\C \varphi^\u \in \Delta$, then $\varphi^\u \in \Delta$ or $\K_\agi \C \varphi^\u \in \Delta$ for some $\agi \in \A$.
        \item $\C \varphi^\f \not \in \Delta$ for all formulae $\varphi$.
        \item If $\hickp \in \{\nickt, \nicks\}$ and $\K_\agi \varphi \in \Gamma$ for $\agi \in \A$, then $\varphi \in \Gamma$.
    \end{enumerate}
    Given a sequent $\sigma$, a formula occurring in $\sigma$ is said to be \emph{$\hickp$-saturated} if $\sigma$ satisfies the corresponding clause above for that formula. 
\end{definition}
If $\hickp$ is clear from the context or irrelevant to the argument, we will usually drop it and simply call sequents \emph{saturated}. As we are working with set sequents, formulae can simultaneously function as principal and as side formulae. We call an application of a rule \emph{preserving} if the principal formula(e) also occurs as a side formula. 

\begin{example}
    In the following two applications of $\wedge \mathsf{L}$ to the sequent $\varphi \wedge \psi^\u \Rightarrow \gamma^\u$, the left application is preserving but the right application is not:
\begin{center}
    \begin{tabular}{l l}
       $\infer[\wedge \mathsf{L}]{\varphi \wedge \psi^\u \Rightarrow \gamma^\u}{\varphi \wedge \psi^\u, \varphi^\u, \psi^u \Rightarrow \gamma^\u}$  &  $\infer[\wedge \mathsf{L}]{\varphi \wedge \psi^\u \Rightarrow \gamma^\u}{\varphi^\u, \psi^u \Rightarrow \gamma^\u}$\\
    \end{tabular}
\end{center}
\end{example}

When a rule is applied preservingly to a sequent $\sigma$, we ensure that the premise(s) preserve the information contained in $\sigma$ by preserving all formulae of $\sigma$. Note that applications of the non-invertible rules $\mathsf{\K_\agi}$, $\mathsf{S4_\agi}$ and ${\to}{\mathsf{R}}$ are never preserving. Moreover, applications of $\mathsf{\C \mathsf{R}}$ where the principal formula is in focus are not preserving either (since sequents can only have one formula in focus).

We will now give a general definition of proof-search trees, that will be employed in the following completeness arguments. The particular form of the proof-search tree depends on the kind of countermodel one wants to obtain from a winning strategy of Refuter. In general, proof-search trees are built by first applying invertible rules until a saturated sequent is encountered. Then a non-invertible rule must be applied. Since we do not know which non-invertible rule to apply to which formula, all possible non-invertible rule instances (as well as the focus rules) are merged into a new rule, called the \emph{choice rule}. Depending on the specific form of the choice rule, different kinds of countermodels can be read off from a failed proof-search and so each completeness proof we present will be relative to a suitable `choice' rule $\mathsf{C_n}$. Therefore, the following general definition of proof-search tree is given relative to an arbitrary choice rule.

\begin{definition}\label{d: proof-search tree nIM}
    Let $\hickp \in \{\cick, \cickt, \cicks\}$ and fix some inference rule $\mathsf{C_n}$ and a sequent $\sigma$. A \emph{$\hickp$-proof-search tree (with choice rule $\mathsf{C_n}$)} for $\sigma$ consists of a finite or countably infinite tree $\mathcal{T}$ whose nodes are labeled by sequents according to $\mathsf{C_n}$  and the invertible logical rules of $\hickp$ such that the following hold.
    \begin{enumerate}
        \item The root is labeled by $\Gamma_\sigma\Rightarrow \Delta_\sigma$.
        \item Every invertible logical rule is applied preservingly, with the exception of $\mathsf{\C \mathsf{R}}$ if the principal formula is in focus.
        \item No invertible logical rule is applied to a sequent in which the principal formula is already $\hickp$-saturated.
         \item A node is a leaf if and only if it is labeled by an axiom or by a $\hickp$-saturated sequent to which the $\mathsf{C_n}$-rule cannot be applied.
         \item The rule $\mathsf{C_n}$ is only applied to $\hickp$-saturated sequents.
    \end{enumerate}
\end{definition}

 Let $\hickp \in \{\nick, \nickt, \nicks\}$ and let $\sigma$ be a sequent and let $\Gamma_\sigma^{ns} \subseteq \Gamma_\sigma$ and $\Delta_\sigma^{ns} \subseteq \Delta_\sigma$ be the sets of formulae occurring on the left side and right side of $\sigma$, respectively, which are \emph{not} $\hickp$-saturated and are not of the form $\K_\agi \varphi$. Recall the definition of complexity of a formula or a set of formulae (c.f. Definition \ref{d: complexity}).

\begin{lemma}\label{l: proof-search trees exist}
   Every sequent $\sigma$ has a $\hickp$-proof-search tree. 
\end{lemma}
\begin{proof}[Proof sketch.]
    Let $\sigma$ be a sequent and suppose that $\sigma_1$ is a premise of some preserving rule application of an invertible logical rule to $\sigma$ where the principal formula is not saturated or a premise of a rule instance of $\mathsf{\C R}$ with the principal formula in focus. An inspection of the rules yields that 
    \begin{equation}\label{e: proof-search tree exists}
        c(\Gamma_\sigma^{ns}) + c(\Delta_\sigma^{ns}) > c(\Gamma_{\sigma_1}^{ns}) + c(\Delta_{\sigma_1}^{ns}).
    \end{equation}
    In particular if $\sigma$ is $\Gamma, \C \varphi^\u \Rightarrow \Delta$ and the rule applied is $\mathsf{\C L}$ with $\C \varphi^\u$ principal, then the premise $\sigma_1$ is $\Gamma, \C \varphi^\u, \varphi^\u, \K_\agi \C \varphi^\u \Rightarrow \Delta$. Note that $c(\varphi) < c(\C \varphi)$ and $\C \varphi, \K_\agi \C \varphi \not \in \Gamma_{\sigma_1}^{ns}$, implying that (\ref{e: proof-search tree exists}) holds. By induction on $c(\Gamma_\sigma^{ns}) + c(\Delta_\sigma^{ns})$ we can therefore prove that preservingly applying invertible logical rules bottom-up to non-saturated formulae starting at $\sigma$ results in a finite tree where each leaf is either an axiom or saturated. 
    
    Leafs that are axioms or saturated sequents which are not conclusions of a $\mathsf{C_n}$-rule instance are closed. To any other leaf we apply the $\mathsf{C_n}$-rule bottom-up. We then co-recursively repeat this procedure for each open leaf, resulting in the construction of a proof-search tree.
\end{proof}
 
  As a corollary of this construction we obtain the following result.
  
\begin{lemma}\label{lemma every inf branch inf many C}
    Every infinite branch of a proof-search tree contains infinitely many applications of $\mathsf{C_n}$. 
\end{lemma}

Next, we define \emph{proof-search games}, which are played by two players called Prover and Refuter on a given proof-search tree $\mathcal{T}$.

\begin{definition}\label{d: proof-search tree for iltln}
    Let $\hickp \in \{\cick, \cickt, \cicks\}$, $\sigma$ be a sequent and $\mathcal{T}$ a $\hickp$-proof-search tree for $\sigma$ with choice rule $\mathsf{C_n}$. The \emph{proof-search game} $\mathcal{G}_\hickp(\mathcal{T}, \mathsf{C_n})$ is played by two players called \emph{Prover} and \emph{Refuter}. The \emph{arena} is the proof-search tree $\mathcal{T}$, where each \emph{position} is a node of $\mathcal{T}$. Prover \emph{owns} every position $t \in \mathcal{T}$ which is labeled by the conclusion of a $\mathsf{C_n}$-rule instance. Refuter owns every other position. If the game is in position $t \in \mathcal{T}$ owned by Prover/Refuter, then the \emph{admissible moves} for Prover/Refuter are the children of $t$, i.e. Prover/Refuter moves by chosing a child of $t$ (if one exists). A \emph{play} is a sequence of positions $(t_i)_i$ such that $t_0$ is the root of $\mathcal{T}$ and any two consecutive positions are related by an admissible move. A play is either finite and ends in a leaf of $\mathcal{T}$ or infinite. Note that every play is a branch of $\mathcal{T}$. The \emph{winning conditions} are as follows.
    \begin{enumerate}
        \item Prover wins a play $(t_i)_i$ if the play is finite and ends in an axiom or if it is infinite and $(t_i)_i$ has a good suffix.
        \item Refuter wins a play $(t_i)_i$ if the play is finite and ends in a non-axiomatic sequent or infinite and does not have a good suffix.
    \end{enumerate}
\end{definition}

 We will usually identify plays in $\mathcal{G}_\hickp(\mathcal{T}, \mathsf{C_n})$ and branches of $\mathcal{T}$ without explicit mention. Given a partial function $f$, we write $dom(f)$ for the \emph{domain} and $ran(f)$ for the \emph{range} of $f$. The following definitions are given for both players Prover and Refuter. We simply write Player instead.

\begin{definition}\label{d: strategy}
    Given a proof-search game $\mathcal{G}_\hickp(\mathcal{T}, \mathsf{C_n})$, a \emph{strategy} for Player is a partial function $f : \mathcal{T} \longrightarrow \mathcal{T}$ such that for any node $u$ of $\mathcal{T}$ if $u$ is not owned by Player or if $u$ is a leaf, then  $u \not \in dom(f)$ and if $u$ is owned by Player and not a leaf, then $u \in dom(f)$ where $f(u)$ is a child node of $u$.
\end{definition}

A strategy thus tells Player how to move in every owned position. Note that each node $u \in \mathcal{T}$ is unique, implying that strategies as defined here are aware of the history of the game. Player \emph{uses strategy $f$} if whenever the current play is in position $u$ owned by Player and $u \in dom(f)$, then Player moves to $f(u)$.

\begin{definition}\label{d: winning strategy}
    Let $\mathcal{G}_\hickp(\mathcal{T}, \mathsf{C_n})$ be a proof-search game and $f$ a strategy for Player. The strategy $f$ is \emph{winning} if Player wins every play in which $f$ is used.
\end{definition}

Given a game $\mathcal{G}_\hickp(\mathcal{T}, \mathsf{C_n})$ and a strategy $f$ for Player, the \emph{strategy tree} of $f$ is the subtree of $\mathcal{T}$ which consists of all plays that can occur when Player uses $f$. The formal definition is as follows.

\begin{definition}\label{d: strategy tree}
    The \emph{strategy tree} $\mathcal{T}_f$ of a strategy $f$ for Player in $\mathcal{G}_\hickp(\mathcal{T}, \mathsf{C_n})$ is the subtree of $\mathcal{T}$ defined by
    \begin{enumerate}
        \item $\mathcal{T}_f$ contains the root of $\mathcal{T}$.
        \item If $u \in \mathcal{T}$ belongs to $\mathcal{T}_f$ and $u \in dom(f)$, then for any child $v$ of $u$ in $\mathcal{T}$ holds that $v \in \mathcal{T}_f$ if and only if $v=f(u)$.
        \item If $u \in \mathcal{T}$ belongs to $\mathcal{T}_f$ and $u \not \in dom(f)$, then all children $v$ of $u$ in $\mathcal{T}$ belong to $\mathcal{T}_f$.
    \end{enumerate}
\end{definition}


\begin{definition}\label{def refutation}
    Let $\sigma$ be a sequent and $\mathcal{T}$ a $\hickp$-proof-search tree for $\sigma$. A \emph{refutation} of $\sigma$ is a subtree $\mathscr{R}$ of $\mathcal{T}$ satisfying the following properties.
    \begin{enumerate}
        \item $\mathscr{R}$ contains the root of $\mathcal{T}$.
        \item No leaf of $\mathscr{R}$ is an axiom. 
        \item No infinite branch of $\mathscr{R}$ has a good suffix.
        \item If $\mathscr{R}$ contains a node $u$ that is labeled by the conclusion of a $\mathsf{C_n}$-application, then $\mathscr{R}$ contains all children of $u$ in $\mathcal{T}$.
        \item If $\mathscr{R}$ contains a node $u$ that is labeled by the conclusion of any other rule than $\mathsf{C_n}$, $\mathsf{id}$ or $\bot$, then $\mathscr{R}$ contains exactly one child of $u$ in $\mathcal{T}$.
    \end{enumerate}
\end{definition}

It is immediate to see from the definition of the winning conditions of Refuter that the strategy trees of winning strategies for Refuter are refutations.

\begin{lemma}\label{l: winning strategy for refuter is refutation IM}
    Let $f$ be a winning strategy for Refuter in the game $\mathcal{G}_\hickp(\mathcal{T}, \mathsf{C_n})$. Then $\mathcal{T}_f$ is a refutation.
\end{lemma}

In the next subsection we will show that when using specific choice rules $\mathsf{C_n}$, refutations of $\sigma$ correspond to countermodels for $\sigma$, while winning strategies for Prover correspond to non-wellfounded proofs of $\sigma$. For showing that refutations correspond to countermodels, we will make use of the following `skeleton' model construction. For this construction, we assume that the premises of the choice rule $\mathsf{C_n}$ have been partitioned into $\lvert \A \rvert + 1$ groups: the first group being the \emph{intuitionistic premises} and the remaining $\lvert \A \rvert$ groups being the \emph{modal premises of agent $\agi$} for each $\agi \in \A$.

\begin{definition}\label{d: canonical model for nIM}
    Let $\sigma$ be a sequent, $\mathcal{T}$ a $\hickp$-proof-search tree for $\sigma$ and $\mathscr{R}$ a refutation of $\sigma$. The \emph{skeleton model} based on $\mathscr{R}$ is the structure $\fw_\mathscr{R}=(\W, \leq, \rel, \V)$ defined as follows.
    \begin{enumerate}
        \item $\W = \sfrac{\mathscr{R}}{\sim}$, where $s \sim t$ iff there exists a path between $s$ and $t$ in $\mathscr{R}$ in which no instance of the $\mathsf{C_n}$-rule occurs.

        \item $\leq$ is the reflexive, transitive closure of the relation $ {\leq_0} \subseteq \W \times \W$ given by
        \begin{equation*}
        \begin{split}
            w \leq_0 v \text{ iff } &\text{there exist } s \in w \text{ and } t \in v \text{ such that } s \text{ is the conclusion and }\\ &t \text{ an intuitionistic premise}
            \text{ of the same } \mathsf{C_n}\text{-rule instance.}
         \end{split}
        \end{equation*}

        \item $\rel = \{R_\agi \mid \agi \in \A\}$ where $R_\agi\subseteq \W\times \W$ is such that
        \[
        \begin{split}
            w \mathrel{R_\agi} v \text{ iff } &\text{there exist } s \in w \text{ and } t \in v \text{ such that } s \text{ is the conclusion and }\\ &t \text{ is a modal} \text{ premise of agent $\agi$ of the same } \mathsf{C_n}\text{-rule instance.}
         \end{split}
        \]

        \item $\V \colon \W \longrightarrow \mathcal{P}(\Prop)$ is defined by  $V(w) = \Gamma_w \cap \Prop$ where $\Gamma_w$ is the left side of the sequent labeling the unique node in \( w \) that is the conclusion of a $\mathsf{C_n}$-rule application or a leaf of $\mathscr{R}$.
    \end{enumerate}

\end{definition}
\noindent The construction of $\fw_\mathscr{R}$ reflects the idea that applying invertible rules to sequents in $\mathscr{R}$ provides more information about the current `world', while applying non-invertible rules (captured by the $\mathsf{C_n}$-rule) corresponds to taking either a modal or an intuitionistic step. Next, we introduce the following operations on birelational models.

\begin{definition}
    Let $\fw=(\W, \leq, \rel, \V)$ be a birelational model.
    \begin{enumerate}
        \item The \emph{triangle closure} $\hat{\fw}$ of $\fw$ is defined as $\hat{\fw} \coloneqq (\W, \leq, \{(\R_\agi \circ \leq)\}_{\R_\agi \in \rel}, \V)$ where $\circ$ denotes relation composition.
        \item The \emph{reflexive closure} $r(\fw)$ of $\fw$ is defined as 
        \begin{equation*}
            {r(\fw) \coloneqq (\W, \leq, \{\R_\agi \cup \{(w,w) \mid w \in \W\}\}_{\R_\agi \in \rel}, \V)}.
        \end{equation*}
        \item The \emph{transitive closure} $r(\fw)$ of $\fw$ is defined as 
        \begin{equation*}
            {r(\fw) \coloneqq (\W, \leq, \{\R_\agi \cup \{(w,u) \mid (w,v), (v,u) \in \R_\agi\}\}_{\R_\agi \in \rel}, \V)}.
        \end{equation*}
    \end{enumerate}
\end{definition}

The following lemma is then easy to check. The proof is omitted.

\begin{lemma}\label{l: model operations}
    If $\fw$ is a birelatonal model, then $\hat{\fw}$ is an epistemic model. Moreover, $r(\hat{\fw})$ is a reflexive model and $t(r(\hat{\fw}))$ is an S4 model.
\end{lemma}

\subsubsection{Completeness for Epistemic, Reflexive and  S4 Frames}\label{ss: completeness for epistemic, reflexive and s4 frames}
By using the framework introduced above, we now prove completeness for $\nick$, $\nickt$ and $\nicks$ with respect to the classes of epistemic, reflexive and S4 frames, respectively. Recall that for a set of formula $\Gamma$, $\K_\agi \Gamma = \{ \K_\agi \varphi \mid \varphi \in \Gamma\}$. Define $\K_\agi^0 \Gamma \coloneqq \Gamma$ and $\K_\agi^{n+1} \Gamma \coloneqq \K_\agi \K_\agi^n \Gamma$. We define the following choice rules.


\begin{definition}
The choice rule $\mathsf{C_n}$ for $n \in \{0,1\}$ is given by
    \begin{equation*}
            \infer[\mathsf{C_n}]{\Gamma \Rightarrow \Delta}{\{\Gamma, \varphi^\u \Rightarrow \psi^\u\}_{\varphi \rightarrow \psi \in \Delta^-} & \{ \{\K_\agi^n \K_\agi^{-1} \Gamma \Rightarrow \chi^\an \}_{\chi \in \K_\agi^{-1}\Delta^{-}} \}_{\agi \in \A}}
        \end{equation*}
        where the annotation $\an$ is equal to $\f$ if the underlying formula $\chi$ is a formula of the form $\chi = \C \gamma$ for some formula $\gamma$, and equal to $\u$ otherwise. The modal premises are those of the form $\K_\agi^n\K_\agi^{-1} \Gamma \Rightarrow \chi^\agi$ for each $\agi \in \A$, where the set of premises $\{\K_\agi^n\K_\agi^{-1}\Gamma \Rightarrow \chi^{\an} \}_{\chi \in \K_\agi^{-1}\Delta^-}$ forms the group of \emph{modal premises} of agent $\agi$ and the set of premises $\{\Gamma, \varphi^\u \Rightarrow \psi^\u\}_{\varphi \rightarrow \psi \in \Delta^-}$ forms the group of \emph{intuitionistic premises}.
\end{definition}

Note that for $n=0$ the modal premises together with the conclusion of $\Ck$ form instances of $\mathsf{\K_\agi}$ while for $n=1$ the modal premises of $\Crt$ together with the conclusion form instances of $\mathsf{S4_\agi}$. In case $\K_\agi \chi \in \Delta^-$ where $\chi = \C \gamma$ for some formula $\gamma$, the choice rules also apply the focus rules if $\K_\agi \chi$ does not already occur focused in $\Delta$. Furthermore, the conclusion and intuitionistic premises of $\mathsf{C_n}$ form instances of ${\rightarrow}\mathsf{R}$.

\begin{lemma}\label{lem WS for P with Ct correspond to proofs}
     Let $\hickp \in \{\nick, \nickt, \nicks\}$ and $\sigma$ be a sequent. If $\mathcal{T}$ is a $\hickp$-proof-search tree for $\sigma$ with choice rule $\mathsf{C_n}$, where $n=1$ if $\hickp = \nicks$ and $n=0$ otherwise, and Prover has a winning strategy in $\mathcal{G}_\hickp(\mathcal{T}, \mathsf{C_n})$, then $\sigma$ has a $\hickp$-proof. 
 \end{lemma}
 
 \begin{proof}
     Let $f$ be a winning strategy for Prover in $\mathcal{G}_\hickp(\mathcal{T}, \mathsf{C_n})$ and $\mathcal{T}_f$ be the strategy tree of $f$. By definition $\mathcal{T}_f$ contains the root of  $\mathcal{T}$ which is labeled by $\sigma$. By definition of $\mathcal{G}_\hickp(\mathcal{T}, \mathsf{C_n})$, the only positions owned by Prover are those nodes of $\mathcal{T}$ that are labeled by the conclusion of a $\mathsf{C_n}$-rule instance. Therefore whenever $\mathcal{T}_f$ contains a node $u$ labeled by the conclusion of an invertible rule instance, $\mathcal{T}_f$ contains all of $u$'s children in $\mathcal{T}$, implying that every such instance forms a correct rule instance of an invertible rule. If $u$ is labeled by the conclusion of a $\mathsf{C_n}$-rule instance, then $\mathcal{T}_f$ contains exactly one child $v$ of $u$. Observe that if $v$ is an intuitionistic premise, then the sequents labeling $v$ and $u$ form an instance of the rule ${\rightarrow}\mathsf{R}$, and if $v$ is a modal premise of agent $\agi$, then the sequents labeling $v$ and $u$ form an instance of $\mathsf{\K_\agi}$, potentially succeeded by an instance of $\mathsf{f}$, in case $n=0$, or an instance of $\mathsf{S4_\agi}$, potentially succeeded by an instance of $\mathsf{f}$, in case $n=1$. Let $\pi$ be $\mathcal{T}_f$ where whenever $u$ is labeled by the conclusion of an instance of $\mathsf{C_n}$ of the form
     \begin{equation*}
         \Gamma \Rightarrow \K_\agi \chi^\u, \Delta
     \end{equation*}
    and its unique child $v$ in $\mathcal{T}_f$ is a modal premise labeled by
     \begin{equation*}
         \K_\agi^n\K_\agi^{-1} \Gamma \Rightarrow \chi_\agi^\f
     \end{equation*}
    then $u$ has a unique child $u'$ in $\pi$ (where $u'$ does not occur in $\mathcal{T}_f$ or $\mathcal{T}$) labeled by
     \begin{equation*}
          \K_\agi^n \K_\agi^{-1}\Gamma \Rightarrow \chi_\agi^\u
     \end{equation*}
     and $u'$ has a unique child which is $v$. Observe that $(u,u')$ forms an instance of $\mathsf{\K_\agi}$ in case $n=0$ and an instance of $\mathsf{S4_\agi}$ if $n=1$, and $(u', v)$ forms an instance of $\mathsf{f}$. Otherwise we leave $\mathcal{T}_f$ unchanged. It follows from our construction that  $\pi$ is a $\hickp$-pre-proof of $\sigma$. Since $\mathcal{T}_f$ is the strategy tree of a \emph{winning strategy} for Prover, every leaf of $\mathcal{T}_f$ is labeled by an axiom, implying that the same holds for $\pi$, and every infinite branch contains a good suffix $\rho = (\rho(i))_i$. Note that whenever $\rho(i)$ is the conclusion of an instance of $\mathsf{C_n}$ in $\mathcal{T}_f$, then $\rho(i)$ contains a formula $\K_\agi \chi^\f$ in focus and $\rho(i+1)$ is the modal premise with $\chi^\f$ as the only formula on the right. Therefore our construction does not alter this suffix, implying that the corresponding infinite branch in $\pi$ contains the same suffix $\rho$ and is therefore good. Thus $\pi$ is a $\hickp$-proof of $\sigma$.
 \end{proof}

Prover having a winning strategy in a proof-search game for sequent $\sigma$ therefore implies that $\sigma$ is provable. In order to show that Refuter having a winning strategy implies that $\sigma$ is falsifiable, we need the following lemma.

\begin{lemma}\label{l: skeleton model is birelation}
    Let $\sigma$ be a sequent, $\mathcal{T}$ a $\hickp$-proof-search tree for $\sigma$ relative to choice rule $\mathsf{C_n}$ for $n \in \{0,1\}$ and $\mathscr{R}$ a refutation of $\sigma$. The skeleton model $\fw_\mathscr{R}$ based on $\mathscr{R}$ is a birelational model.
\end{lemma}
\begin{proof}
    The only non-trivial part is that $\V$ is monotone in $\leq$. So suppose $w \leq v$. First, let us consider the case where $w \leq_0 v$. Then there exist $s \in w$ and $t \in v$ such that $s$ is the conclusion and $t$ an intuitionistic premise of an instance of $\mathsf{C_n}$. Let $\Gamma_s \Rightarrow \Delta_s$ and $\Gamma_t \Rightarrow \Delta_t$ be the sequents labeling $s$ and $t$, respectively. By inspection of the rule $\mathsf{C_n}$, note that $\Gamma_s \subseteq \Gamma_t$. Since every rule is applied preservingly, it further follows that $\Gamma_t \subseteq \Gamma_{t'}$ where $t'$ is the unique node in $v$ which is the conclusion of an instance of $\mathsf{C_n}$ or a leaf. Hence $\Gamma_s \subseteq \Gamma_t \subseteq \Gamma_{t'}$, implying that $\V(w) \subseteq \V(v)$. The general case where $w \leq v$ is then shown by induction on the length of the path from $w$ to $v$, where the base case (where $v = w)$ is trivial and the induction step follows by using the induction hypothesis and the argument above.
\end{proof}

\begin{proposition}\label{p: refutation gives triangle model}
Let $\sigma$ be a sequent.
\begin{enumerate}
    \item If $\mathcal{T}$ is a $\nick$-proof-search tree for $\sigma$ with choice rule $\mathsf{C_0}$ and Refuter has a winning strategy in $\mathcal{G}_\nick(\mathcal{T}, \mathsf{C_0)}$, then $\sigma$ is falsifiable over the class of epistemic frames.

    \item If $\mathcal{T}$ is a $\nickt$-proof-search tree for $\sigma$ with choice rule $\mathsf{C_0}$ and Refuter has a winning strategy in $\mathcal{G}_\nickt(\mathcal{T}, \mathsf{C_0)}$, then $\sigma$ is falsifiable over the class of reflexive frames.

    \item If $\mathcal{T}$ is a $\nicks$-proof-search tree for $\sigma$ with choice rule $\mathsf{C_1}$ and Refuter has a winning strategy in $\mathcal{G}_\nicks(\mathcal{T}, \mathsf{C_1)}$, then $\sigma$ is falsifiable over the class of S4 frames.
\end{enumerate}

\end{proposition}

\begin{proof}
We prove all three items simultaneously. Let $\hickp \in \{\nick, \nickt, \nicks\}$ and let $f$ be a winning strategy for Refuter in $\mathcal{G}_\hickp(\mathcal{T}, \mathsf{C_n})$, where $n=1$ in case $\hickp = \nicks$ and $n=0$ otherwise, and let $\mathcal{T}_f$ be the strategy tree of $f$. By Lemma \ref{l: winning strategy for refuter is refutation IM},  $\mathcal{T}_f = \mathscr{R}$ is a subtree of $\mathcal{T}$ that is a refutation of $\sigma$. Let $\fw_\mathscr{R}$ be the skeleton model based on $\mathscr{R}$. If $\hickp = \nick$, then define $\fw \coloneqq \hat{\fw_\mathscr{R}}$. If $\hickp = \nickt$, then define $\fw \coloneqq r(\hat{\fw_\mathscr{R}})$ and if $\hickp = \nicks$, then define $\fw \coloneqq t(r(\hat{\fw_\mathscr{R}}))$. We write $\fw =(\W, \leq, \rel, \V)$. By Lemma \ref{l: skeleton model is birelation} and Lemma \ref{l: model operations}, $\fw$ is an epistemic, reflexive or S4 model, respectively. We will show that $\fw$ falsifies $\sigma$.

For each $s \in \mathcal{T}$ let $\Gamma_s \Rightarrow \Delta_s$ be the sequent labeling $s$. For each $w\in \W$, let $\Gamma_w \coloneqq \bigcup_{s \in w} \Gamma_s$ and let $\Delta_w \coloneqq \bigcup_{s \in w} \Delta_s$. Let $\varphi$ be a formula. By induction on the structure of $\varphi$, we simultaneously prove that for any $w\in \W$ holds
\begin{enumerate}
    \item[(a)] $\fw,w\models \varphi$ if $\varphi\in \Gamma^-_{w}$, and
    \item[(b)] $\fw,w\not\models \varphi$ if $\varphi\in \Delta^-_{w}$.
\end{enumerate}

 For any $w \in \W$ let $t_w \in w$ be the unique node labeled by the conclusion of a $\mathsf{C_n}$-instance or a leaf of $\mathscr{R}$. For the case for $p \in \Prop$ suppose for (a) that $p \in \Gamma_w^-$. Then there exists $t \in w$ such that $p \in \Gamma_t^-$. Since there are no applications of $\mathsf{\K_\agi}$ or $\mathsf{S4_\agi}$ on the path from $t$ to $t_w$, we have $p \in \Gamma_{t_w}^-$ and hence $p \in \V(w)$, implying that $\fw, w \models p$. For (b) suppose that $p \in \Delta_w^-$. Then there exists $t \in w$ such that $p \in \Delta_{t}^-$. Since there are no applications of ${\rightarrow\mathsf{R}}$, $\mathsf{\K_\agi}$ or $\mathsf{S4_\agi}$ on the path from $t$ to $t_w$, we have that $p \in \Delta_{t_w}^-$. Suppose for contradiction that $p \in \Gamma_{t_w}^-$. Then $\Gamma_{t_w} \Rightarrow \Delta_{t_w}$ is an instance of $\mathsf{id}$. By definition of a proof-search tree, $t_w$ is an axiomatic leaf, which contradicts the assumption that $\mathscr{R}$ is a refutation. Hence $p \not \in \Gamma_{t_w}^-$ and so $p \not \in \V(w)$, implying that $\fw, w \not \models p$.  For $\varphi = \bot$ note that if $\bot \in \Gamma_w^-$, then there exists $t \in w$ such that $\bot \in \Gamma_t^-$. Therefore $\Gamma_t \Rightarrow \Delta_t$ is an instance of the axiom $\bot$. By definition of a proof-search tree, $t$ is a leaf of $\mathcal{R}$, implying that $\mathcal{R}$ contains an axiomatic leaf, which contradicts the assumption that $\mathcal{R}$ is a refutation. Hence $\bot \not \in \Gamma_w^-$. If $\bot \in \Delta_w^-$, then by definition $\fw, w \not \models \bot$. The cases for $\varphi = \psi \wedge \chi$ and $\varphi = \psi \vee \chi$ follow immediately from saturation and the induction hypothesis. For example if $\psi \wedge \chi \in \Gamma_w^-$, then, since invertible rules are applied preservingly, $\psi \wedge \chi \in \Gamma_{t_w}^-$. Since $\Gamma_{t_w} \Rightarrow \Delta_{t_w}$ is saturated, we have $\psi, \chi \in \Gamma_{t_w}^-$ and hence $\psi, \chi \in \Gamma_w^-$. By induction hypothesis $\fw, w \models \psi$ and $\fw, w \models \chi$, implying that $\fw, w \models \psi \wedge \chi$. The other cases involving $\psi \wedge \chi$ and $\psi \vee \chi$ are similar and omitted. \smallskip

\noindent \textsc{Case for $\rightarrow$.} For (a) if $\varphi\rightarrow \psi\in \Gamma^-_w$, then since rules are applied preservingly we have $\varphi \rightarrow \psi \in \Gamma_{t_w}^-$. If $v\geq w$, then since rules are applied preservingly and by definition of $\mathsf{C_n}$, it holds that $\varphi\to \psi\in \Gamma^-_{t_v}$. Since $\Gamma_{t_v} \Rightarrow \Delta_{t_v}$ is saturated, $\varphi \in \Delta^-_v$ or $\psi \in \Gamma^-_v$. By the induction hypothesis, $\fw,v\models \psi$ or $\fw,v\not\models \varphi$, so we obtain $\fw,w\models \varphi\to\psi$. 
   For (b) if $\varphi\to \psi\in \Delta_w^-$, then $\varphi \rightarrow \psi \in \Delta_{t_w}^-$, since $\varphi \rightarrow \psi$ may only be principal in instances of $\mathsf{C_n}$. By definition of $\mathsf{C_n}$ and construction of $\leq$ there exists a $v\geq_0 w$ such that $\varphi\in \Gamma^-_v$ and $\psi\in \Delta^-_v$. The induction hypothesis then implies that $\fw,v\models \varphi$ and $\fw,v\not\models\psi$, and so $\fw,w\not\models \varphi\to\psi$. \smallskip
   
\noindent \textsc{Case for $\K_\agi$.} For (a) if $\K_\agi\varphi\in \Gamma_w^-$, then $\K_\agi \varphi \in \Gamma_{t_w}^-$, since $\K_\agi \varphi$ can only be principal in the rule $\mathsf{C_n}$ and - if present in $\hickp$ - in the rule $\mathsf{T_\agi}$, which is applied preservingly. There are three cases to consider. 1. Let $\hickp = \nick$ and suppose $w \Rel_\agi v$. There are two subcases. (i) Suppose there exists a node $s \in v$ which is a modal premise of the $\mathsf{C_0}$-instance with $t_w$ as conclusion, then by definition of $\mathsf{C_0}$ we have $\varphi\in \Gamma_s^-$ and so $\varphi \in \Gamma_v^-$. (ii) Suppose there exists $u \geq w$ and a node $s \in v$ which is a modal premise of a $\mathsf{C_0}$-instance with conclusion $t_u$. Then, by definition of $\mathsf{C_0}$ and the proof-search tree, $\K_\agi \varphi \in \Gamma_{t_u}^-$ and hence, as before, $\varphi \in \Gamma_v^-$.\footnote{Since $\fw$ is the skeleton model where each modal relation was closed under triangle confluence, we have $w \Rel_\agi v$. However, in the refutation $v$ does not contain a node which is a modal premise of the same $\Ck$-rule instance of which $t_w$ is a conclusion of. Therefore this additional case has to be considered.} The induction hypothesis then implies  $\fw,v\models \varphi$, so we obtain $\fw,w\models \K_\agi\varphi$. 2. Let $\hickp = \nickt$ and suppose $w \Rel_\agi v$. There are two subcases. (i) Suppose that $w \not = v$. The argument is then identical to case 1. (ii) Suppose that $w = v$. Then since $\K_\agi \varphi \in \Gamma_t^-$ for some $t \in w$ by assumption, the definition of proof-search tree guarantees that on the path from $t$ to $t_w$ the rule $\mathsf{T_\agi}$ is applied with $\K_\agi \varphi^\u$ principal. Therefore $\varphi \in \Gamma_{t_w}^-$ and so the induction hypothesis yields $\fw, w \models \varphi$. Combining both subcases gives the desired result that $\fw, w \models \K_\agi \varphi$. 3. Let $\hickp = \nicks$ and suppose $w \Rel_\agi v$. Since $\fw$ for this case is an S4 model and by construction of the skeleton model, there are pairwaise distinct worlds $w= w_0, \ldots w_n = v$ such that $w_i \Rel_\agi w_{i+1}$ and whenever there is a world $u$ with $w_i \Rel_\agi u \Rel_\agi w_{i+1}$ for $i \in [n]$, then $u = w_i$ or $u= w_{i+1}$. We show that $\varphi, \K_\agi \varphi \in \Gamma_v^-$ by induction on $n$. For $n=0$ we have $w = v$ and $\fw, w \models \varphi$ follows by the same argument as used for 2.(ii). For $n>0$ we have that $\varphi, \K_\agi \varphi \in \Gamma_{w_{n-1}}^-$ by induction hypothesis. By assumption $w_{n-1} \not = w_n$ and $w_{n-1} \Rel_\agi w_n$ with no world strictly inbetween. Suppose there exists $s \in w_n$ such that $s$ is a modal premise of agent $\agi$ of the same instance of $\mathsf{C_1}$ of which $t_{w_{n-1}}$ is the conclusion of. Since $\K_\agi \varphi \in \Gamma_{w_{n-1}}^-$, we have $\K_\agi \varphi \in \Gamma_{t_{w_{n-1}}}$ by the same argument as before. By inspection of the rule $\mathsf{C_1}$, note that then $\K_\agi \varphi \in \Gamma_{w_n}$. Since $\Gamma_{w_n} \Rightarrow \Delta_{w_n}$ is saturated, we also have $\varphi \in \Gamma_{w_n}$. Otherwise suppose there exists $u \geq w$ such that there exists $s \in w_n$ which is a modal premise of agent $\agi$ of the same instance of $\mathsf{C_1}$ of which $t_{u}$ is the conclusion of. As before we have $\K_\agi \varphi \in t_u$ since rules are applied preservingly and by definition of $\mathsf{C_1}$. It then follows that $\K_\agi \varphi, \varphi \in \Gamma_{w_n}^-$ by the same argument as before. Hence for any $v \in \W$ with $w \Rel_\agi v$ holds that $\fw,v \models \varphi$ and so $\fw, w \models \K_\agi \varphi$. 

For (b) if $\K_\agi\varphi\in \Delta_w^-$, then $\K_\agi \varphi \in \Delta_{t_w}^-$. By definition of $\mathsf{C_n}$ there exists a modal premise $s$ of agent $\agi$ with $\varphi \in \Delta_s^-$. For $v = [s]$, we therefore have $\varphi \in \Delta_v$. By induction hypothesis $\fw,v\not\models \varphi$. By construction of $\fw$ it holds that $w \Rel_\agi v$ and  so $\fw,w\not\models \K_\agi\varphi$.\footnote{Note that it does not matter to which class of frames $\fw$ belongs for this argument.} \smallskip
    
\noindent \textsc{Case for $\C$.} For (a) if $\C\varphi\in \Gamma^-_w$, then $\C \varphi \in \Gamma_{t_w}^-$, since invertible rules are applied preservingly. Suppose $w\Rel^* v$. Let 
\begin{equation*}
    w = u_0 \Rel_{\agi_0} u_1 \Rel_{\agi_1} \ldots \Rel_{\agi_{n-1}} u_n = v
\end{equation*}
be the $R$-path from $w$ to $v$. We prove by induction that for every $i \in [n+1]$ holds that ${\varphi, \K_\agi \C \varphi \in \Gamma_{u_i}^-}$ for each $\agi \in \A$. For $u_0$, note that $u_0= w$ and so it follows from saturation that  $\varphi, \K_\agi \C \varphi \in \Gamma_w^-$ for each $\agi \in \A$. For $i >0$ we have by induction hypothesis that $\K_\agi \C \varphi \in \Gamma_{u_{i-1}}^-$ for each $\agi \in \A$ and hence $\K_\agi \C \varphi \in \Gamma_{t_{u_{i-1}}}^-$ for each $\agi \in \A$ since invertible rules are applied preservingly. If $u_i$ contains a node $s$ which is a modal premise of some agent $\agii$ of the same $\mathsf{C_n}$-instance of which $t_{u_{i-1}}$ is the conclusion of, then since $\K_\agii \C \varphi \in \Gamma_{t_{u_{i-1}}}^-$, the definition of $\mathsf{C_0}$ yields that $\C \varphi \in \Gamma_s^-$ and the definition of $\mathsf{C_1}$ yields that $\K_\agii \C \varphi \in \Gamma_s^-$. By the presence of $\mathsf{T_\agii}$ in $\nicks$ and saturation we thus have $\C \varphi \in \Gamma_{u_i}^-$. Moreover, by saturation, we have $\varphi, \K_\agi \C \varphi \in \Gamma_{u_i}^-$ for each $\agi \in \A$. Otherwise $u_i$ is a modal successor of some agent $\agii$ of some world $u' \geq u_{i-1}$. Since $\K_\agi \C \varphi \in \Gamma_{t_{u_{i-1}}}^-$ for each $\agi \in \A$, it follows from the definition of $\mathsf{C_n}$ and the construction of the proof-search tree that $\K_\agi \C \varphi \in \Gamma_{\Tilde{u}}^-$ for all $\Tilde{u} \geq u_{i-1}$ and all $\agi \in \A$. Hence, in particular, $\K_\agii \C \varphi \in \Gamma_{u'}^-$ and so $\varphi, \K_\agi \C \varphi \in \Gamma_{u_\agi}^-$ for each $\agi \in \A$ by the same argument as in the previous case. We conclude that $\varphi \in \Gamma_v^-$ and so by induction hypotheses $\fw, v \models \varphi$. Hence $\fw, w \models \C \varphi$.

For (b) if $\C\varphi\in \Delta^-_w$, then saturation implies $\varphi\in \Delta^-_{t_w}$ or $\K_\agi\C\varphi\in \Delta^-_{t_w}$ for some $\agi \in \A$.\footnote{If $\C \varphi$ is unfocused, then the saturation clause 7. guarantees this. Otherwise the saturation clause 8. guarantees it. In the second case, there exists a node $s \in w$ with $\C \varphi^\f \in \Delta_s$. Since $\C \varphi^\f$ cannot occur in $\Delta_{t_w}$, there must be a (non-preserving) application of $\mathsf{\C R}$ with $\C \varphi^\f$ as principal formula.} Suppose, for contradiction, that for all $w \Rel^* v$ holds that $\varphi \not \in \Delta^-_v$. This implies that $\K_\agi \C \varphi \in \Delta_{t_w}^-$ for some $\agi \in \A$. Then we can define an infinite path $\rho$ in $\mathscr{R}$ starting from $t_w$ as follows: at each $\mathsf{C_n}$-application, we pick the modal premise that has $\C\varphi^\f$ as consequent. Since no $w \Rel^* v$ satisfies $\varphi\in \Delta^-_v$, this choice is always available. The suffix of the path $\rho$ which starts at the modal premise of the $\mathsf{C_n}$ instance that has $t_w$ as conclusion always has a formula in focus and between any two instances of $\mathsf{C_n}$ passes through an instance of $\mathsf{C \mathsf{R}}$ where the principal formula is in focus. Since by Lemma~\ref{lemma every inf branch inf many C} the suffix passses through infinitely many instances of $\mathsf{C_n}$, it passes through infinitely many instances of $\mathsf{C \mathsf{R}}$ where the principal formula is in focus and is therefore good. Thus the branch of $\mathscr{R}$ which contains $\rho$ as a suffix has a good suffix, contradicting the assumption that $\mathscr{R}$ is a refutation. So there must be some $w \Rel^* v$ with $\varphi\in \Delta^-_v$, and thus $\fw,v\not\models \varphi$ by the induction hypothesis. Therefore $\fw,w\not\models \C\varphi$.

Let $w \in \W$ contain the root of $\mathscr{R}$. By definition $\Gamma_\sigma \subseteq \Gamma_w$ and $\Delta_\sigma \subseteq \Delta_w$. Thus  $\fw,w \not \models \sigma$. Since $\fw$ is an epistemic/reflexive/S4 model depending on $\hickp$, we thus obtain that $\sigma$ is falsifiable over the respective class of frames.
\end{proof}

We have established that winning strategies for Prover correspond to non-wellfounded proofs and winning strategies for Refuter to countermodels. In order to obtain completeness, we require that the game $\mathcal{G}_\hickp(\mathcal{T}, \mathsf{C_n})$ is determined, which means that exactly one of the two players has a winning strategy. It is well-known that proof-search games like $\mathcal{G}_\hickp(\mathcal{T}, \mathsf{C_n})$ can be reformulated as \emph{Gale--Steward games}. In fact, since Gale--steward games only permit infinite plays, it suffices to extend every finite branch in the proof-search tree $\mathcal{T}$ by an infinite sequence of nodes labeled by the same sequent as the leaf and determine that such plays are then won by Prover if the suffix contains only axioms and by Refuter otherwise.  The resulting game is then a Gale--Steward game. It is a standard result in non-wellfounded proof theory that the set of winning plays for each player in such a game is Borel. In fact, the winning sets belongs to a low level of the Borel hierarchy, namely $\Delta_3$ (see for example~\cite{menendez_2024}). Therefore we may apply the following theorem.

\begin{theorem}[Martin 1975~\cite{Martin_1975}]
    If $\mathcal{G}$ is a Gale--Steward game and the set of winning plays for each player is Borel, then $\mathcal{G}$ is determined.
\end{theorem}

With Borel determincacy at hand, we may prove completeness.

\begin{theorem}[Completeness of Non-Wellfounded Calculi]\label{t: nonwellfounded completeness}
Let $\sigma$ be a sequent.
\begin{enumerate}
    \item If $\sigma$ is valid over the class of epistemic frames, then $\sigma$ has a $\nick$-proof.
    \item If $\sigma$ is valid over the class of reflexive frames, then $\sigma$ has a $\nickt$-proof.
    \item If $\sigma$ is valid over the class of S4 frames, then $\sigma$ has a $\nicks$-proof.
\end{enumerate}
\end{theorem}
\begin{proof}
   Suppose a sequent $\sigma$ is valid over class of epistemic frames. Let $\mathcal{T}$ be a proof-search tree with choice rule $\mathsf{C_0}$ for $\sigma$ and consider the two-player game $\mathcal{G}_\nick(\mathcal{T}, \mathsf{C_0})$. By determinacy, exactly one of the two players Prover or Refuter has a winning strategy. Since $\sigma$ is valid, Proposition \ref{p: refutation gives triangle model} implies that Refuter cannot have a winning strategy. Hence Prover must have a winning strategy. Therefore, by Lemma \ref{lem WS for P with Ct correspond to proofs}, $\sigma$ is provable in $\nick$. The cases for the classes of reflexive and S4 frames are argued similarly, where the choice rule $\mathsf{C_1}$ is used for $\nicks$.
\end{proof}

\subsubsection{Translating Non-Wellfounded into Cyclic Proofs}\label{ss: translating non-wellfounded into cyclic proofs}

With completeness of the non-wellfounded calculi established, we now show that the corresponding cyclic calculi are also complete by translating non-wellfounded proofs into cyclic proofs.

\begin{lemma}\label{l: infinite branches contain good repetitions}
    Let $\hickp \in \{\nick, \nickt, \nicks\}$ and let $\rho$ be an infinite branch of a $\hickp$-proof $\pi$. Then $\rho$ contains a successful repetition.
\end{lemma}
\begin{proof}
    Recall that each rule of $\hickp$ is analytic. This implies that every sequent occurring in $\rho$ consists of formulae of $\Cl(\sigma)$ where $\sigma$ is the sequent labeling the root of $\pi$. Since $\Cl(\sigma)$ is finite by Lemma \ref{l: negation closure is finite}, there are only finitely many pairwise different sequents occurring in $\rho$. Thus, there exists a sequent $\Gamma \Rightarrow \Delta$ which labels infinitely many nodes of $\rho$. Since $\pi$ is a $\hickp$-proof, there exists a good suffix $\rho'$ of $\rho$, i.e. every sequent in $\rho'$ has a formula in focus and $\rho'$ passes infinitely often through $\mathsf{\C R}$ with the principal formula in focus. Let $i < \omega$ be the least natural number such that $\rho(i)$ belongs to $\rho'$ and is labeled by $\Gamma \Rightarrow \Delta$. Since there are infinitely many nodes labeled by $\Gamma \Rightarrow \Delta$ in $\rho'$ and $\rho'$ passes through infinitely many instances of $ \mathsf{\C R}$ with the principal formula in focus, there exists a least natural number $i < j < \omega$ such that $\rho(j)$ is labeled by $\Gamma \Rightarrow \Delta$ as well and the path from $\rho(i)$ to $\rho(j)$ passes through an instance of $ \mathsf{\C R}$ with the principal formula in focus. Therefore $(\rho(i), \rho(j))$ is a successful repetition occurring in the branch $\rho$.
\end{proof}

Note that the existence of a successful repetition in $\rho$ implies the existence of a \emph{lowermost} successful repetition in $\rho$.

\begin{definition}
    Let $\hickp \in \{\nick, \nickt, \nicks\}$ and let $\pi$ be a $\hickp$-proof. The \emph{induced cyclic proof} $\pi_c$ is the subtree of $\pi$ obtained by pruning every infinite branch at the lowermost successful repetition.
\end{definition}

The following then follows immediately from Lemma \ref{l: infinite branches contain good repetitions}.

\begin{lemma}\label{l: non-wellfounded proof implies cyclic proof}
   Let $\hickp \in \{\mathsf{ICK}, \mathsf{ICKT}, \mathsf{ICKS4}\}$. If $\pi$ is a $\mathsf{n}\hickp$-proof, then $\pi_c$ is a $\mathsf{c}\hickp$-proof.
\end{lemma}
\begin{proof}
    Since $\pi$ is an $\mathsf{n}\hickp$-proof, every infinite branch of $\pi$ contains a successful repetition by Lemma \ref{l: infinite branches contain good repetitions}. Therefore, by K\H{o}nig's Lemma, $\pi_c$ is finite and every branch $\rho$ of $\pi_c$ corresponds to an initial segment of a branch $\rho'$ of $\pi$. If $\rho'$ is a finite branch, then $\rho = \rho'$ and since $\pi$ is an $\mathsf{n}\hickp$-proof, $\rho$ ends in an axiom. Otherwise $\rho'$ is an infinite branch and $\rho$ is the initial segment of $\rho'$ ending at the first successful repetition of $\rho'$. Thus every leaf of $\pi_c$ is an axiom or belongs to a successful repetition, implying that $\pi_c$ is a $\mathsf{n}\hickp$-proof.
\end{proof}

Completeness of the cyclic calculi then follows as a corollary of Theorem~\ref{t: nonwellfounded completeness} and Lemma~\ref{l: non-wellfounded proof implies cyclic proof}.

\begin{theorem}[Completeness of Cyclic Calculi]\label{t: cyclic completeness}
   Let $\sigma$ be a sequent.
   \begin{enumerate}
       \item If $\sigma$ is valid over the class of epistemic frames, then $\sigma$ has a $\cick$-proof.
        \item If $\sigma$ is valid over the class of reflexive frames, then $\sigma$ has a $\cickt$-proof.
         \item If $\sigma$ is valid over the class of S4 frames, then $\sigma$ has a $\cicks$-proof.
   \end{enumerate}
\end{theorem}
\begin{proof}
     Suppose $\sigma$ is valid over the class of epistemic frames. By Theorem \ref{t: nonwellfounded completeness}, $\sigma$ has a $\nick$-proof $\pi$. Lemma \ref{l: non-wellfounded proof implies cyclic proof} implies that $\pi_c$ is a $\cick$-proof of $\sigma$. The cases for reflexive and S4 frames are similar.
\end{proof}

\subsubsection{Completeness for S5 frames}\label{ss: completeness for s5 frames}

 This section shows that $\cickss$ is complete with respect to the class of S5 frames. Due to the presence of $\mathsf{cut}$ in the system we can prove completeness directly via a canonical model construction, which arguably simplifies the completeness proof. To obtain \emph{analytic completeness}, applications of $\mathsf{cut}$ will be restricted to a finite set of formulae relevant to the root sequent. So far we counted as relevant the formulae in the closure of the root sequent. However, as we have seen in the completeness of the axiomatization for the S5 case, the stronger notion of negation closure is needed to deal with the classicality of the modalities. In the following we provide the basic definitions and lemmas to obtain the canonical model for $\cickss$.

\begin{definition}
Let $\Sigma$ be a non-empty, negation closed and finite set of formulae. A sequent $\sigma$ is called a
\begin{enumerate}
    \item \emph{$\Sigma$-sequent} if $\Gamma_\sigma \cup \Delta_\sigma \subseteq \Sigma$;
    \item \emph{$\Sigma$-provable} if there exists a $\cickss$-proof of $\sigma$ in which only $\Sigma$-sequents occur;
    \item \emph{$\Sigma$-saturated} if $\sigma$ is $\Sigma$-unprovable and $\Gamma_\sigma \cup \Delta_\sigma = \Sigma$.
\end{enumerate} 
\end{definition}

Not that if $\Sigma=\Cl^{\neg}(\sigma)$, then $\sigma$ being $\Sigma$-provable means that $\sigma$ has an analytic proof. $\Sigma$-saturated sequents will be the worlds of the canonical model. Let us first prove an analogue of the Lindenbaum Lemma for saturated sequents.

\begin{lemma}[Lindenbaum]\label{l: Lindenbaum for saturated sequents}
    If a $\Sigma$-sequent $\sigma$ is $\Sigma$-unprovable, then there exists a $\Sigma$-saturated sequent $\sigma'$ with $\Gamma_\sigma \subseteq \Gamma_{\sigma'}$ and $\Delta_\sigma \subseteq \Delta_{\sigma'}$.
\end{lemma}
\begin{proof}
    Suppose $\sigma$ is $\Sigma$-unprovable. Let $\pi$ be the pre-proof of $\sigma$ built as follows: starting with $\sigma$ at the root, repeatedly apply the rule $\mathsf{cut}$ bottom-up to introduce fresh $\Sigma$-formulae (i.e. formulae not yet occurring in the conclusion) until every leaf is labeled by a sequent $\Gamma \Rightarrow \Delta$ with $\Gamma \cup \Delta = \Sigma$. At least one of these sequents must be $\Sigma$-unprovable (and thus $\Sigma$-saturated), as otherwise $\pi$ could be extended into a $\Sigma$-proof of $\sigma$.
\end{proof}

Next, we will establish some important properties of saturated sequents, that will become useful in the proof of the Truth Lemma (c.f. Lemma \ref{l: truth lemma for S5}).

\begin{lemma}\label{l: properties of saturated sequents}
    Let $\Gamma \Rightarrow \Delta$ be a $\Sigma$-saturated sequent. 
    \begin{enumerate}
        \item If $\varphi \wedge \psi \in \Sigma$, then $\varphi \wedge \psi \in \Gamma^-$ if and only if $\varphi \in \Gamma^-$ and $\psi \in \Gamma^-$. 
        \item If $\varphi \vee \psi \in \Sigma$, then $\varphi \vee \psi \in \Gamma^-$ if and only if $\varphi \in \Gamma^-$ or $\psi \in \Gamma^-$.
        \item If $\K_\agi \varphi \in \Gamma^-$, then $\varphi \in \Gamma^-$ for all $\agi \in \A$.
        \item $\K_\agi \varphi \in \Gamma^-$ if and only if $\neg \K_\agi \varphi \in \Delta^-$ for all $\agi \in \A$.
        \item $\C \varphi \in \Gamma^-$ if and only if $\varphi \in \Gamma^-$ and $\K_\agi \C \varphi \in \Gamma^-$ for all $\agi \in \A$.
    \end{enumerate}
\end{lemma}
\begin{proof}
    \textsc{1.} Let $\varphi \wedge \psi \in \Sigma$ and suppose towards contradiction that $\varphi \wedge \psi \in \Gamma^-$ but (without loss of generality) $\varphi \not \in \Gamma^-$. By saturation $\varphi \in \Delta^-$. We assume (without loss of generality) that $\varphi$ occurs unfocused in $\Delta$. Thus $\Gamma \Rightarrow \Delta$ is $\Sigma$-provable as follows:
    \begin{prooftree}
        \AxiomC{}
        \RightLabel{$\mathsf{id}$}
        \UnaryInfC{$\Gamma',\varphi^\u, \psi^\u \Rightarrow \varphi^\u, \Delta'$}
        \RightLabel{$\wedge\mathsf{L}$}
        \UnaryInfC{$\Gamma', \varphi \wedge \psi^\u \Rightarrow \varphi^\u, \Delta'$}
    \end{prooftree}
    where $\Gamma = \Gamma' \cup \{\varphi \wedge \psi^\u\}$ and $\Delta = \Delta'\cup\{\varphi^\u\}$. For the other direction suppose that $\varphi \wedge \psi \in \Sigma$ and both $\varphi \in \Gamma^-$ and $\psi \in \Gamma^-$. Assume towards contradiction that $\varphi \wedge \psi \not \in \Gamma^-$. By saturation $\varphi \wedge \psi \in \Delta^-$. Therefore $\Gamma \Rightarrow \Delta$ is $\Sigma$-provable as follows:
    \begin{prooftree}
        \AxiomC{}
        \RightLabel{$\mathsf{id}$}
        \UnaryInfC{$\Gamma', \varphi^\u, \psi^\u \Rightarrow \varphi^\u, \Delta'$}
        \AxiomC{}
        \RightLabel{$\mathsf{id}$}
        \UnaryInfC{$\Gamma', \varphi^\u, \psi^\u \Rightarrow \psi^\u, \Delta'$}
        \RightLabel{$\wedge \mathsf{R}$}
        \BinaryInfC{$\Gamma', \varphi^\u, \psi^\u \Rightarrow \varphi \wedge \psi^\u, \Delta'$}
    \end{prooftree}
    where $\Gamma = \Gamma' \cup \{\varphi^\u, \psi^\u\}$ and $\Delta = \Delta' \cup \{\varphi \wedge \psi^\u\}$. \smallskip
    
    \noindent The proofs of 2. and 3. are analogous by using the rules for disjunction and the rule $\mathsf{T_\agi}$, respectively. \smallskip \\
    \textsc{4.} Suppose towards contradiction that $\K_\agi \varphi \in \Gamma^-$ and $\neg \K_\agi \varphi \not \in \Delta^-$. Note that $\K_\agi \varphi \in \Gamma$ implies $\K_\agi \varphi \in \Sigma$ and since $\Sigma$ is negation closed, $\neg \K_\agi \varphi \in \Sigma$ as well. By saturation $\neg \K_\agi \varphi \in \Gamma^-$. Then $\Gamma \Rightarrow \Delta$ is $\Sigma$-provable as follows:
    \begin{prooftree}
        \AxiomC{}
        \RightLabel{$\mathsf{id}$}
        \UnaryInfC{$\Gamma', \K_\agi \varphi^\u \Rightarrow \K_\agi \varphi^\u, \Delta$}
        \AxiomC{}
        \RightLabel{$\bot$}
        \UnaryInfC{$\Gamma', \K_\agi \varphi^\u, \bot^\u \Rightarrow \Delta$}
        \RightLabel{${\rightarrow}\mathsf{L}$}
        \BinaryInfC{$\Gamma', \K_\agi \varphi^\u, \K_\agi \varphi \rightarrow \bot^\u \Rightarrow \Delta$}
    \end{prooftree}
    where $\Gamma = \Gamma' \cup \{\K_\agi \varphi^\u, \K_\agi \varphi \rightarrow \bot^\u\}$. For the other direction suppose towards contradiction that $\neg \K_\agi \varphi \in \Delta^-$ and $\K_\agi \varphi \not \in \Gamma^-$. By saturation $\K_\agi \varphi \in \Delta^-$. We assume (without loss of generality) that $\K_\agi \varphi$ occurs unfocused in $\Delta$. Then $\Gamma \Rightarrow \Delta$ is $\Sigma$-provable as follows:
    \begin{prooftree}
        \AxiomC{}
        \RightLabel{$\mathsf{id}$}
        \UnaryInfC{$\Gamma, \K_\agi \varphi^\u \Rightarrow \K_\agi \varphi^\u, \bot^\u, \Delta'$}
        \RightLabel{${\rightarrow}\mathsf{\K_\agi}$}
        \UnaryInfC{$\Gamma \Rightarrow \K_\agi \varphi^\u, \K_\agi \varphi \rightarrow \bot^\u, \Delta'$}
    \end{prooftree}
    where $\Delta = \Delta' \cup \{ \K_\agi \varphi^\u, \K_\agi \varphi \rightarrow \bot^\u\}$. \smallskip
    
    \noindent The proof of 5. is similar, using saturation and the rules for common knowledge.
\end{proof}

The canonical model relative to $\Sigma$ is defined as follows.

\begin{definition}
    Let $\Sigma$ be a non-empty, negation closed and finite set of formulae. The \emph{canonical model} (relative to $\Sigma)$ is given by $\fw^\Sigma =(\W^\Sigma, \leq^\Sigma,\rel^\Sigma, \V^\Sigma)$ where
    \begin{itemize}
		\item $\W^\Sigma \coloneqq \{\Gamma^- \mid \Gamma \Rightarrow \Delta \text{ is a $\Sigma$-saturated sequent}\}$;
        \item $\Gamma^- \leq^\Sigma \Pi^-$ if and only if $\Gamma^- \subseteq \Delta^-$ for any $\Gamma^-,\Pi^- \in \W^\Sigma$;
        \item $\rel^\Sigma \coloneqq \{R_\agi^\Sigma \mid \agi \in \A\}$ where for each $\agi \in \A$, $R_\agi^\Sigma \subseteq \W^\Sigma \times \W^\Sigma$ given by 
        \begin{center}
            $\Gamma^- \Rel_\agi^\Sigma \Pi'$ iff $\K_\agi \K_\agi^{-1} \Gamma^- = \K_\agi \K_\agi^{-1} \Pi^-$,
        \end{center} 
        for any $\Gamma^-, \Pi^- \in \W$;
        \item $\V^\Sigma(\Gamma^-) := \Gamma^- \cap \Prop$. 
		\end{itemize}

\end{definition}

\begin{lemma}\label{l: canonical model is S5}
    The canonical model is an S5 model.
\end{lemma}
\begin{proof}
    That $(\W^\Sigma, \leq^\Sigma)$ is a partial order and each $\R_\agi^\Sigma$ is reflexive, transitive and symmetric is immediate. Furthermore, by definition of $\V^\Sigma$ and $\leq^\Sigma$ the valuation is monotone. It remains to show that each $\R_\agi^\Sigma$ is triangle confluent. Suppose that $\Gamma^- \leq^\Sigma \Pi^-$ and $\Pi^- \mathrel{R_\agi^\Sigma} \Omega^-$. By definition $\K_\agi \K_\agi^{-1} \Pi^- = \K_\agi \K_\agi^{-1}\Omega^-$. Moreover, $\K_\agi \K_\agi^{-1} \Gamma^- \subseteq \K_\agi \K_\agi^{-1} \Pi^-$. Suppose towards contradiction that there exists a formula $\K_\agi \varphi \in  \Pi^-$ such that $\K_\agi \varphi \not \in \Gamma^-$. Let $\Gamma \Rightarrow \Delta$ be a saturated sequent. By saturation $\K_\agi \varphi \in \Delta^-$. By Lemma \ref{l: properties of saturated sequents}, $\neg \K_\agi \varphi \in \Gamma^-$. Hence $\neg \K_\agi \varphi \in \Pi^-$, contradicting that $\K_\agi \varphi \in \Pi^-$. Therefore $\K_\agi \K_\agi^{-1} \Gamma^- = \K_\agi \K_\agi^{-1} \Pi^- = \K_\agi \K_\agi^{-1} \Omega^-$, implying that $\Gamma^- \mathrel{R_\agi^\Sigma} \Omega^-$.
\end{proof}

\begin{lemma}[Truth Lemma]\label{l: truth lemma for S5}
    For any $\varphi \in \Sigma$ and any $\Gamma^- \in \W^\Sigma$: $\varphi \in \Gamma^-$ if and only if $\fw^\Sigma, \Gamma^- \models \varphi$. 
\end{lemma}

\begin{proof}
    By induction on the structure of $\varphi$. For $\varphi = \bot$ observe that if $\bot \in \Gamma^-$ then for any $\Delta \subseteq \Sigma$, the sequent $\Gamma \Rightarrow \Delta$ is an instance of the axiom $\bot$ and therefore $\Sigma$-provable. Hence $\bot \not \in \Gamma^-$ and by definition $\fw^\Sigma, \Gamma^- \not \models \bot$. For $\varphi = p$ where $p \in \Prop$, observe that $p \in \Gamma^-$ if and only if $p \in \V^\Sigma(\Gamma^-)$ if and only if $\fw^\Sigma, \Gamma^- \models p$. The cases for $\varphi = \psi \wedge \chi$ and $\varphi = \psi \vee \chi$ follow directly from Lemma \ref{l: properties of saturated sequents} and the induction hypothesis. \smallskip
    
    \noindent \textsc{Case for $\varphi = \psi \rightarrow \chi$.} First suppose that $\psi \rightarrow \chi \in \Gamma^-$. We want to show that $\fw^\Sigma, \Gamma^- \models \psi \rightarrow \chi$. Let $\Pi^- \in \W^\Sigma$ be any world such that $\Gamma^- \leq^\Sigma \Pi^-$. Hence $\Gamma^- \subseteq \Pi^-$ and so $\psi \rightarrow \chi \in \Pi^-$. Suppose $\fw^\Sigma, \Pi^- \models \psi$. By induction hypothesis $\psi \in \Pi^-$. Let $\Pi \Rightarrow \Omega$ be $\Sigma$-saturated and suppose towards contradiction that $\chi \in \Omega^-$. We may assume without loss of generality that $\chi$ occurs unfocused in $\Omega$ (otherwise apply $\mathsf{u}$). Write $\Pi_0$ for $\Pi \setminus \{\psi^\u, \psi \rightarrow \chi^\u\}$ and $\Omega_0$ for $\Omega \setminus \{\chi^\u\}$. The following is a $\Sigma$-proof of $\Pi \Rightarrow \Omega$, which contradicts the assumption that $\Pi \Rightarrow \Omega$ is $\Sigma$-saturated:
    \begin{prooftree}
        \AxiomC{}
        \RightLabel{$\mathsf{id}$}
        \UnaryInfC{$\Pi_0, \psi^\u \Rightarrow \psi^\u, \chi^\u, \Omega_0 $}
        \AxiomC{}
        \RightLabel{$\mathsf{id}$}
        \UnaryInfC{$\Pi_0, \psi^\u, \chi^\u \Rightarrow  \chi^\u, \Omega_0 $}
        \RightLabel{${\rightarrow} \mathsf{L}$}
        \BinaryInfC{$\Pi_0, \psi^\u, \psi \rightarrow \chi^\u \Rightarrow \chi^\u, \Omega_0$} 
    \end{prooftree}
    Therefore $\chi \in \Pi^-$. The induction hypothesis now yields that $\fw^\Sigma, \Pi^- \models \chi$, implying that $\fw^\Sigma, \Gamma^- \models \psi \rightarrow \chi$. 
    
    For the other direction suppose $\psi \rightarrow \chi \not \in \Gamma^-$. Let $\Gamma \Rightarrow \Delta$ be $\Sigma$-saturated. By saturation $\psi \rightarrow \chi \in \Delta^-$. Apply the rule ${\rightarrow}\mathsf{R}$ to $\psi \rightarrow \chi$. The premise of this application is $\Gamma, \psi^u \Rightarrow \chi^u$ which must be $\Sigma$-unprovable. Lemma \ref{l: Lindenbaum for saturated sequents} implies that there exists a $\Sigma$-saturated sequent $\Pi \Rightarrow \Omega$ with $\Gamma, \psi^u \subseteq \Pi$ and $\chi \in \Omega^-$. By induction hypothesis $\fw^\Sigma, \Pi^- \models \psi$ and ${\fw^\Sigma, \Pi^- \not \models \chi}$.  Since $\Gamma^- \leq^\Sigma \Pi^-$ it follows that $\fw^\Sigma, \Gamma^- \not \models \psi \rightarrow \chi$.  \smallskip
    
    \noindent \textsc{Case for $\varphi = \K_\agi \psi$.} First suppose that $\K_\agi \psi \in \Gamma^-$. Let $\Pi^- \in \W^\Sigma$ be any world such that $\Gamma^- \mathrel{R^\Sigma_\agi} \Pi^-$. By definition $\K_\agi \psi \in \Pi^-$ and so by Lemma \ref{l: properties of saturated sequents}, $\psi \in \Pi^-$. Hence, by induction hypothesis, $\fw^\Sigma, \Pi^- \models \psi$, implying that $\fw^\Sigma, \Gamma^- \models \K_\agi \psi$.
    
    For the other direction suppose $\K_\agi \psi \not \in \Gamma^-$. Let $\Gamma \Rightarrow \Delta$ be $\Sigma$-saturated. By saturation $\K_\agi\psi^\an \in \Delta$ for some $\an \in \{\u, \f \}$. Let $\Delta_0 = \Delta \setminus \{\K_\agi \psi^\an \}$. Observe that the sequent
    \begin{equation*}
        \K_\agi \K_\agi^{-1} \Gamma \Rightarrow \psi^\an, \K_\agi \K_\agi^{-1} \Delta_0
    \end{equation*}
    is $\Sigma$-unprovable, as otherwise, by applying $\mathsf{S5_\agi}$, we could derive $\Gamma \Rightarrow \Delta$. By Lemma~\ref{l: Lindenbaum for saturated sequents} there exists a saturated sequent $\Pi \Rightarrow \Omega$ with $\K_\agi\K_\agi^{-1}\Gamma \subseteq \Pi$ and $ \{\psi^\an \} \cup \K_\agi \K_\agi^{-1} \Delta_0 \subseteq \Omega$. By construction $\psi \not \in \Pi^-$ and so the induction hypothesis yields $\fw^\Sigma, \Pi^- \not \models \psi$. It remains to show that $\Gamma^- \mathrel{R_\agi^\Sigma} \Pi^-$. By construction $\K_\agi \K_\agi^{-1} \Gamma^- \subseteq \K_\agi \K_\agi^{-1} \Pi^-$. For the other direction suppose $\K_\agi \gamma \in \Pi^-$. Then $\K_\agi \gamma \not \in \Omega^-$, and hence in particular $\K_\agi \gamma \not \in \Delta^-$. Thus, by saturation, $\K_\agi \gamma \in \Gamma^-$. Therefore $\Gamma^- \mathrel{R_\agi^\Sigma} \Pi^-$ and so $\fw^\Sigma, \Gamma^- \not \models \K_\agi \psi$.\smallskip
    
    \noindent \textsc{Case for $\varphi = \C \psi$.} First suppose that $\C \psi \in \Gamma^-$. We want to show that $\fw^\Sigma, \Pi^- \models \psi$ for any $\Pi^- \in \W^\Sigma$ such that $\Gamma^- \mathrel{(R^\Sigma)^*} \Pi^-$. We prove by induction on the length $n$ of the path from $\Gamma^-$ to $\Pi^-$ that $\psi \in \Pi^-$ and $\C \psi \in \Pi^-$. For the base case by assumption $\C \psi \in \Gamma^-$. Moreover, by Lemma \ref{l: properties of saturated sequents}, $\psi \in \Gamma^-$. For the induction step suppose that $\Gamma^- \mathrel{(R^\Sigma)^n} \Omega^-$ and $\Omega^- \mathrel{R_\agi^\Sigma} \Pi^-$ for some $\agi \in \A$. By induction hypothesis $\psi \in \Omega^-$ and $\C \psi \in \Omega^-$. By Lemma~\ref{l: properties of saturated sequents} also $\K_\agi \C \psi \in \Omega^-$. Hence, by definition of $R_\agi^\Sigma$, we have that $\K_\agi \C \psi \in \Pi^-$. Lemma~\ref{l: properties of saturated sequents} implies that $\C \psi \in \Pi^-$ and $\psi \in \Pi^-$. Hence $\psi \in \Pi^-$ for any $\Pi^- \in \W^\Sigma$ with $\Gamma^- \mathrel{(R^\Sigma)^*} \Pi^-$. By induction hypothesis $\fw^\Sigma, \Pi^- \models \psi$ and therefore $\fw^\Sigma, \Gamma^- \models \C \psi$. \smallskip
    
	For the other direction suppose $\C \psi \notin \Gamma^-$ and consider a $\Sigma$-saturated sequent $\Gamma \Rightarrow \Delta$.  By the presence of the rules $\mathsf{u}$ and $\mathsf{f}$, we may assume without loss of generality that $\C \psi^\f \in \Delta$. Now suppose, towards a contradiction, that $\fw^\Sigma, \Gamma^-  \models \C \psi$. For every $\Pi^- \in \W^\Sigma$ with $\Gamma^- \mathrel{(R^\Sigma)^*} \Pi^-$ then holds that $\fw^\Sigma, \Pi^- \models \psi$. In particular, it follows that $\fw^\Sigma, \Gamma^- \models \psi$, hence, by induction hypothesis, we have $\psi^\u \in \Gamma$.
			
	Let $\Delta_0 = \Delta \setminus \{\C \psi^\f\}$. Consider the following proof, where we assume without loss of generality that there are $n$ agents $\agi_1, \ldots, \agi_n$:
			\begin{prooftree}
			\AxiomC{}
			\RightLabel{$\mathsf{id}$}
			\UnaryInfC{$\Gamma \Rightarrow \psi^\u, \Delta_0$}
			\noLine
			\AxiomC{$\pi_1$}
			\UnaryInfC{$\Gamma \Rightarrow \K_{\agi_1} \C \psi^\f, \Delta_0$}
			\AxiomC{}
			\noLine
			\UnaryInfC{$\cdots$}
			\AxiomC{$\pi_n$}
			\noLine
			\UnaryInfC{$\Gamma \Rightarrow \K_{\agi_n} \C \psi^\f, \Delta_0$}
			\RightLabel{$\mathsf{\C R}$}
			\QuaternaryInfC{$\Gamma \Rightarrow \C \psi^\f, \Delta_0$}
            \end{prooftree}
			 where each $\pi_i$ is constructed as follows:
			\begin{prooftree}
			\AxiomC{$\pi'$}
			\noLine
			\UnaryInfC{$\sigma'$}
			\AxiomC{$\pi'_1$}
			\noLine
			\UnaryInfC{$\sigma'_1$}
			\AxiomC{$\cdots$}
			\noLine
			\UnaryInfC{$\cdots$}
			\AxiomC{$\pi'_n$}
			\noLine
			\UnaryInfC{$\sigma'_n$}
			\RightLabel{$\mathsf{\C R}$}
			\QuaternaryInfC{$\K_{\agi_i} \K_{\agi_i}^{-1} \Gamma \Rightarrow \C \psi^\f, \K_{\agi_i} \K_{\agi_i}^{-1} \Delta_0$}
			\RightLabel{$\mathsf{S5_{\agi_i}}$}
			\UnaryInfC{$\Gamma \Rightarrow \K_{\agi_i} \C \psi^\f, \Delta_0$}
			\end{prooftree}
			In the above proof the sequent $\sigma'$ is given by
			\[
			\K_{\agi_i} \K_{\agi_i}^{-1} \Gamma \Rightarrow \psi^\u, \K_{\agi_i} \K_{\agi_i}^{-1} \Delta_0
			\]
			and the proof $\pi'$ is obtained from the $\Sigma$-provability of $\sigma'$. Indeed, if $\sigma'$ were not $\Sigma$-provable, then by applying Lemma \ref{l: Lindenbaum for saturated sequents} and the induction hypothesis, we would obtain a saturated sequent $\Pi \Rightarrow \Omega$ extending $\sigma'$ and such that $\fw^\Sigma, \Pi^- \not \models \psi$. First note that $\K_{\agi_i} \C \psi \not \in \Pi^-$, as otherwise $\psi \in \Pi^-$, implying that the sequent $\Pi \Rightarrow \Omega$ is an instance of $\mathsf{id}$ and hence $\Sigma$-provable. Therefore $\K_{\agi_i} \C \psi \in \Omega^-$. Note that $\Gamma^- \mathrel{R_{\agi_i}^\Sigma} \Pi^.$: by construction $\K_{\agi_i} \K_{\agi_i}^{-1} \Gamma^- \subseteq \K_{\agi_i} \K_{\agi_i}^{-1} \Pi^-$. Suppose towards contradiction that there exists a formula $\K_{\agi_i} \chi \in \Pi^-$ such that $\K_{\agi_i} \chi \not \in \Gamma^-$. Therefore, by saturation, $\K_{\agi_i} \chi \in \Delta^-$. Since $\chi \not = \C \psi$, $\K_{\agi_i} \chi \in \Delta_0^-$ and so $\K_{\agi_i} \chi \in \K_{\agi_i} \K_{\agi_i}^{-1} \Delta_0^-$.  Therefore $\K_{\agi_i} \chi \in \Omega^-$, implying that $\K_{\agi_i} \chi \not \in \Pi^-$, a contradiction. Therefore $\Gamma^- \mathrel{R_{\agi_i}^\Sigma} \Pi^-$. Since $\fw^\Sigma, \Pi^- \not \models \psi$, this contradicts the assumption that ${\fw^\Sigma, \Gamma^- \models \C \psi}$. Hence, $\sigma'$ must be $\Sigma$-provable. \smallskip
			
		    Furthermore, each sequent $\sigma'_l$ for $ l \in [n+1]$ in the derivation $\pi_i$ is given by
			\[
			  \K_{\agi_i} \K_{\agi_i}^{-1} \Gamma \Rightarrow \K_{\agi_l} \C \psi^\f, \K_{\agi_i} \K_{\agi_i}^{-1} \Delta_0
			\]
			and each derivation $\pi'_l$ is constructed by repeatedly applying $\mathsf{cut}$ to add formulae from $\Sigma$ until every leaf is either saturated or $\Sigma$-provable. To the leaves that are $\Sigma$-provable we append their respective proofs. Suppose $\Pi \Rightarrow \K_{\agi_l} \C \psi^\f, \Omega$ is a saturated leaf. Note that by construction $\K_{\agi_i} \K_{\agi_i}^{-1} \Gamma = \K_{\agi_i} \K_{\agi_i}^{-1} \Pi$ and therefore $\Gamma^- \mathrel{R_{\agi_i}^\Sigma} \Pi^-$. The assumption $\fw^\Sigma, \Gamma^- \models \C \psi$ therefore entails that $\fw^\Sigma, \Pi^- \models \C \psi$ and thus $\fw^\Sigma, \Pi^- \models \psi$ which, by induction hypothesis, implies that $\psi \in \Pi^-$. Hence we can apply the same process to $\Pi \Rightarrow \K_{\agi_l} \C \psi^\f, \Omega$ as we have just applied to $\Gamma \Rightarrow \K_{\agi_i} \C \psi^\f, \Delta_0$. \smallskip
			
			Since $\Sigma$ is finite, there are only finitely many distinct $\Sigma$-saturated sequents. This entails, by the pidgeonhole principle, that at some point one of the saturated leaves obtained from our construction must be identical to a saturated leaf reached earlier in the construction. Note that in this case the upward path from the earlier saturated leaf to the later one is successful. We then terminate the construction of this branch. Since every branch is terminated at some point, we end up with a $\Sigma$-proof of $\Gamma \Rightarrow \Delta$, a contradiction. Hence $\fw^\Sigma, \Gamma^- \not \models \C \psi$, which finishes the proof.
\end{proof}

\begin{theorem}[Analytic Completeness]\label{t: completeness of S5}
     If a sequent $\sigma$ is valid over the class of S5 frames, then $\sigma$ has an analytic proof in $\cickss$.
\end{theorem}
\begin{proof}
    We prove the contrapositive. Let $\Sigma = \Cl^\neg(\sigma)$ and suppose that $\sigma $ is not $\Sigma$-provable. By Lemma \ref{l: Lindenbaum for saturated sequents} there exists a saturated sequent $\Gamma \Rightarrow \Delta$ such that $\Gamma_\sigma \subseteq \Gamma$ and $\Delta_\sigma \subseteq \Delta$. By Lemma \ref{l: canonical model is S5} the canonical model $\fw^{\Sigma}$ is an S5 model. Furthermore, by definition $\Gamma^- \in \W^\Sigma$. By the Truth Lemma we have that $\fw^\Sigma, \Gamma^- \models \varphi$ for each $\varphi \in \Gamma_\sigma$ and $\fw^\Sigma, \Gamma^- \not \models \psi$ for each $\psi \in \Delta_\sigma$. Hence, $\fw^\Sigma, \Gamma^- \not \models \sigma^I$, implying that $\sigma$ is not valid over the class of S5 frames.
\end{proof}

\section{Automated proof-search}\label{s: automated proof-search}

 The last section established that for every valid sequent \emph{there exists} a cyclic proof. In this section we are going to show \emph{how to find} them. We will prove that the cyclic calculi are amenable for automated proof-search. To that achieve this we will show how to reduce the problem of finding a cyclic proof for a sequent to the problem of solving a certain game. As for the proof-search games employed in the completeness proof of the non-wellfounded calculi, this game is played two players, one called Prover and the other called Refuter. Intuitively, given a sequent $\sigma$, Prover is trying to build a cyclic proof of $\sigma$, while Refuter tries to show that such a proof does not exist. In difference to the proof-search games employed before, the games here are \emph{parity games} and not Gale--Steward games. Crucially, parity games are \emph{memoryless determined}, meaning that exactly one of the two players has a memoryless winning strategy (see e.g.~\cite[Chapter 6]{graedel_2001}). From this we will be able to obtain that our cyclic calculi are \emph{uniformly complete}, c.f. Theorem \ref{t: strategy iff provable}. Furthermore, efficient algorithms exist to solve parity games, i.e. algorithms that given a parity game compute the winner as well as the winning strategy. Our games are set up in such a way that winning strategies for Prover correspond directly to cyclic proofs, implying that the algorithm either constructs a proof for a given sequent or determines that the sequent is not provable. The results of this section will then also be used to establish an exponential upper bound for the validity problem for all considered logics in Section \ref{s: complexity}, by using known complexity bounds for solving parity games from the literature. Furthermore, using the translation of $\mathbf{CKS5}$ into $\mathbf{ICKS5}$ as a reduction, we establish a precise complexity bound for $\mathbf{ICKS5}$: its validity problem is \textsc{ExpTime}-complete. We assume familiarity with parity games. For an introduction, the reader is referred to~\cite[Chapter 2]{graedel_2001}. The presented material closely follows the joint publication with Rooduijn~\cite[Section 6]{rooduijn_analytic_2022}, where a proof-search game for $\mathbf{CKS5}$ is defined. The presented proof there contains a mistake, which is fixed here.

\subsection{A parity game for proof-search}

In this subsection we set up the parity games for proof-search in our cyclic calculi. The next subsection then shows how to reduce the problem of finding a cyclic proof for a given sequent to the corresponding parity game.
 
\begin{definition}
    Let $\hickp \in \{\cick, \cickt, \cicks, \cickss\}$ and let $\mathsf{r} \in \hickp$ be a rule. A \emph{rule position} in $\hickp$ is a triple $i =\langle \sigma, \mathsf{r}, \langle \sigma_1, \ldots, \sigma_n\rangle \rangle$ such that
    \begin{equation*}
        \infer[\mathsf{r}]{\sigma}{\sigma_1 & \ldots & \sigma_n}
    \end{equation*}
       is a rule instance of $\mathsf{r}$. Given a finite set of formulae $\Sigma$, a \emph{$\Sigma$-rule position} is a rule position involving only $\Sigma$-sequents.
\end{definition}

  For the remainder of this section $\Sigma$ always denotes a finite set of formulae which is closed or, if we consider $\cickss$, negation closed. We use $\pi_n^k$ to denote the projection function which takes as input an $n$-tuple and outputs the k-th component.

	\begin{definition}
		Let $\hickp \in \{\cick, \cickt, \cicks, \cickss\}$, $\Sigma$ be a finite and closed (negation closed) set of formulae and let $\sigma$ be a $\Sigma$-sequent. The \emph{proof-search game} $\mathscr{G}(\hickp,\sigma)$ associated to $\sigma$ and $\hickp$ takes positions in $S \cup I$, where $S$ is the set of $\Sigma$-sequents and $I$ is the set of $\Sigma$-rule positions in $\hickp$. The ownership function and admissible moves are as described in the following table:
		\begin{center}
			\begin{tabular}{|c|c|c|}
				\hline
				Position & Owner & Admissible moves \\
				\hline
				$\sigma$ & Prover & $\{i \in I \mid \pi_3^1(i) = \sigma\}$ \\
				$\langle \sigma, \mathsf{r}, \langle \sigma_1, \ldots, \sigma_n \rangle\rangle$ & Refuter & $\{\sigma_i \mid 1 \leq i \leq n\}$\\
				\hline 
			\end{tabular}
		\end{center}
		The positions are given the following priorities:
		\begin{enumerate}
			\item Every position of the form $\Gamma \Rightarrow \Delta^\u$ has priority $3$;
			\item Every position of the form $\langle\sigma, \mathsf{C R}, \langle \sigma_1, \ldots, \sigma_n \rangle\rangle$ where the principal formula is in focus has priority $2$;
			\item Every other position has priority $1$.
		\end{enumerate}
		A position is called a \emph{dead end} if its owner has no admissible moves in this position available. A \emph{play} in $\mathscr{G}(\hickp,\sigma)$ is a sequence of positions starting in $\sigma$, such that any two consecutive positions are related by an admissible move. A play is either finite and ends in a dead end or infinite. The \emph{winning conditions} are as follows: Prover wins every finite play in which the dead end belongs to Refuter and every infinite play in which the highest priority encountered infinitely often is even. Refuter wins every finite play in which the dead end belongs to Prover and every infinite play in which the highest priority encountered infinitely often is odd.
	\end{definition}

Observe that the only dead ends are rule instances of axioms. Therefore Prover wins every finite play. It is straightforward to check that for every $\Sigma$-sequent $\sigma$ and every considered calculus $\hickp$, the game $\mathscr{G}(\hickp,\sigma)$ is a parity game. Given a set $X$, let $X^n$ denote the cartesian product 
\begin{equation*}
   \underbrace{X \times X \times ... \times X}_{n \text{ times}} 
\end{equation*} 
and define $X^{< \omega} \coloneqq \bigcup_{n < \omega} X^n$. Therefore the set of all initial segments of all possible plays in a game is a \emph{subset} of $(S \cup I)^{< \omega}$. The following definitions apply to both Prover and Refuter. We write `Player' instead.

  \begin{definition}
      Let $\mathscr{G}(\hickp,\sigma)$ be a proof-search game with positions $S \cup I$. 
      \begin{enumerate}
          \item A \emph{strategy}\index{strategy} for Player is partial function $\mathcal{S}: (S \times I)^{< \omega} \longrightarrow S \cup I$ which maps each initial segment of a play ending in a position owned by Player onto an admissible move if such a move exist.
          \item A \emph{memoryless strategy}\index{strategy!memoryless} for Player is a partial function $\mathcal{S}: S \cup I \longrightarrow S \cup I$ which maps each position owned by Player onto an admissible move if such a move exist.
      \end{enumerate}
  \end{definition}

A strategy $\mathcal{S}$ for Player is \emph{winning} if Player wins every play in which $\mathcal{S}$ is \emph{used}. The \emph{strategy tree} of a memoryless strategy $\mathcal{S}$ is the tree of all possible plays that can occur when Player uses $\mathcal{S}$, formally defined as follows.

\begin{definition}
    Let $\sigma$ be a $\Sigma$-sequent and $\mathscr{G}(\hickp,\sigma)$ be the proof-search game associated to $\sigma$ and $\hickp$. Let $\mathcal{S}$ be a memoryless strategy for Player. The \emph{strategy tree} $\mathcal{T}_\mathcal{S}$ is the tree whose nodes are labeled by elements of $(S \cup I)$ defined as follows.
    \begin{enumerate}
        \item The root of $\mathcal{T}_\mathcal{S}$ is labeled by $\sigma$.
        \item If a node $t$ of $\mathcal{T}_\mathcal{S}$ is labeled by $p \in S \cup I$, $p$ is owned by Player and $p \in dom(\mathcal{S})$, then $t$ has a unique child $u$ labeled by $\mathcal{S}(p)$. If $p \not \in dom(\mathcal{S})$, then $p$ is a leaf.
        \item If a node $t$ of $\mathcal{T}_\mathcal{S}$ is labeled by $p \in S \cup I$ which is not owned by Player, then 
        \begin{enumerate}
            \item if $p \in S$ and $I_p = \{i \in I \mid \pi_3^1(i)=p\}$, then $t$ has $\lvert{I_p}\rvert$ children, each labeled with a different $i \in I_p$.
            \item if $p \in I$ and $\pi_3^3(i) = \langle \sigma_1, \ldots, \sigma_n\rangle$, then $t$ has $n$ children each labeled by a different $\sigma_i$.
        \end{enumerate}
    \end{enumerate}
\end{definition}

Note that due to $\Sigma$ being finite and rules having finitely many premises, strategy trees are finite branching (and possibly non-wellfounded). The goal is to show that a sequent $\sigma$ has a $\hickp$-proof if and only if Prover has a memoryless winning strategy in the game $\mathscr{G}(\hickp,\sigma)$. To that end we require to do some preliminary work. Recall that in a cyclic proof $\pi$ there exists for every non-axiomatic leaf $l$ a node $u$ such that $(u, l)$ is a successful repetition. As there might exist several candidates for the node $u$, we fix for each non-axiomatic leaf $l$ one candidate $c(l)$ which we call the \emph{companion} of $l$.
      
      \begin{definition}
          A \emph{finite tree with back-edges} is a pair $(\mathcal{T},f)$ where $\mathcal{T}$ is a finite tree and $f: \mathcal{T} \longrightarrow \mathcal{T}$ a partial function, such that every $u \in dom(f)$ is a leaf of $\mathcal{T}$ and the node $f(u)$ is a proper ancestor of $u$.\footnote{By proper ancestor we mean that $f(u)$ occurs on the branch from the root to $u$ and $f(u) \not = u$.}
      \end{definition}

     Observe that cyclic proofs can be considered to be finite trees with back-edges, that satisfy the property that if $u \in dom(f)$, then $u$ is a non-axiomatic leaf and $f(u)$ is the companion of $u$. We introduce two notions of `path' through a finite tree with back-edges.

     \begin{definition}
         Let $(\mathcal{T}, f)$ be a finite tree with back-edges. An \emph{upward path} $\tau$ through $(\mathcal{T},f)$ is a finite sequence $\tau = \tau(0), \tau(1),\ldots, \tau(n)$ of nodes in $\mathcal{T}$ such that for each $i \in [n]$ $\tau(i+1)$ is a child of $\tau(i)$.
	  \end{definition}

Given an upward path $\tau=\tau(0), \ldots, \tau(n)$ observe that $\tau(0)$ may be any node of $\mathcal{T}$ and $\tau(n)$ is an ancestor of $\tau(0)$. If $n > 0$, then we call $\tau$ \emph{non-trivial}.
	
	  \begin{definition}
	      Let $(\mathcal{T}, f)$ be a finite tree with back-edges. A \emph{looping path} $\rho$ through $(\mathcal{T},f)$ is a (possibly infinite) sequence $\rho = \rho(0), \rho(1),\ldots$ of nodes in $\mathcal{T}$ which satisfies the following properties:
	      \begin{enumerate}
	         \item $\rho(0)$ is the root of $\mathcal{T}$.
	         \item If $\rho(i)$ is a leaf of $\mathcal{T}$ and $\rho(i) \not \in dom(f)$, then $\rho$ is finite and ends at $\rho(i)$.
	          \item If $\rho(i)$ is a leaf of $\mathcal{T}$ and $\rho(i) \in dom(f)$, then $\rho(i+1)$ is a child of $f(\rho(i))$.
	          \item Otherwise, $\rho(i+1)$ is a child of $\rho(i)$.
	      \end{enumerate}
	  \end{definition}
	  
	   If condition 3. applies, then we say that $\rho$ \emph{passes through the leaf} $\rho(i)$. Looping paths may thus be infinite and loop through $\mathcal{T}$ via the back-edge function. Moreover, they are maximal in the sense that they begin at the root and are either infinite or end in a leaf that does not belong to $dom(f)$. We write $\tau$ for upward paths and $\rho$ for looping paths. The following lemma states a first basic result about infinite looping paths through finite trees with back-edges. The proof of the lemma follows immediately from the fact that such trees are finite.

	  \begin{lemma}\label{l: infinite path passes through leaf}
	      Suppose $(\mathcal{T},f)$ is a finite tree with back-edges and $\rho$ is an infinite looping path through $(\mathcal{T},f)$. Then there exists a leaf $u \in dom(f)$ through which $\rho$ passes infinitely often.
	  \end{lemma}

      Next, we define a preorder on the domain of the back-edge function $f$.

      \begin{definition}
          Let $(\mathcal{T},f)$ be a finite tree with back-edges. Define the \emph{dependency order} $\preceq$ on $dom(f)$ as follows:
          \begin{center}
              $u \preceq v$ if and ony if there exists an upward path from $f(v)$ to $f(u)$.
          \end{center}
      \end{definition}
      
      It is routine to check that the dependency order $\preceq$ is reflexive and transitive: $u \preceq u$ since $\tau = u$ is an upward path (a trivial one); moreover if $u_0 \preceq u_1 \preceq u_2$, then there are upward paths $\tau_0$ from $f(u_1)$ to $f(u_0)$ and $\tau_1$ from $f(u_2)$ to $f(u_1)$ and so $\tau_1, \tau_0$ is an upward path from $f(u_2)$ to $f(u_0)$ implying that $u_0 \preceq u_2$. Note that $\preceq$ is not antisymmetric, since two leafs might have the same companion. Hence $\preceq$ is a preorder on $dom(f)$. For a finite tree with back-edges $(\mathcal{T}, f)$ and an infinite looping path $\rho$, we denote by $\rho^{\infty}$ the set of nodes of $\mathcal{T}$ that occur infinitely often in $\rho$. A leaf $u \in dom(f) \cap \rho^\infty$ is called \emph{$\preceq$-maximal} if for all $v \in dom(f) \cap \rho^\infty$ whenever $u \preceq v$, then $v \preceq u$. Observe that $dom(f) \cap \rho^\infty$ may have multiple maximal elements. We will now show that every maximal element has the same companion, implying that this companion is lowermost among all companions in $dom(f) \cap \rho^\infty$.
      
      \begin{lemma}\label{l: existence of greatest element}
          Let $(\mathcal{T},f)$ be a finite tree with back-edges and let $\rho$ be an infinite looping path through $(\mathcal{T},f)$. Then every $\preceq$-maximal element of $\rho^{\infty} \cap dom(f)$ has the same companion.
      \end{lemma}
      
      \begin{proof}
          Observe that the set $\rho^{\infty} \cap dom(f)$ is finite since $\mathcal{T}$ is a finite tree. Furthermore, $\rho^{\infty} \cap dom(f)$ is non-empty, as $\rho$ must pass through some leaf $u \in dom(f)$ infinitely often (see Lemma \ref{l: infinite path passes through leaf}). Let $u \in dom(f) \cap \rho^\infty$ be $\preceq$-maximal and consider the subtree $\mathcal{T}_u$ of $\mathcal{T}$ rooted at $f(u)$. By maximality of $u$, note that there does not exist a maximal leaf $v \in dom(f) \cap \rho^\infty$ such that $f(v)$ lies strictly below $f(u)$ in $\mathcal{T}$, as otherwise there would be an upward path from $f(v)$ to $f(u)$ contradicting that $u$ is $\preceq$-maximal. So suppose $v \in dom(f) \cap \rho^\infty$ is $\preceq$-maximal. Then either $f(v)$ belongs to $\mathcal{T}_u$ or $\mathcal{T}_u$ and $\mathcal{T}_v$ are distinct. In the former case, there exists an upward path from $f(u)$ to $f(v)$ and so $v \preceq u$. But since $v$ is $\preceq$-maximal, it follows that also $u \preceq v$, which implies that $f(v) = f(u)$. In the latter case, since $\mathcal{T}_u$ and $\mathcal{T}_v$ are distinct subtrees and $u,v \in \rho^\infty$, there must exist a leaf $w \in dom(f)$ such that there are non-trivial upward paths $\tau_0$ and $\tau_1$ from $f(w)$ to $f(u)$ and from $f(w)$ to $f(v)$ respectively, as otherwise, once $\rho$ enters either $\mathcal{T}_u$ or $\mathcal{T}_v$ it can never pass through the other leaf again. Since $\rho$ passes through $u$ and $v$ infinitely often and $dom(f)$ is finite, $\rho$ must pass through $w$ infinitely often as well, implying that $w \in dom(f) \cap \rho^\infty$. Since $u \preceq w$, $v \preceq w$ and $\tau_0$ and $\tau_1$ are non-trivial, it holds that $w \not \preceq u$ and $w \not \preceq v$, contradicting that $u$ and $v$ are maximal. Hence the latter case cannot happen and so all $\preceq$-maximal leafs have the same companion.
      \end{proof}
      
We use Lemma~\ref{l: existence of greatest element} to show that every infinite looping path through a cyclic proof contains a suffix in which every sequent has a formula in focus.

     \begin{lemma}\label{l: infinite looping path has good suffix}
         Suppose $\pi$ is a $\hickp$-proof for $\hickp \in \{\cick, \cickt, \cicks, \cickss\}$ and $\rho$ is an infinite looping path through $\pi$. Then there exists a suffix $\rho'$ of $\rho$ in which every sequent has a formula in focus.
     \end{lemma}
     \begin{proof}
         For each non-axiomatic leaf $u$ of $\pi$ chose one companion $c(u)$ of $u$ and define the back-edge function of $\pi$ to be $f(u) = c(u)$. Let $\rho$ be an infinite looping path through $(\pi, f)$. By Lemma \ref{l: existence of greatest element} every $\preceq$-maximal leaf $u \in dom(f) \cap \rho^\infty$ has the same companion $f(u)$. This implies that the following hold.
         \begin{enumerate}
             \item $\rho$ passes through $u$ infinitely often.
             \item If $dom(f) \cap \rho^{\infty} = \{u_0, \ldots, u_n\}$ with $u=u_0$, then there exists an upward path from $f(u_0)$ to $f(u_i)$ for all $i \in [n+1]$.
         \end{enumerate}
         Let $\pi_0$ be the subtree rooted at $f(u_0)$. We claim that every upward path from $f(u_0)$ to $u_i$ for ${i \in [n+1]}$ always has a formula in focus. The proof proceeds by induction on the cardinality of ${dom(f) \cap \rho^{\infty}}$. The base case where $dom(f) \cap \rho^{\infty} = \{u_0\}$ is trivial since the upward path from $f(u_0)$ to $u_0$ is successful by definition of a cyclic proof. For the induction step suppose that $dom(f) \cap \rho^{\infty} = \{u_0, \ldots, u_n, u_{n+1}\}$. Consider the upward path $\tau$ from $f(u_0)$ to $u_{n+1}$. If $f(u_{n+1})$ lies on the upward path from $f(u_0)$ to $u_0$, then every sequent occurring on the upward path from $f(u_0)$ to $f(u_{n+1})$ has a formula in focus and, by definition of a cyclic proof, also every sequent from $f(u_{n+1})$ to $u_{n+1}$. If the companion $f(u_{n+1})$ of $u_{n+1}$ does not lie on the upward path from $f(u_0)$ to $u_0$, then there must exist $i \in [n+1]$ such that $u_i$ belongs to the subtree $\pi_1$ of $\pi_0$ rooted at $f(u_{n+1})$ and $f(u_i)$ does not belong to $\pi_1$, as otherwise, $\rho$ cannot pass through both $u_0$ and $u_{n+1}$ infinitely often. By induction hypothesis the path from $f(u_0)$ to $u_i$ has always a formula in focus. Observe that this path must contain the path from $f(u_0)$ to $f(u_{n+1})$ as an initial segment, since $u_i$ belongs to $\pi_1$. Hence the path from $f(u_0)$ to $u_{n+1}$ always has a formula in focus.

         Let $i$ be the least natural number such that the suffix $\rho'$ of $\rho$ starting at $\rho(i)$ only passes through non-axiomatic leafs in $\rho^{\infty}$. Then, by the previous observation, $\rho'$ always has a formula in focus.
     \end{proof}

	 We will now show that from a cyclic $\hickp$-proof $\pi$ of $\sigma$ we can construct a winning strategy for Prover in the game $\mathscr{G}(\hickp, \sigma)$. Recall that a play $m$ in a proof-search game is a finite or infinite sequence of positions $m(0),m(1),m(2), \ldots$ and so on. Observe that all even positions $m(2i)$ are owned by Prover and all odd positions $m(2i+1)$ by Refuter. For an initial segment $m(0), \ldots, m(i)$ of $m$ we say that its \emph{length} is $i$.
	
	\begin{definition}
	Let $\hickp \in \{\cick, \cickt, \cicks, \cickss\}$, $\sigma$ be a $\Sigma$-sequent and let $\pi$ be a $\Sigma$-proof of $\sigma$ in $\hickp$. Let $\rho$ be an infinite looping path through $\pi$ and let $m$ be an infinite play in $\mathscr{G}(\hickp,\sigma)$.
	\begin{enumerate}
	    \item An initial segment $m(0), \ldots, m(i)$ of $m$ \emph{corresponds} to an initial segment $\rho(0), \ldots, \rho(j)$ of $\rho$ if $i = 2j$ and for each $k \in [i+1]$ with $k = 2l$ it holds that $m(k)$ is the sequent that labels $\rho(l)$.
	    \item The play $m$ and the path $\rho$ are called \emph{corresponding} if every initial segment of $m$ with even length corresponds to an initial segment of $\rho$.
	\end{enumerate}
	\end{definition}

     \begin{proposition}\label{p: provable implies winning strategy}
		Let $\hickp \in \{\cick, \cickt, \cicks, \cickss\}$ and $\sigma$ be a $\Sigma$-sequent. If $\sigma$ is $\Sigma$-provable in $\hickp$, then Prover has a memoryless winning strategy in $\mathscr{G}(\hickp,\sigma)$.
	\end{proposition}
	
	\begin{proof}
	 Suppose that $\sigma$ is $\Sigma$-provable in $\hickp$ and let $\pi$ be a $\hickp$-proof of $\sigma$ in which every occurring sequent is a $\Sigma$-sequent. We denote the root of $\pi$ by $r_\pi$. We simultaneously define a strategy $\mathcal{S}$ for Prover in the game $\mathscr{G}(\hickp,\sigma)$ and show how to map each initial segment of a play of even length in which Prover uses $\mathcal{S}$ onto an initial segment of a looping path through $\pi$. The strategy $\mathcal{S}$ is a partial function which maps initial segments of plays $ m(0),\ldots, m(2i) $ of even length onto rule positions. Therefore strategy $\mathcal{S}$ uses a \emph{memory}.\smallskip
	 
	 For the base case observe that each play in $\mathscr{G}(\hickp,\sigma)$ begins in $\sigma$. Therefore $ m(0)$ for $m(0)= \sigma$ is an initial segment of every play in $\mathscr{G}(\hickp,\sigma)$. Similarly, every looping path through $\pi$ starts in $r_\pi$ which is labeled by $\sigma$. Thus $ \rho(0)$ for $\rho(0) = r_\pi$ is an initial segment of every looping path through $\pi$. Observe that $ m(0) $ corresponds to $ \rho(0)$.\smallskip
	 
	 For the inductive step suppose that we have already mapped the initial segment $m_i =  m(0), \ldots, m(2i)$ of a play onto the initial segment $\rho_i =  \rho(0),  \ldots, \rho(i)$ of a looping path, such that $m_i$ corresponds to $\rho_i$, where $i \geq 0$. We distinguish three cases. First, suppose that $\rho(i)$ is an axiomatic leaf of $\pi$. Then let $j \in I$ be the $\Sigma$-rule position $j=\langle m(2i), \mathsf{r}, \langle \rangle \rangle$ where $\mathsf{r} \in \{\mathsf{id, \bot}\}$, depending on whether the sequent labeling $\rho(i)$ is an instance of $\mathsf{id}$ or of $\bot$ and define $\mathcal{S}(m_i) \coloneqq j$. Second, suppose that $\rho(i)$ is a not a leaf of $\pi$. Let $j \in I$ be the $\Sigma$-rule position $j = \langle m(2i), \mathsf{r}, \langle \sigma_1', \ldots, \sigma_k' \rangle\rangle$ which generates $\rho(i)$ in $\pi$ when read top down. Then define $\mathcal{S}(m_i) \coloneqq j$. Third, suppose that $\rho(i)$ is a non-axiomatic leaf of $\pi$ with companion $f(\rho(i))$. Let $j \in I$ be the $\Sigma$-rule $j = \langle \sigma', \mathsf{r}, \langle \sigma_1', \ldots, \sigma_k' \rangle\rangle$ which generates $f(\rho(i))$ when read top down, where $\sigma'$ is the sequent labeling $f(\rho(i))$, and define $\mathcal{S}(m_i) \coloneqq j$.
	
	 In the first case the play has a reached a dead end and therefore ends. In the other two cases suppose that Refuter extends the play by choosing premise $\sigma_l'$ for some $l \in [n+1]$. Then let $m(2i+1) = j$ and let $m(2i+2) = \sigma_l'$ and extend the initial segment $m_i$ to
	 \begin{equation*}
	     m_{i+1} =  m(0),  \ldots, m(2i), m(2i+1), m(2i+2)
	 \end{equation*}
	 Furthermore, let $\rho(i+1)$ be the child of $\rho(i)$ which is labeled by $\sigma_l'$ and extend $\rho_i$ to
	 \begin{equation*}
	     \rho_{i+1} = \rho(0), \ldots, \rho(i), \rho(i+1)
	 \end{equation*}
	 Observe that $m_{i+1}$ corresponds to $\rho_{i+1}$.\smallskip
	 
	 Finally, in order to turn $\mathcal{S}$ into a function which maps \emph{every} initial segment of a play with even length (and thus every initial segment ending in a position owned by Prover) onto a rule instance (and not just those that correspond to initial segments of looping paths), we add the following clause. For any initial segment $m_i' = m(0)', \ldots, m(2i)' $ of a play which is not covered in the above construction let $\mathcal{S}$ map $m_i'$ onto an arbitrary admissible move. Observe that $\mathcal{S}$ is a well-defined strategy for Prover, which has the property that if $m$ is a play of $\mathscr{G}(\hickp,\sigma)$ in which Prover uses $\mathcal{S}$, then there exists a looping path $\rho$ through $\pi$ such that every initial segment of $m$ of even length corresponds to some initial segment of $\rho$. Therefore if $m$ is infinite, then $\rho$ is infinite as well and $m$ corresponds to $\rho$.\smallskip
	 
	 We show that $\mathcal{S}$ is a winning strategy for Prover. To that end let $m$ be a play in $\mathscr{G}(\hickp,\sigma)$ in which Prover uses strategy $\mathcal{S}$. If $m$ is finite, then Prover wins by default. Suppose $m$ is infinite and let $\rho$ be the infinite looping path through $\pi$ to which $m$ corresponds. By Lemma~\ref{l: infinite looping path has good suffix}, $\rho$ has a suffix $\rho'$ in which every sequent has a formula in focus. Note that $\rho'$ passes through infinitely many rule instances of $\mathsf{\C R}$ where the principal formula is in focus: let $f(u)$ be the companion of all $\preceq$-maximal elements of $dom(f) \cap \rho^{\infty}$ (which exists by Lemma \ref{l: existence of greatest element}) where $f$ is the back-edge function of $\pi$. By assumption $\rho'$ passes infinitely often through the non-axiomatic leaf $u$. Between each two passings, $\rho'$ must pass through an instance of $\mathsf{\C R}$ with the principal formula in focus, since the path from $f(u)$ to $u$ is successful. Therefore, since $\rho$ corresponds to $m$, the play $m$ passes, after finitely many moves, only through positions with priorities 1 or 2, and infinitely often through positions with priority 2. Hence, the highest priority encountered infinitely often is even and Prover wins the play. We conclude that $\mathcal{S}$ is a winning strategy for Prover. Finally, since in a given parity game exactly one of the two players has a \emph{memoryless} winning strategy~\cite{graedel_2001}, the existence of a winning strategy for Prover implies the existence of a memoryless winning strategy for Prover.
	 \end{proof}

Next we show the converse direction of Proposition \ref{p: provable implies winning strategy}: from a memoryless winning strategy of Prover in $\mathscr{G}(\hickp, \sigma)$ we can construct a $\Sigma$-proof of $\sigma$ in $\hickp$.

	\begin{proposition}\label{proposition winning strategy implies provable}
		Let $\hickp \in \{\cick, \cickt, \cicks, \cickss\}$ and $\sigma$ be a $\Sigma$-sequent. If Prover has a memoryless winning strategy in $\mathscr{G}(\hickp,\sigma)$, then $\sigma$ is $\Sigma$-provable in $\hickp$. 
	\end{proposition}
	\begin{proof}
		Suppose that Prover has a memoryless winning strategy in $\mathscr{G}(\hickp,\sigma)$. Let $\mathcal{T}'$ be the corresponding strategy tree. Define the `condensed' labeled tree $\mathcal{T}$ of $\mathcal{T'}$ as follows.
        \begin{itemize}
            \item A node $t$ of $\mathcal{T}'$ is a node of $\mathcal{T}$ if and only if $t$ is labeled by $p \in S$ in $\mathcal{T}'$.
            \item A node $s$ in $\mathcal{T}$ is a child of a node $t$ in $\mathcal{T}$ if and only if there exists a node $s'$ in $\mathcal{T}'$ such that $s'$ is a child of $t$ and $s$ is a child of $s'$ in $\mathcal{T}'$.
            \item A node $t$ in $\mathcal{T}$ is labeled by $p \in S$ if and only if $t$ is labeled by $p$ in $\mathcal{T}'$. 
        \end{itemize}
        Note that $\mathcal{T}$ is a finite branching (and possibily non-wellfounded) tree labeled by $\Sigma$-sequents according to the rules of $\hickp$. Let $\pi$ be the finite subtree of $\mathcal{T}$ obtained by pruning every infinite branch of $\mathcal{T}$ at the first repetition (note that repetitions exist since there are only finitely many different $\Sigma$-sequents). Observe that the root of $\pi$ is labelled by $\sigma$ and $\pi$ is generated by rules of $\hickp$. Furthermore, since $\mathcal{T}$ is finite branching, $\pi$ is finite by K\H{o}nig's Lemma. Therefore $\pi$ is a cyclic $\hickp$-pre-proof. In order to show that $\pi$ is a proof, let $u$ be an arbitrary leaf of $\pi$. If $u$ is also a leaf of $\mathcal{T}$, then $u$ is a node in $\mathcal{T}'$ which has a unique child $u'$ labeled by a rule position of the form $\langle \sigma', \mathsf{id}, \langle \rangle \rangle$ or of the form $\langle \sigma', \bot, \langle \rangle \rangle$, implying that $\sigma'$ is an instance of $\mathsf{id}$ or of $\bot$. Hence, $u$ is an axiomatic leaf in $\pi$. Otherwise, $u$ was generated by pruning an infinite branch of $\mathcal{T}$. In that case there exists a node $c(u)$, such that $\langle c(u), u \rangle$ is a repetition. It remains to show $\langle c(u), u \rangle$ is successful. Suppose towards a contradiction that $\langle c(u), u \rangle$ is not successful. Then on the path $\rho$ from $c(u)$ to $u$ either some sequent does not have a formula in focus or $\rho$ does not pass through an application of $\mathsf{\C R}$ where the principal formula is in focus. Let $\rho'$ be the finite sequence of positions in $S \cup I$ obtained from $\rho$ by adding after each position $\rho(i) \in S$ of $\rho$ the rule position which has $\rho(i)$ as conclusion and $\rho(i+1)$ as premise. Then either there is a position occurring in $\rho'$ which has priority $3$, or every position in $\rho'$ has priority $1$. Since $\mathcal{T}'$ is the strategy tree of a \emph{memoryless} strategy, there exists an infinite branch in $\mathcal{T}'$ that has a suffix which is an infinite concatenation $\rho' \cdot \rho' \cdot \rho' \cdots$. On this path the highest priority encountered infinitely often is odd, contradicting the assumption that $\mathcal{T}'$ is the strategy tree of a winning strategy for Prover. Therefore each repetition is successful and so we conclude that $\pi$ is a $\Sigma$-proof of $\sigma$ in $\hickp$.
	\end{proof}
	
	Let $R_I$ be the set of all rule instances of $\hickp$ involving only $\Sigma$-sequents. Observe that the above constructed proof is \emph{uniform} in the following sense:
	
	\begin{definition}
		A $\Sigma$-proof $\pi$ is \emph{uniform} if there exists a function $g: S \longrightarrow R_I$ such that whenever a sequent $\sigma \in S$ occurs in $\pi$, it occurs as the conclusion of the rule instance $g(\sigma)$.
	\end{definition}
   Note that in a uniform proof the first repetition in each branch is successful. We conclude:
	\begin{theorem}\label{t: strategy iff provable}
		The following are equivalent for any sequent $\sigma$ and any cyclic calculus $\hickp \in \{\cick, \cickt, \cicks, \cickss\}$:
		\begin{enumerate}
			\item $\sigma$ is $\hickp$-provable.
            \item $\sigma$ has a $\Sigma$-proof in $\hickp$.
			\item Prover has a memoryless winning strategy in $\mathscr{G}(\hickp,\sigma)$.
			\item $\sigma$ has a uniform $\Sigma$-proof in $\hickp$.
		\end{enumerate}
	\end{theorem}

\subsection{Exponential upper bound for proof-search}\label{s: complexity}

We will now show that the problem of finding cyclic proofs in our calculi can be solved by a program in exponential time, thus, by soundness and completeness, we also provide an exponential upper bound for solving the validity problem (see Table \ref{tab:validity problem}) of our logics. The problem of finding cyclic proofs is called the \emph{proof-search problem}, see Table \ref{tab:proof-search problem}.

\begin{table}[t]
    \centering
    \begin{tabular}{|l l|}
    \hline
        \textbf{Input:} &  A sequent $\sigma$ and a proof system $\hickp \in \{\cick, \cickt, \cicks, \cickss\}$. \\
         \textbf{Question:} &  Does $\sigma$ have a $\hickp$-proof? \\
         \hline
    \end{tabular}
    \caption{The proof-search problem}
    \label{tab:proof-search problem}
\end{table}
Recall the notion of \emph{complexity} of a formula and of a set of formulae (c.f. Definition \ref{d: complexity}). Given a sequent $\sigma$, define the \emph{complexity} of $\sigma$ to be $c(\sigma) \coloneqq c(\Gamma_\sigma) + c(\Delta_\sigma)$.

Note that if $\Sigma$ is the (negation) closure of $\Gamma^- \cup \Delta^-$, then $\lvert \Sigma \rvert$ is linear in $c(\Gamma \Rightarrow \Delta)$. Moreover it is easy to check that for a $\Sigma$-sequent $\sigma$ the number of sequent positions in a game $\mathscr{G}(\hickp, \sigma)$ is polynomial in the size of the powerset of $\Sigma$, since the left-hand side of each sequent is a subset of $\Sigma$ with each formula unfocused and the right hand side is a subset of $\Sigma$ with at most one formula in focus. Furthermore, for each rule $\mathsf{r} \in \hickp$ the number of rule instances with a given sequent as conclusion is bounded by the number of possible principal formulae times a constant, since rules may be applied preservingly or for modality rules different formulae may be weakened. Thus overall we obtain the following result. 
    
	\begin{lemma}\label{l: number of positions in games}
		Let $\hickp \in \{\cick, \cickt, \cicks, \cickss\}$. Given a $\Sigma$-sequent $\sigma$, the number of positions in $\mathscr{G}(\hickp,\sigma)$ is polynomial in $\lvert \mathcal{P}(\Sigma) \rvert$.
	\end{lemma}
	
    In order to get a polynomial bound for deciding the winner of a given proof-search game we can now make use of one of the many existing algorithms for solving parity games. For instance the following result by Calude et al.~\cite{calude_2017}. 
	\begin{theorem}[{\cite[Theorem 2.9]{calude_2017}}]\label{Theorem 2.9 Calude}
		There is an algorithm which finds the winner of a parity game in time $\mathcal{O}(n^{log(m)+6})$ for a parity game with $n$ positions and priorities in $\{1,2,...,m\}$. Furthermore, the algorithm can compute a memoryless winning strategy for the winner in time $\mathcal{O}(n^{log(m)+7} \cdot log(n))$.
	\end{theorem}
	
	Let $\hickp \in \{\cick, \cickt, \cicks, \cickss\}$, $\sigma$ be a sequent and let $\Sigma$ be its negation closure. By Lemma \ref{l: number of positions in games} the number $n$ of positions in $\mathscr{G}(\hickp,\sigma)$ is polynomial in the size of $\lvert \mathcal{P}(\Sigma) \lvert$. Since the number of different priorities in our games is constant, Theorem~\ref{Theorem 2.9 Calude} implies that there is an algorithm deciding the winner of $\mathscr{G}(\hickp,\sigma)$ in time polynomial in $\lvert \mathcal{P}(\Sigma) \lvert$ and so exponential in $c(\sigma)$. By the same argument the above mentioned algorithm also computes a memoryless winning strategy in time exponential in  $c(\sigma)$. From this winning strategy we can then read of a proof of $\sigma$ by using Proposition~\ref{proposition winning strategy implies provable}.

   \begin{theorem}
       Given $\hickp \in \{\cick, \cickt, \cicks, \cickss\}$ and a sequent $\sigma$, there is an algorithm deciding whether $\sigma$ is $\hickp$-provable that runs in exponential time in $c(\sigma)$. The proof-search problem is therefore solvable in exponential time.
   \end{theorem}
As a corollary of this theorem as well as of soundness and completeness (Theorem~\ref{t: CIM soundness}, Theorem~\ref{t: cyclic completeness} and Theorem~\ref{t: completeness of S5}) we obtain the last main results of this paper: the validity problem of each considered logic is in \textsc{ExpTime} and in the case of $\mathbf{ICKS5}$ is \textsc{ExpTime}-complete.

\begin{theorem}\label{t: validity problem is Exptime}
    The validity problem (c.f. Table \ref{tab:validity problem}) is solvable in exponential time.
\end{theorem}

\begin{theorem}
    The validity problem for $\mathbf{ICKS5}$ is \textsc{ExpTime}-complete.
\end{theorem}
\begin{proof}
    That the validity problem belongs to \textsc{ExpTime} is Theorem~\ref{t: validity problem is Exptime}. For \textsc{ExpTime}-hardness consider the translation $tr$ from Section~\ref{s: translation}. Note that given a formula $\varphi$ we can compute $tr(\varphi)$ in polynomial time in $c(\varphi)$. Theorem~\ref{t: negative translation} therefore implies that $tr$ is a polynomial-time reduction from the validity problem for $\mathbf{CKS5}$ onto the validity problem for $\mathbf{ICKS5}$. Since the former problem is \textsc{ExpTime}-complete~\cite{halpern_1992}, we obtain that the latter problem is \textsc{ExpTime}-hard.
\end{proof}

In a recent publication, Santiago-Fern{\'a}ndez, Fern{\'a}ndez-Duque and Joosten showed that the validity problem for $\mathbf{ICK}$ is also \textsc{ExpTime}-complete~\cite{Fernandez-Duque_complexity_2025}. We conjecture that the same result holds for $\mathbf{ICKT}$ and $\mathbf{ICKS4}$.

\section{Conclusion}\label{s: conclusion}

We have presented an in-depth study of intuitionstic common knowledge logic over various classes of frames. We have presented sound and complete axiomatizations for each logic as well as analytic cyclic sequent calculi. Completeness of the cyclic calculi was proven via a detour to non-wellfounded proofs and a robust proof-search construction, which was developed in a previous publication with Afshari, Grotenhuis and Leigh~\cite{afshari_intuitionistic_2024}. We have also shown that proof-search in the cyclic calculi can be automated by reducing the problem of finding a proof to the problem of solving a parity game, for which efficient algorithms exist. Using such an algorithm established an exponential upper bound on the proof-search and validity problems for all considered logics. Finally, using a translation of $\mathbf{CKS5}$ into $\mathbf{ICKS5}$, we showed that the validity problem of $\mathbf{ICKS5}$ is in fact \textsc{ExpTime}-complete. Our work leaves several questions open for future research. In particular, with analytic cyclic calculi for $\ICK$ at hand, it is natural to study the proof-theoretic properties of these systems, such as whether we can extract interpolants in a way similar as in~\cite{Afshari_lyndon_2022, Shamkanov_circular_2014}. Moreover, we plan to study extensions of $\ICK$ with co-implication from bi-intuitionistic logic~\cite{Rauszer_1974a, Rauszer_1974b}. As co-implication ($\dimp$) is order-dual to the intuitionistic implication (in the sense that $w \models \varphi \dimp \psi$ if there exists $v \leq w$ with $v \models \varphi$ and $v \not \models \psi$), such a logic is capable of expressing facts about worlds `below' and could thus be used in our interpretation of $\ICK$ to reason about knowledge in a setting where information may be gained or lost. We believe that such a system would be of interest to the knowledge representation and reasoning (KRR) community (see e.g.~\cite{fernandez-duque_family_2023}).

\bibliographystyle{plain}
\bibliography{references}

\end{document}